\documentclass{amsart}
\usepackage{amssymb}
\usepackage{amsmath}

\newcommand{\map}[1]{\xrightarrow{#1}}
\newcommand{\iso}{\cong}

\newcommand{\Gal}{\mathrm{Gal}}
\newcommand{\Spec}{\mathrm{Spec}}
\newcommand{\Q}{\mathbb Q}
\newcommand{\Z}{\mathbb Z}
\newcommand{\R}{\mathbb R}
\newcommand{\C}{\mathbb C}
\newcommand{\alg}{\mathrm{alg}}
\newcommand{\A}{\mathbb A}
\newcommand{\co}{\mathcal O}
\newcommand{\ord}{\mathrm {ord}}
\newcommand{\GL}{\mathrm{GL}}
\newcommand{\Pic}{\mathrm{Pic}}
\newcommand{\Alb}{\mathrm{Alb}}

\begin{document}

\author{Benjamin Howard}
\email{howardbe@bc.edu}
\address{Department of Mathematics, Boston College, Chestnut Hill, MA, 02467}
\title{Twisted Gross-Zagier theorems}
\thanks{This research was supported in part by NSF grant DMS-0556174}

\begin{abstract}
The theorems of Gross-Zagier and Zhang relate the N\'eron-Tate heights of complex multiplication points  on the modular curve $X_0(N)$ (and on Shimura curve analogues) with the central derivatives of automorphic $L$-function.  We extend these results to include certain CM points on modular curves of the form $X(\Gamma_0(M)\cap\Gamma_1(S))$ (and on Shimura curve analogues).  These results are motivated by applications to Hida theory which are described in the companion article \cite{howard-derivatives}.
\end{abstract}

\maketitle

\theoremstyle{plain}
\newtheorem{Thm}{Theorem}[subsection]
\newtheorem{Prop}[Thm]{Proposition}
\newtheorem{Lem}[Thm]{Lemma}
\newtheorem{Cor}[Thm]{Corollary}
\newtheorem{Con}[Thm]{Conjecture}
\newtheorem{BigThm}{Theorem}

\theoremstyle{definition}
\newtheorem{Def}[Thm]{Definition}
\newtheorem{Hyp}[Thm]{Hypothesis}

\theoremstyle{remark}
\newtheorem{Rem}[Thm]{Remark}
\newtheorem{Ques}[Thm]{Question}

\renewcommand{\labelenumi}{(\alph{enumi})}
\renewcommand{\theBigThm}{\Alph{BigThm}}

\numberwithin{equation}{section}

\section{Introduction}

Let $\chi_0$ be a finite order character of the idele class group $\Q^\times\backslash \A^\times$ of $\Q$, and suppose that $f\in S_2(\Gamma_0(N),\chi_0^{-1},\C)$ is a normalized newform of level $N$ and character $\chi_0^{-1}$.  In particular we assume that $f$ is an eigenform for all Hecke operators $T_n$ with $(n,N)=1$.  Writing $f=\sum_n b_n q^n$ the $L$-series of $f$ is defined as the analytic continuation of $L(s,f)=\sum_n b_n n^{-s}$.    To compare with the notation used in the body of the article, $L(s,\Pi)=L^*(s+1/2,f)$ where 
$$
L^*(s,f)=2 (2\pi)^{-s} \Gamma(s) L(s,f)
$$ 
is the completed $L$-function of $f$ and $\Pi$ is the automorphic representation of $\GL_2(\A)$ attached to $f$.
Let $E$ be a quadratic imaginary field of discriminant $-D$ and let $\chi$ be a finite order character of the idele class group $E^\times\backslash \A_E^\times $  whose restriction to $\A^\times$ agrees with $\chi_0$.  Factor  $N=MS$ in such a way that  $S$ is divisible only by primes dividing $\mathrm{N}_{E/\Q}(\mathrm{cond}(\chi))$ and $M$ is relatively prime to $\mathrm{N}_{E/\Q}(\mathrm{cond}(\chi))$.
We assume
\begin{enumerate}
\item $N$ and  $\mathrm{N}_{E/\Q}(\mathrm{cond}(\chi))$ are each  relatively prime to $D$,
\item for any prime  $p\mid S$ the restriction of $\chi$ to $E_p^\times=(E\otimes_\Q\Q_p)^\times$ factors through the norm $E_p^\times\map{}\Q_p^\times$,
\item $S=\mathrm{cond}(\chi_0)$.
\end{enumerate}
It is easy to see from these hypotheses that  $\mathrm{cond}(\chi)=C\co_E$ for some positive integer $C$ which is divisible by $S$.

Let $\omega$ denote the quadratic Dirichlet character attached to $E$.  The $L$-function of $f$  and the Hecke $L$-series of $\chi$ each admit Euler products over the rational primes.  For each prime $p$ the  local Eulers factors  have the form   
\begin{eqnarray*}
L_p(s,f)&=&(1-\alpha_{1}p^{-s})^{-1}(1-\alpha_{2} p^{-s})^{-1} \\
L_p(s,\chi)&=& (1-\beta_{1}p^{-s})^{-1}(1-\beta_{2} p^{-s})^{-1}
\end{eqnarray*}
and we define a new Euler factor
$$
L_p(s,\chi,f) = \prod_{1\le i,j\le 2} (1-\alpha_{i}\beta_{j} p^{-s})^{-1}.
$$
The Rankin-Selberg convolution $L$-function $L(s,\chi,f)=\prod_p L_p(s,\chi,f)$  has analytic continuation to an entire function of $s$ and satisfies the functional equation 
$$
L^*(s,\chi,f)= -\omega(M) \cdot (C^2DM)^{2-2s} \cdot L^*(2-s,\chi,f)
$$
where
$$
L^*(s,\chi,f) = 4 (2\pi)^{-2s} \Gamma(s)^2 L(s,\chi,f).
$$
In the notation of the body of the text $L(s,\Pi\times\Pi_\chi)=L^*(s+1/2,\chi,f)$, and so the functional equation follows from the functional equation (\ref{kernel functional}) of the Rankin-Selberg kernel and the integral representation of the $L$-function (\ref{rankin integral}).

Assume that every prime divisor of $M$ splits in $E$.  In particular the functional equation forces $L(1,\chi,f)=0$.  Let $\co=\Z+C\co_E$ and $\co'=\Z+CS^{-1}\co_E$ be the orders of conductors $C$ and $CS^{-1}$, respectively, of $\co_E$.   Fix an invertible ideal $\mathfrak{M}\subset\co$ such that $\co/\mathfrak{M}\iso \Z/M\Z$ and consider the isogenies of complex elliptic curves
$$
\C/\co \map{ F_M }  \C/\mathfrak{M}^{-1}  
\hspace{1cm}
\C/\co \map{F_S} \C/\co'.
$$
These isogenies are cyclic of degree $M$ and $S$, respectively, and  if we pick an arbitrary generator $\pi\in\mathrm{ker}(F_S)$  the triple $Q=(\C/\co,\mathrm{ker}(F_M), \pi)$ determines a point on the moduli space $X_\Gamma(\C)$ parametrizing complex elliptic curves with $$\Gamma=\Gamma_0(M)\cap \Gamma_1(S)$$ level structure.  We view $X_\Gamma$ as a scheme over $\Spec(\Q)$.  Let $\widehat\co$ denote the closure of $\co$ in the ring $\A_{E,f}$ of finite adeles of $E$ and let $\theta:\widehat{\co}^\times\map{} (\Z/S\Z)^\times$ denote homomorphism giving the action of $\widehat{\co}^\times$ on $\widehat{\co}'/\widehat{\co}\iso \Z/S\Z$.  The character $\chi$ has trivial restriction to $\mathrm{ker}(\theta)$, and by the theory of complex multiplication the point $Q$ is rational over the abelian extension of $E$  with class group $E^\times\backslash \A_{E,f}^\times/  \mathrm{ker}(\theta)$.  Thus we may form the divisor with complex coefficients
$$
Q_\chi = \sum_{t \in E^\times\backslash \A_{E,f}^\times/  \mathrm{ker}(\theta) } \overline{\chi(t)} \cdot Q^{[t,E] }
$$
on $X_{\Gamma} \times_\Q E_\chi$  where $[\cdot, E]$ is the Artin symbol normalized  as in \cite[\S 5.2]{shimura} and $E_\chi$ is the abelian extension of $E$ cut out by $\chi$.  Assume that $\chi$ is nontrivial (otherwise $S=1$ and we are in the case originally considered by Gross and Zagier \cite{gross-zagier}) so that $Q_\chi$ has degree zero and may be viewed as a point in the modular Jacobian $Q_\chi\in J_\Gamma(E_\chi )\otimes_\Z\C$.   Denote by $\mathbb{T}$  the (semi-simple)  $\C$-algebra generated by the Hecke operators $\{ T_n\mid (n,N)=1\}$ and the diamond operators $\{\langle d\rangle\mid (d,S)=1\}$ acting on $S_2(\Gamma,\C)$.  By the Eichler-Shimura theory the algebra $\mathbb{T}$ acts  on $J_\Gamma(E_\chi )\otimes_\Z\C$ via the Albanese endomorphisms $T_{n*}$ and $\langle d\rangle_*$ as in \cite[\S 2.4]{MW}).  

The following theorem is a special case of Theorem \ref{thm:twisted gz}.  When $S=1$ this result is due to Zhang \cite[Theorem 6.1]{zhang3}.  When $S=1$ and $\chi$ is unramified it is due to Gross-Zagier \cite{gross-zagier}.

\begin{BigThm}\label{the result}
Let $Q_{\chi,f}$ denote the projection of $Q_\chi$ to the maximal summand of $ J_\Gamma(E_\chi )\otimes_\Z\C$ on which $\mathbb{T}$ acts through $T_n\mapsto b_n$ and $\langle d\rangle\mapsto \chi_0^{-1}(d)$.  Then
$$
L'(1,\chi,f)=0\iff Q_{\chi,f}=0.
$$
\end{BigThm}

\begin{Rem}
The hypotheses (b) and (c) placed on the primes divisors of $S$ are not made for the sake of convenience; rather these hypotheses seem to be closely related to the particular choice of $\Gamma_1(S)$ level structure on $\C/\co$, given by a generator of the kernel of an isogeny to an  elliptic curve with complex multiplication by a \emph{different} quadratic order.
\end{Rem}

\begin{Rem}
If $\Pi\iso \bigotimes_v\Pi_v$ denotes the automorphic representation of $\GL_2(\A)$ generated by the adelization of $f$ then  the condition (c) above is equivalent to  Hypothesis \ref{hyp}(b) below, with $F=\Q$, $\mathfrak{s}=S\Z$, and $\mathfrak{c}=C\Z$.  This follows from the formulas of \cite[\S 12.3]{nek} and \cite[Theorem 4.6.17]{miyake}.
\end{Rem}

 Throughout the body of the article we work in much greater generality than the situation described above;  instead of a classical modular form $f$ as above we work with a Hilbert modular newform $\phi_\Pi$ over a totally real field $F$, and assume that $\phi_\Pi$ is either holomorphic of parallel weight $2$ or is a Maass form of parallel weight $0$.   Let $\chi$ be a finite order character of the idele class group of a totally imaginary quadratic extension $E$ of $F$, and assume that the restriction of $\chi^{-1}$ to the ideles of $F$ agrees with the central character of the automorphic representation $\Pi$ generated by $\phi_\Pi$.  We assume that $\Pi$, $\chi$, and $E$ also satisfy the hypotheses of \S \ref{notations} below.  The Rankin-Selberg $L$-function $L(s,\Pi\times\Pi_\chi)$, where $\Pi_\chi$ is the theta series representation associated to $\chi$, is normalized so that the center of symmetry of the functional equation is at $s=1/2$.  

Assume first that $\phi_\Pi$ is holomorphic of parallel weight $2$.  When the sign in the functional equation of $L(s,\Pi\times\Pi_\chi)$ is $1$ we prove a formula (Theorem \ref{second holomorphic values}) relating the central value $L(1/2, \Pi\times\Pi_\chi)$ to certain CM-points on a totally definite quaternion algebra over $F$.  In special cases such results go back to Gross's special value formula \cite{gross-values}.   Such special value formulas have been used by Bertolini and Darmon to construct anticyclotomic $p$-adic $L$-functions for elliptic curves \cite{BD-L-functions}, and such $L$-functions play a central role both in those authors' work on the anticyclotomic Iwasawa main conjecture for elliptic curves \cite{BD-iwasawa} as well as  in the work of Vatsal \cite{vatsal-values} and Cornut-Vatsal \cite{cornut-vatsal, cornut-vatsal2} on the nonvanishing of $L$-values in towers of ring class fields.  We point out also the helpful expository article of Vatsal \cite{vatsal}.  When the sign in the functional equation of $L(s,\Pi\times\Pi_\chi)$ is $-1$ we prove a theorem (Theorem \ref{thm:twisted gz}, which includes Theorem \ref{the result} as a special case) which  generalizes results  of Zhang \cite[Theorem 6.1]{zhang3} and Gross-Zagier \cite{gross-zagier} by relating the central derivative $L'(1/2,\Pi\times\Pi_\chi)$ to the N\'eron-Tate height of CM-cycles on a Shimura curve over $F$.  Now assume that  $\phi_\Pi$ is  Maass form of parallel weight $0$  and that the sign in the functional equation of $L(s,\Pi\times\Pi_\chi)$ is $1$.  In this case we prove (Theorem \ref{second maass values}) a formula expressing the central value $L(1/2,\Pi\times\Pi_\chi)$ as a weighted sum of the values  at CM points of a weight $0$ Maass form (related to $\phi_\Pi$ by the Jacquet-Langlands correspondence) on a Shimura variety of dimension $[F:\Q]$.   

Our methods follow those of Zhang \cite{zhang2, zhang3} and we freely use his results and calculations when they carry over to our setting without significant change; the reader is advised to keep copies of \cite{zhang2, zhang3} close at hand.   The original contributions are primarily found in  \S \ref{quaternion generalities} and \S \ref{s:central value}.  

The primary motivation for this work is to obtain results  on the behavior of Selmer groups and  $L$-functions in Hida families.  Indeed, the somewhat  peculiar point  $Q\in X_\Gamma(\C)$ defined above plays a central role in the construction of \emph{big Heegner points} \cite{howard-hida} in the cohomology of Galois representations for $\Lambda$-adic modular forms.  Theorem \ref{the result} can be used to verify, in any particular case, the conjectural nonvanishing of these big Heegner points and can also be used to give examples of Hida families of modular forms whose $L$-functions vanish to exact order one with only finitely many exceptions. The applications to Hida theory and Iwasawa theory of the results contained herein is found in the separate article \cite{howard-derivatives}.

\bigskip

The author thanks Shou-Wu Zhang for many very helpful conversations.


\subsection{Notation and conventions}
\label{notations}


The following choices and conventions apply throughout the remainder of the article.

Fix a totally real field $F$, a CM-extension $E/F$ of relative discriminant $\mathfrak{d}$ and relative different $\mathfrak{D}$, and denote by $\A$ and $\A_E$ the adele rings of $F$ and $E$, respectively.  The integer rings of $F$ and $E$ are denoted $\co_F$ and $\co_E$, respectively, and $\omega$ denotes the quadratic  character of $\A^\times/F^\times$ corresponding to the extension $E/F$.    If $M$ is any finitely generated $\Z$-module we let $\widehat{M}$ denote its profinite completion.   If $\mathfrak{a}$ is any nonzero $\co_F$-ideal, $\mathrm{N}_{F/\Q}(\mathfrak{a})$ denotes the cardinality of $\co_F/\mathfrak{a}$.
If $v$ is a real place of $F$ then $|\cdot|_v$ denotes the usual absolute value on $F_v\iso \R$.  If $v$ is a finite place then $|\cdot|_v$ is normalized so that for any uniformizing parameter $\varpi$ of $F_v$, $|\varpi|_v^{-1}$ is the size of the residue field of $v$. 
For any $\co_F$-module $M$ and any place $v$ of $F$, set $M_v=M\otimes_{\co_F} \co_{F,v}$.  For any $x\in\A^\times$ let $x\co_F$ denote the fractional ideal of $\co_F$ determined by $(x\co_F)_v=x_v\co_{F,v}$ for every finite place $v$.

Fix a finite order character $\chi:\A_E^\times/E^\times\map{}\C^\times$.  Let $\chi_0$ denote the restriction of $\chi$ to $\A^\times/F^\times$ and let $\mathfrak{C}$ denote the conductor of $\chi$.   We abbreviate $\mathrm{N}(\mathfrak{C})=\mathrm{N}_{E/F}(\mathfrak{C})$.   For each place $v$ of $F$ let $\chi_v$ denote the restriction of $\chi$ to $E_v^\times=(E\otimes_F F_v)^\times$.
Let $\Pi$ be an irreducible infinite dimensional cuspidal automorphic representation of $\GL_2(\A)$ of central character $\chi_0^{-1}$ and conductor $\mathfrak{n}$, as defined in \S \ref{ss:automorphic forms}.  Factor $\mathfrak{n}=\mathfrak{ms}$ in such a way  that  $\mathfrak{m}$ is prime to $\mathrm{N}(\mathfrak{C})$ and $\mathfrak{s}$ is divisible only by primes dividing $\mathrm{N}(\mathfrak{C})$.  We assume throughout that $\mathfrak{n}$ and $\mathrm{N}(\mathfrak{C})$ are both prime to $\mathfrak{d}$.

\begin{Hyp}\label{hyp}
At times we will  assume that $\Pi$ satisfies the following hypotheses.
\begin{enumerate}
\item
For every $v\mid\mathfrak{s}$ there is a character $\nu_v$ of $F_v^\times$ such that $\chi_v=\nu_v\circ\mathrm{N}_{E_v/F_v}$. Note that this hypothesis implies that $\mathfrak{C}=\mathfrak{c}\co_E$ for some ideal $\mathfrak{c}$ of $\co_F$.
\item
For every  $v\mid \mathfrak{s}$, $\Pi_v$ is a principal series representation $\Pi(\mu_v,\chi_{0,v}^{-1}\mu_v^{-1})$ of $\GL_2(F_v)$  with $\mu_v$ an unramified quasi-character of $F_v^\times$.  In particular 
$$
\ord_v(\mathfrak{s})=\ord_v(\mathrm{cond}(\chi_0)) \le \ord_v(\mathfrak{c}).
$$
\end{enumerate}
These hypotheses will be assumed in \S \ref{s:central value} and \S \ref{derivatives} but are not needed for the calculations of \S \ref{quaternion generalities}, or for the calculations of \S \ref{analytic section} unless otherwise indicated.
\end{Hyp}


\section{Automorphic forms and the Rankin-Selberg integral}
\label{analytic section}


Let $\psi:  \A/F\map{}\C^\times$  be a nontrivial additive character.  
Fix an idele $\delta\in\A^\times$ in such a way that for every finite place $v$ of $F$ the restriction to $F_v$ of  the additive character  $\psi^0:\A\map{}\C^\times$ defined by $\psi^0(x)=\psi(\delta^{-1} x)$
has conductor $\co_{F,v}$ and so that for  every archimedean place $v$ the restriction of $\psi^0$ to $F_v\iso \R$ is given by $\psi_v^0(x)=e^{2\pi i x}$.  This implies that $F$ has absolute discriminant $D_F=|\delta|^{-1}$.  For any finite place $v$ of $F$ we normalize the additive Haar measure $dx$ on $F_v$ in such a way that the volume of $\co_{F,v}$ is equal to $|\delta|_v^{1/2}$, and  normalize the multiplicative Haar measure $d^\times x$ on $F_v^\times$ in such a way that the volume of $\co_{F,v}^\times$ is $1$.  Then $dx$ and $d^\times x$ are related by 
\begin{equation}
\label{additive-multiplicative measures}
 |\delta|_v^{1/2}(1-|\varpi|_v)\cdot d^\times x =  |x|_v^{-1}\cdot dx
\end{equation}
for any uniformizer $\varpi$ of $F_v$.  On $\R^\times$ we normalize the Haar measure $d^\times x$ by  $d^\times x=|x|^{-1}d^{\mathrm{Leb}}x$, where $d^{\mathrm{Leb}}x$ is the usual Lebesgue measure giving $[0,1]$ unit mass.  For an archimedean place $v$ the additive Haar measure $dx$ on $F_v\iso \R$ is normalized by $dx=|\delta|_v^{1/2} d^{\mathrm{Leb}}x$.  In all cases the Haar measure on the additive group $F_v$ is self-dual with respect to $\psi_v$.  Endow $\A$ and $\A^\times$ with the product measures;  the quotient measure on $\A/F$ has total volume $1$ by \cite[Proposition V.4.7]{weil}.

 Fix $d\in\A^\times$ such that $d\co_F=\mathfrak{d}$ and $d_v=1$ for $v\mid\infty$.  Let $S$ denote the set of places of $F$ dividing $\mathfrak{d}$, and for each $v\in S$ set 
$h_v  =  \left(\begin{matrix} & 1 \\ -d_v \end{matrix}\right) \in\GL_2(F_v),$
viewed as an element of $\GL_2(\A)$ with trivial components away from $v$.  For each subset $T\subset S$ set $h_T=\prod_{v\in T}h_v$ and view $h_{T}$ as an operator on automorphic forms on $\GL_2(\A)$ via $(h_T\phi) (g)=\phi(gh_T)$.  For $a\in\A^\times$  define $e_\infty(a)=\prod_{v\mid\infty}e_v(a)$ where 
$$
e_v(a) = \left\{ \begin{array}{ll}
 2e^{-2\pi a_v} &\mathrm{if\ } a_v>0 \\
0 &\mathrm{otherwise.} \end{array}\right.  
$$
for each $v\mid\infty$.  Define the usual gamma factors
$$
G_1(s)=\pi^{-s/2}\Gamma(s/2)
\hspace{1cm}
G_2(s)= 2(2\pi)^{-s}\Gamma(s).
$$


\subsection{Automorphic forms}
\label{ss:automorphic forms}


Let $\phi$ be an automorphic form on $\GL_2(\A)$. Then $\phi$ admits a Fourier expansion
$$
\phi (g)= C_\phi(g) + \sum_{\alpha \in F^\times} W_\phi
\left(  \left( \begin{matrix}
\alpha & \\ & 1
\end{matrix}  \right) g \right)
$$
in which the constant term $C_\phi$ and Whittaker function $W_\phi$ (with respect to $\psi$) are defined by \cite[(2.4.3)]{zhang2} and \cite[(2.4.4)]{zhang2}, respectively. For every $a\in\A^\times$  the \emph{Whittaker coefficient}
$$
B(a;\phi)  =   
W_\phi \left( \begin{matrix}
a \delta^{-1} & \\ & 1
\end{matrix}  \right)
$$
is independent of the choice of $\psi$, and a simple calculation shows that the Whittaker coefficients of $\phi$ and $\overline{\phi}$ are related by  $B(a;\overline{\phi})= \overline{B(-a;\phi)}.$  The \emph{zeta function} of $\phi$ is defined as the meromorphic continuation of 
\begin{eqnarray*} 
Z(s;\phi)   &=&  |\delta|^{1/2-s} \int_{\A^\times} 
B(y;\phi)\cdot |y|^{s-1/2}\ d^\times y   \\
&=&  
 \int_{\A^\times/F^\times} 
( \phi-C_\phi)
\left(   \begin{matrix}
y & \\ & 1
\end{matrix}  \right) \cdot 
|y|^{s-1/2}\ d^\times y
\end{eqnarray*}
in which both integrals are convergent for $\mathrm{Re}(s)\gg 0$.  As in \cite[\S 3.5]{zhang2} we say that an automorphic form $\phi$ of parallel weight $2$ is \emph{holomorphic} if its Whittaker coefficient has the form
$$
B(a;\phi)=|a|_\infty e_\infty(a)\cdot \widehat{B}(\mathfrak{a};\phi)
$$
with $\mathfrak{a}=a\co_F$ for some function $\widehat{B}(\mathfrak{a};\phi)$ on fractional ideals of $\co_F$ which vanishes on non-integral ideals.

Let $v$ be a finite place of $F$.  If $\mathfrak{n}_v$ is an ideal of $\co_{F,v}$ define the habitual congruence subgroup
$$
K_1(\mathfrak{n}_v)=\left\{ \left(\begin{matrix} a&b\\c&d\end{matrix}\right) \in\GL_2(\co_{F,v})\  \Big| \ c\in \mathfrak{n}_v, d\in 1+\mathfrak{n}_v  \right\}.
$$
For an irreducible, admissible, infinite dimensional representation $\pi_v$ of $\GL_2(F_v)$ the \emph{conductor} of $\pi_v$ is the largest ideal $\mathfrak{n}_v$ such that $\pi_v$ admits a $K_1(\mathfrak{n}_v)$-fixed vector.  The space $K_1(\mathfrak{n}_v)$-fixed vectors is then $1$-dimensional, and any nonzero vector on this line will be called a \emph{newvector}.  If $v$ is an infinite place of $F$ then any $\pi_v$ as above has a unique line of vectors of minimal non-negative weight for the action of $\mathrm{SO}_2(\R)$; a nonzero vector on this line is again called a \emph{newvector}.  If $\pi\iso \bigotimes_v\pi_v$ is an irreducible automorphic representation of $\GL_2(\A)$ then a newvector in $\pi$ is a product of local newvectors.  Such a newvector is unique up to scaling, and we define  the \emph{normalized newvector} $\phi_\pi\in \pi$ to be the unique newvector  satisfying 
$$
Z(s,\phi_\pi) =|\delta|^{1/2-s}  L(s,\pi).
$$
If $\mathfrak{n}$ is an ideal of $\co_F$ set $K_1(\mathfrak{n})=\prod_v K_1(\mathfrak{n}_v)$, where the product is over all finite places.  

Suppose $v$ is a finite place of $F$,  $\phi$ is an automorphic form which is fixed by the action of $K_1(\mathfrak{n})$, and $(\mathfrak{a},\mathfrak{n})=1$.  We define 
$$
(T_{\mathfrak{a}}\phi)(g)=\sum_{h\in H(\mathfrak{a}) / K_1(\mathfrak{n}) } \phi(gh)
$$
where $H(\mathfrak{a}_v)$ is the set of elements of $M_2(\co_{F,v})$ whose determinant generates $\mathfrak{a}_v$ and 
$$
H(\mathfrak{a})=\prod_{v\nmid \mathfrak{a}} K_1(\mathfrak{n}_v)\cdot \prod_{v\mid\mathfrak{a}}H(\mathfrak{a}_v).
$$
If $a\in\A^\times$ satisfies $\mathfrak{a}=a\co_F$ and $a_v=1$ for $v\mid\infty$ then the Hecke operator $T_\mathfrak{a}$ satisfies \cite[Proposition 3.1.4]{zhang1}
$$B(1;T_\mathfrak{a}\phi)= \mathrm{N}_{F/\Q}(\mathfrak{a}) \cdot B(a;\phi).$$


\subsection{Eisenstein series}
\label{ss:ES}


For any place $v$ of $F$ and any subset $X\subset F_v$ let $\mathbf{1}_X$ denote the characteristic function of $X$.  Let  $\mathcal{S}(\A^2)$ denote the space of Schwartz functions on $\A^2$ and fix $\Omega\in\mathcal{S}(\A^2)$.   Given a pair $\eta=(\eta_1,\eta_2)$ of quasi-characters of $\A^\times/F^\times$ we define 
$$
f_{\Omega,\eta,s}(g)  =  |\det(g)|^s \eta_1(\det(g)) \int_{\A^\times} \Omega \big( [0,t] \cdot g \big) |t|^{2s} \eta_1(t) \eta_2(t^{-1}) \ d^\times t
$$
for $s$ a complex variable and $g\in \GL_2(\A)$.  Then $f_{\Omega,\eta,s} $ lies in the space of the induced representation $\mathcal{B}(\eta_1 |\cdot|^{s-1/2} , \eta_2 |\cdot|^{1/2-s})$  of \cite[\S 2.2]{zhang2}. The Eisenstein series  defined by the meromorphic continuation of
$$
E_{\Omega,\eta,s}(g)  = \sum_{\gamma\in B(F)\backslash \GL_2(F)} f_{\Omega,\eta,s}(\gamma g)
$$
is an automorphic form with central character $\eta_1\eta_2$. If we set $w_0=  \left( \begin{matrix}  &  1  \\ -1&  \end{matrix} \right)$ then according to  \cite[\S 3.3]{zhang2} $E_{\Omega,\eta,s}(g)$ has constant term
\begin{equation*}
C_{\Omega,\eta,s}(g)  =  f_{\Omega,\eta,s}(g)+ \int_{\A} f_{\Omega,\eta,s}\left( w_0\left( \begin{matrix} 1 & x  \\ & 1 \end{matrix} \right) g \right) \ dx
\label{constant to section}
\end{equation*}
and Whittaker function
\begin{equation*}
W_{\Omega,\eta,s}(g) =  \int_{\A} 
f_{\Omega,\eta,s}\left( w_0\left( \begin{matrix} 1 & x  \\ & 1 \end{matrix} \right) g \right) \psi(-x) \ dx.  \\
\end{equation*}

To fix a particular Eisenstein series we let $\mathfrak{r}$ be an $\co_F$-ideal relatively prime to $\mathfrak{d}$ and  choose $r\in\A^\times$ so that  $r\co_F=\mathfrak{r}$ and $r_v=1$ for $v\mid\infty$.
Define a Schwartz function  $\Omega_\mathfrak{r}=\prod\Omega_{\mathfrak{r},v}$ on $\A^2$ by
$$
\Omega_{\mathfrak{r}, v} (x,y)= \left\{\begin{array}{ll}
\mathbf{1}_{\mathfrak{r}_v}(x)\mathbf{1}_{\co_{F,v}}(y) & \mathrm{if\ }v\nmid \mathfrak{d}\infty \\
\omega_v(y) \mathbf{1}_{\mathfrak{d}_v }(x)  \mathbf{1}_{\co_{F,v}^\times }(y)  & \mathrm{if\ }v\mid\mathfrak{d} \\
(ix+y) e^{-\pi(x^2+y^2)}&\mathrm{if\ }v\mid\infty.
\end{array}\right.
$$
Taking $\eta=(1,\omega)$ we abbreviate
$$
E_{\mathfrak{r},s}(g)=E_{\Omega_\mathfrak{r},\eta,s}(g)
\hspace{1cm}
f_{\mathfrak{r},s}(g)= f_{\Omega_\mathfrak{r},\eta,s}(g).
$$

\begin{Prop}\label{eisenstein coefficients}
Fix $a\in\A^\times$ and set $\mathfrak{a}=a\co_F$.  There is a product expansion $$B(a;E_{\mathfrak{r},s})=\prod B_v(a, E_{\mathfrak{r},s})$$ over all places $v$ of $F$, in which the local factors are given as follows.  
\begin{enumerate}
\item
If $v$ is a finite place which  does not divide $\mathfrak{d}$  then for any uniformizing parameter $\varpi$ of $F_v$
$$
B_v(a; E_{\mathfrak{r},s})
= \omega_v(\delta) \cdot  |a|_v^s \cdot |\delta|_v^{s-1/2}  \sum_{j=0}^{\ord_v(\mathfrak{ar}^{-1})} |\varpi^j|_v^{1-2s}\omega_v(\varpi^j).
$$
if $\ord_v(\mathfrak{a})\ge \ord_v(\mathfrak{r})$, and otherwise $B_v(a;E_{\mathfrak{r},s})=0$.
\item
If $v\mid\mathfrak{d}$ then 
$$
B_v(a;E_{\mathfrak{r},s})= 
\left\{\begin{array}{ll}
\omega_v(\delta)   |ad|_v^{s}  \cdot |\delta d |_v^{s-1/2}  \epsilon_v(1/2,\omega_v,\psi_v^0) & \mathrm{if\ }\ord_v(\mathfrak{a})\ge 0 \\
0&\mathrm{otherwise}
\end{array}\right.
 $$
and
 $$
 B_v(a; h_vE_{\mathfrak{r},1-s}) = \omega_v(-a) |d|_v^{3/2-3s}|\delta|_v^{1-2s} \epsilon_v(1/2,\omega,\psi_v^0)^{-1}  \cdot B_v(a;E_{\mathfrak{r},s}) 
 $$
where $\epsilon_v(1/2,\omega,\psi_v^0)$ is the usual local epsilon factor as in \cite[\S 3]{kudla-tate}.

 \item
 If $v$ is archimedean then 
 $$
 B_v(a;E_{\mathfrak{r},s}) =  
  \omega_v(a \delta)  |a|_v^{1-s}  |\delta|_v^{s-1/2} \frac{\Gamma(s+1/2)}{\pi^{s+1/2}}
  V_s(-a_v)
  $$
  where for $t\in \R$
  $$
V_s(t)= \int_{\R} \frac{e^{-2\pi i t x}}{(i+x)(1+x^2)^{s-1/2}}\ d^{\mathrm{Leb}}x.
 $$
 \end{enumerate}
 \end{Prop}

\begin{proof}
For $v$ nonarchimedean these formulas are found in Lemmas 3.3.2 and 3.3.3 of \cite{zhang2}.  For $v$ archimedean see \cite[Lemma 3.3.4]{zhang2}.  At each place our formulas differ from Zhang's by a factor of $\omega_v(-1)$.  As $\omega(-1)=1$ this local factor does not change the value of $B(a;E_{\mathfrak{r},s})$.
\end{proof}

\begin{Prop}\label{ES functional}
The Eisenstein series $E_{\mathfrak{r},s}(g)$ satisfies the functional equation
$$
E_{\mathfrak{r},s}(g) =E_{\mathfrak{r},1-s}(gh_S)\cdot (-i)^{[F:\Q]}  |r\delta|^{2s-1} |d|^{3s-3/2} \omega( r\cdot\det g) .
$$
\end{Prop}

\begin{proof}
See \S 3.2 of \cite{zhang2}, especially (3.2.1) and Lemmas 3.2.3 and 3.2.4.
\end{proof}

Let $L(s,\omega)=\prod_v L_v(s,\omega)$ be the usual  Dirichlet $L$-function attached to $\omega$, including the gamma factors $L_v(s,\omega)=G_1(s+1)$ for $v\mid\infty$.  Writing $L(s,\omega)$ as the quotient of the completed Dedekind $\zeta$-functions of $E$ and $F$ and using the functional equation and residue formulas of \cite[VII.6]{weil} gives the functional equation
\begin{equation}\label{dirichlet functional}
L(s,\omega)= |d\delta|^{s-1/2}\cdot L(1-s,\omega)
\end{equation}
and the special value formula
\begin{equation}\label{dirichlet value}
L(0,\omega) = \frac{ H_E}{ H_F} \cdot [\co_E^\times:\co_F^\times]^{-1}\cdot 2^{[F:\Q]-1}
\end{equation}
in which $H_F$ and $H_E$ are the class numbers of $F$ and $E$, respectively.

\begin{Prop}\label{constant term}
  Fix $a\in\A^\times$ and set $\alpha=\left( \begin{matrix} a\delta^{-1} &   \\ & 1 \end{matrix} \right)$.  For any $T\subset S$
 $$
 f_{\mathfrak{r},s}(\alpha h_{T}) =
\left\{ \begin{array}{ll} 
 |a|^s |\delta|^{-s} L(2s,\omega)  & \mathrm{if\ } T=\emptyset \\
 0 & \mathrm{otherwise}. \end{array}\right.
 $$
Furthermore if $T=S$ then
\begin{eqnarray*}\lefteqn{
  \int_{\A} f_{\mathfrak{r},s} \left( w_0\left( \begin{matrix} 1 & x  \\ & 1 \end{matrix} \right) \alpha h_T \right) \ dx  } \\
 & = & 
i^{[F:\Q]}\omega(a\delta)\omega(\mathfrak{r}) |r|^{2s-1}  |a|^{1-s} |\delta|^{3s-2} |d|^{3(s-1/2)} \cdot L(2-2s,\omega),  
\end{eqnarray*}
and otherwise the integral is $0$.
\end{Prop}

\begin{proof}
Let $v$ be a place of $F$ and, if $v$ is finite, let $\varpi$ be a uniformizing parameter of $F_v$. 
We may factor $f_{\mathfrak{r},s}=\prod_v f_{\mathfrak{r},s,v} $ where 
$$
f_{\mathfrak{r},s,v} (g)  =  |\det(g)|_v^s  \int_{F_v^\times} \Omega_{\mathfrak{r},v} \big( [0,t] \cdot g \big) |t|_v^{2s} \omega_v (t) \ d^\times t.
$$
For any place $v$ one easily computes
$$
 f_{\mathfrak{r},s,v}(\alpha)   = |a|_v^s\cdot  |\delta|_v^{-s} \cdot L_v(2s,\omega)
$$
 and, if $v\in S$, 
 $$
f_{\mathfrak{r},s,v} \left(\alpha h_v  \right) =  |a\delta^{-1}r |_v^s \int_{F_v^\times}  \Omega_{\mathfrak{r},v}(-r t,0) |t|_v^{2s}\omega_v(t)\ d^\times t 
$$
which vanishes as $\Omega_{\mathfrak{r},v}(-r t,0)=0$.  This proves the first claim.
If $v$ is a finite place with $v\nmid\mathfrak{d}$ then
\begin{eqnarray*}\lefteqn{
\int_{F_v} f_{\mathfrak{r},s,v}  \left( w_0\left(\begin{matrix} 1& x \\ & 0 \end{matrix}\right)\alpha \right) \ dx } \\
&=&
|a\delta^{-1}|_v^s \int_{F_v^\times} \mathbf{1}_{\mathfrak{r}_v}(ta\delta^{-1}) \left( \int_{F_v} \mathbf{1}_{\co_{F,v}} (tx) \ dx\right)|t|_v^{2s}\omega_v(t)\ d^\times t   \\
&=&
|a|_v^s |\delta|_v^{1/2-s} \int_{F_v^\times} \mathbf{1}_{\mathfrak{r}_v}(ta\delta^{-1}) |t|_v^{2s-1}\omega_v(t)\ d^\times t   \\
&=& \omega_v(a\delta )
|a|_v^{1-s}  |\delta|_v^{s-1/2}   |r|_v^{2s-1} \omega_v(r) L_v(2s-1,\omega).
\end{eqnarray*}
If $v\mid\mathfrak{d}$ then by (\ref{additive-multiplicative measures})
$$
\int_{F_v}\mathbf{1}_{\co_{F,v}^\times}(tx)\omega_v(x) \ dx 
=
|\delta|_v^{1/2} (1-|\varpi|_v) \int_{F_v^\times }\mathbf{1}_{\co_{F,v}^\times}(tx)\omega_v(x) |x|_v\ d^\times x .
$$
The  integral on the right vanishes, and hence so does
\begin{eqnarray*}\lefteqn{
\int_{F_v} f_{\mathfrak{r},s,v}  \left( w_0\left(\begin{matrix} 1& x \\ & 0 \end{matrix}\right)\alpha \right) \ dx } \\
 &=&
  |a\delta^{-1}|_v^s  \int_{F_v} \int_{F_v^\times} \Omega_{\mathfrak{r},v}( -ta\delta^{-1} ,-tx) |t|_v^{2s}\omega_v(t)\ d^\times t \ dx\\
  &=&
  |a\delta^{-1}|_v^s  \int_{F_v^\times} \mathbf{1}_{\mathfrak{d}_v}( ta\delta^{-1} ) \left( \int_{F_v}\mathbf{1}_{\co_{F,v}^\times}(tx)\omega_v(x) \ dx\right)  |t|_v^{2s}\ d^\times t.
 \end{eqnarray*}
Still assuming $v\mid\mathfrak{d}$,
\begin{eqnarray*}\lefteqn{
\int_{F_v} f_{\mathfrak{r}, s,v}  \left( w_0\left(\begin{matrix} 1& x \\ & 0 \end{matrix}\right) \alpha  h_v \right) \ dx } \\
&=&
|a d \delta^{-1}|^s_v  \int_{F_v^\times} \left(  \int_{F_v} \mathbf{1}_{\co_{F,v}}(t x) \ dx\right)  \mathbf{1}_{\co_{F,v}^\times} (ta\delta^{-1} ) |t|_v^{2s}\omega_v(-a\delta )\ d^\times t  \\
&=&
 |a|_v^{1-s} |d |_v^s  |\delta|^{s-1/2}_v \omega_v(-a\delta ) \int_{F_v^\times}  \mathbf{1}_{\co_{F,v}^\times} (ta\delta^{-1} )\ d^\times t  \\
&=&
\omega_v(-a\delta )  |a|_v^{1-s} |\delta |_v^{s-1/2}  |d|_v^s.
\end{eqnarray*}
Finally, if  $v$ is archimedean then
\begin{eqnarray*}\lefteqn{
\int_{F_v} f_{\mathfrak{r},s,v}  \left( w_0\left(\begin{matrix} 1& x \\ & 0 \end{matrix}\right)\alpha \right) \ dx } \\
&=&  - |a|_v^s |\delta|_v^{1/2-s}\int_{\R}\int_{\R^\times}   t
(a\delta^{-1}i+x) e^{-\pi (ta\delta^{-1})^2} e^{-\pi(tx)^2}
|t|_v^{2s}\omega_v(t)\ d^\times t  \ d^{\mathrm{Leb}}x\\
&=&  i\cdot \omega_v( -a\delta)  |a|_v^{s+1} |\delta|_v^{-1/2-s} \int_{\R^\times}   
e^{-\pi (ta\delta^{-1})^2} |t|_v^{2s+1}\left(  \int_{\R} e^{-\pi(tx)^2} \ d^{\mathrm{Leb}}x \right)
\ d^\times t  \\
&=&  i\cdot \omega_v( -a\delta)  |a|_v^{s+1} |\delta|_v^{-1/2-s} \int_{\R^\times}   
e^{-\pi (ta\delta^{-1})^2} |t|_v^{2s} \ d^\times t  \\
& = &
 i\cdot \omega_v( -a\delta)  |a|_v^{1-s} |\delta|_v^{s-1/2} \pi^{-s}\Gamma(s).
 \end{eqnarray*}
Putting everything together gives
\begin{eqnarray*}\lefteqn{
  \int_{\A} f_s\left( w_0\left( \begin{matrix} 1 & x  \\ & 1 \end{matrix} \right) \alpha h_T \right) \ dx  } \\
 & = & \left\{ \begin{array}{ll} 
i^{[F:\Q]}\omega(a\delta)\omega(r) |r|^{2s-1}  |a|^{1-s} |\delta|^{s-1/2} |d|^{s} \cdot L(2s-1,\omega)  & \mathrm{if\ } T=S\\
 0 & \mathrm{otherwise} \end{array}\right.
\end{eqnarray*}
and the second claim now follows from the functional equation  (\ref{dirichlet functional}).
\end{proof}


\subsection{Theta series}


As in \cite[\S 12]{jacquet-langlands} or  \cite[\S 2.2]{zhang2} (see also \S 12.6.1 and \S 12.6.5 of \cite{nek}, and the references therein) there is an irreducible automorphic representation $\Pi_\chi$ of $\GL_2(\A)$ of central character $\omega \chi_0$ and conductor $\mathfrak{d}\mathrm{N}(\mathfrak{C})$ characterized by  $L(s,\Pi_\chi)=L(s,\chi)$, where the right hand side is the  Dirichlet  $L$-function of $\chi$ including the gamma factors $L_v(s,\chi)=G_2(s)$ for $v\mid\infty$.  Denote by $\theta_\chi\in \Pi_\chi$  the normalized newvector and define
$$
\theta(g)= \theta_\chi\left(g \left(\begin{matrix} -1 &  \\ & 1  \end{matrix}\right)\right)
$$
so that $\theta$ has parallel weight $-1$.

\begin{Prop}\label{theta coefficients}
Fix $a\in \A^\times$. The Whittaker coefficient $B(a;\theta)$ admits a product decomposition
$B(a;\theta)=\prod_vB_v(a;\theta)$ over all places of $F$ in which the local factors are given as follows.  Let $v$ be a place of $F$, and if $v$ is finite let $\varpi$ be a uniformizing parameter of $F_v$.
\begin{enumerate}

\item
If $v$ is finite and inert in $K$ then  
$$
B_v(a;\theta)
= |a|_v^{1/2}  \cdot \left\{  
\begin{array}{ll} 
\chi_v(\varpi)^{ \frac{1}{2}\ord_v(a) } & \mathrm{if\ } \ord_v(a) \ge 0, \ \ord_v(a) \mathrm{\ even},\ \chi_v \mathrm{\ unramified}\\ 
1 & \mathrm{if\ } \ord_v(a) = 0,\ \chi_v \mathrm{\ ramified} \\
0 & \mathrm{otherwise}.
\end{array}\right.
$$

\item
If $v$ is finite and splits in $K$ then identify $E_v^\times\iso F_v^\times \times F_v^\times$. Set $\alpha=0$ if the restriction of $\chi_v$ to the first factor is ramified, and  $\alpha=\chi_v(\varpi,1)$ otherwise.  Set $\beta=0$ if the restriction of $\chi_v$ to the second factor is ramified, and $\beta=\chi_v(1,\varpi)$ otherwise.
Then
$$
B_v(a;\theta) = |a|_v^{1/2}    \sum_{\substack{i+j=\ord_v(\mathfrak{a}) \\ i,j\ge 0}} \alpha^i \beta^j.
$$
Here we adopt the convention that $0^0=1$ in case one or both of $\alpha$, $\beta$ is $0$.

\item
If $v\mid\mathfrak{d}$ (so that $\chi_v$ is unramified) let $\varpi_{E}$ denote a uniformizer of $E_v$.  Then
$$
B_v(a;\theta) 
=  |a|_v^{1/2}  \cdot  \left\{  
\begin{array}{ll} 
\chi_v(\varpi_{E})^{\ord_v(a)} & \mathrm{if\ } \ord_v(a)\ge 0 \\
0 & \mathrm{otherwise}.
\end{array}\right.
$$

\item
If $v$ is archimedean  then $B_v(a;\theta) = |a|_v^{1/2} e_v(-a)$.
\end{enumerate}
\end{Prop}

\begin{proof}
When $\chi_0$ is trivial this is a restatement of  Lemmas 3.3.6 and 3.3.7 of \cite{zhang2}.  The proof of the general case is identical.
\end{proof}

\begin{Prop}\label{theta functional}
The local Whittaker coefficients of $\theta$ satisfy
$$
\begin{array}{rcll}
\omega_v(a)B_v(a;\theta) &=& B_v(a;\theta) & \mathrm{if\ }v\nmid \mathfrak{d}\cdot \infty  \\
\omega_v(a)B_v(a;\theta) &=& - B_v(a;\theta)& \mathrm{if\ }v\mid  \infty  \\
\omega_v(a) B_v(a;h_v \theta) &=& \chi_v(\mathfrak{D}) \epsilon_v(1/2,\omega,\psi_v^0)\cdot  B_v(a;\theta)  & \mathrm{if\ }v\mid  \mathfrak{d}.\\
\end{array}
$$
 Furthermore $\theta$ satisfies the global functional equation 
$$
  \theta(g) = \theta(gh_S)\cdot  \omega(  \det g )\cdot \overline{\chi}(\mathfrak{D})\cdot (-i)^{[F:\Q]}.
  $$
\end{Prop}

\begin{proof}
When $\chi_0$ is trivial this is  \cite[Lemma 3.2.5]{zhang2}, and the proof of the general case is identical.
\end{proof}

\begin{Lem}\label{cusp theta}
Let $\chi^*(t)=\chi(\overline{t})$ where $t\mapsto \overline{t}$ is the nontrivial involution of $E/F$, extended to $\A_E^\times$. 
The following are equivalent
\begin{enumerate}
\item $\Pi_\chi$ is noncuspidal
\item
there is a character $\nu: \A^\times/F^\times  \map{}\C^\times$ such that $\chi=\nu\circ \mathrm{N}$
\item
$\chi^*=\chi.$
\end{enumerate}
\end{Lem}

\begin{proof}
If (b) does not hold then  $\Pi_\chi$ is cuspidal by  \cite[Proposition 12.1]{jacquet-langlands}.  Conversely, if (b) does hold then comparing $L$-functions we see that $\Pi_\chi$ is isomorphic to (indeed, is defined as) the principal series $\Pi(\nu,\nu\omega)$, hence is noncuspidal. Thus (a) and (b) are equivalent.  The equivalence of  (b) and  (c) is a consequence of Hilbert's theorem 90.
\end{proof}

\begin{Lem}\label{Lem:theta and eisenstein}
Assume that $\mathfrak{C}=\co_E$ and that the equivalent conditions of Lemma \ref{cusp theta} hold. Then
\begin{equation}\label{Disp:theta and eisenstein}
 \nu(\det g)\cdot E_{\co_F,1/2}(g)=(-1)^{[F:\Q]} |d|^{1/2}  \theta(g)
 \end{equation}
 where $E_{\co_F,s}$ is the Eisenstein series of \S \ref{ss:ES} with $\mathfrak{r}=\co_F$.
\end{Lem}

\begin{proof}
As in the proof of Lemma \ref{cusp theta}, $\Pi_\chi$ is isomorphic to $\Pi(\nu,\nu\omega)$, and so is generated by $\nu(\det g)E_{\co_F,1/2}(g)$.  As both $\theta(g)$ and $\nu(\det g)E_{\co_F,1/2}(g)$ are $K_1(\mathfrak{d})$-fixed and of parallel weight $-1$, they must be scalar multiples of one another.  To compute the scalar we compute Whittaker coefficients.
For any $a\in \A^\times$, comparing Propositions \ref{eisenstein coefficients} and \ref{theta coefficients} gives
$$
B_v(a;E_{\co_F,1/2}) = \overline{\nu}_v(a) \omega_v(a\delta) B_v(a;h_v \theta) \cdot
\left\{\begin{array}{ll}
\overline{\chi}_v(\mathfrak{D}) |d|_v^{1/2}     &\mathrm{if\ } v\nmid\infty \\
i  &\mathrm{if\ } v\mid\infty
\end{array}\right.
$$
Using Proposition \ref{theta coefficients} we see that both sides of (\ref{Disp:theta and eisenstein}) have the same Whittaker coefficients.
\end{proof}


\subsection{The kernel $\Theta$}
\label{ss:kernel}


For each $v\in S$ set
 $ \sigma_{s,v}= 1+\overline{\chi}_v(\mathfrak{D}) |d|_v^{1/2-s} h_v$
 and define the \emph{symmetrized kernel}
 \begin{eqnarray*}
 \Theta_{\mathfrak{r},s}(g)
 &=&\left(\prod_{v\in S} \sigma_{s,v} \right) \cdot \big[\theta(g)E_{\mathfrak{r},s}(g)\big]  \\
 &=&\nonumber
 \sum_{T\subset S}\overline{\chi}_T(\mathfrak{D})|d|_T^{1/2-s} \theta(gh_T)E_{\mathfrak{r},s}(gh_T)
 \end{eqnarray*}
 where the subscript $T$ indicates the product over all $v\in T$; e.g. $\chi_T=\prod_{v\in T}\chi_v$.  For every place $v$ of $F$ define
\begin{equation}\label{epsilon}
\epsilon_v(s,\mathfrak{r},\psi)= |\delta|_v^{2s-1}\cdot
\left\{
\begin{array}{ll}
-1 &\mathrm{if\ }v\mid\infty \\
\omega_v(r)|r|_v^{2s-1} & \mathrm{if\ }v\mid\mathfrak{r}\\
|d|_v^{2s-1}&\mathrm{otherwise}
\end{array}\right.
\end{equation}
and set $\epsilon(s,\mathfrak{r})=\prod_v\epsilon_v(s,\mathfrak{r},\psi)$, so that
$$
\epsilon(s,\mathfrak{r})=(-1)^{[F:\Q]} \omega(\mathfrak{r}) \mathrm{N}_{F/\Q}(\mathfrak{dr})^{1-2s} D_F^{1-2s}.
$$
Combining Propositions \ref{ES functional} and  \ref{theta functional} gives the relation
$$
\theta(g) E_{\mathfrak{r},s}(g)= \epsilon(s,\mathfrak{r}) |d|^{s-1/2} \overline{\chi}(\mathfrak{D}) \cdot \theta(g h_S)E_{\mathfrak{r},1-s}(gh_S)
 $$
 and hence
 $$
 \left(\prod_{v\in S} \sigma_{s,v}\right) \big[\theta(g) E_{\mathfrak{r},s}(g) \big]  
 =
\epsilon(s,\mathfrak{r})  \left(\prod_{v\in S} \overline{\chi}_v(\mathfrak{D})  |d|^{s-1/2}_v\sigma_{s,v}  h_v\right) \big[\theta (g) E_{\mathfrak{r},1-s}(g) \big].
$$ 
For $v\in S$ the operator $h_v^2$ acts  as $\chi_{0,v}(\mathfrak{d})=\chi_v(\mathfrak{D})^2$  on automorphic forms of central character $\chi_0$.  Thus we may replace the expression $\overline{\chi}_v(\mathfrak{D})  |d|^{s-1/2}_v\sigma_{s,v}  h_v$ with $\sigma_{1-s,v}$ to arrive at the functional equation
\begin{equation}\label{kernel functional}
\Theta_{\mathfrak{r},s} (g) =\epsilon(s,\mathfrak{r})  \cdot\Theta_{\mathfrak{r},1-s}(g).
\end{equation}

As in \cite[\S 3.3]{zhang2}, multiplying the Fourier expansions of  $\theta(g)$ and  $E_{\mathfrak{r},s}(g)$ shows that the product $\theta(g) \cdot E_{\mathfrak{r},s}(g)$ has constant term
$$
\mathbf{C}_{\mathfrak{r},s}(g) = C_\theta(g)C_{\mathfrak{r},s}(g)  + 
 \sum_{\substack{\eta,\xi\in F^\times \\ \eta+\xi=0}} 
W_\theta\left(\left(  \begin{matrix}   \eta  & \\ & 1  \end{matrix}  \right)   g  \right)  
W_{\mathfrak{r},s}\left(\left(  \begin{matrix}   \xi  & \\ & 1  \end{matrix}  \right)   g  \right)
$$
 and Whittaker function
\begin{eqnarray*}
\mathbf{W}_{\mathfrak{r},s}(g) &=&
C_\theta(g)  W_{\mathfrak{r},s} (g)  + C_{\mathfrak{r},s}(g) W_\theta(g)  \\
& & + \  \sum_{\substack{\eta,\xi\in F^\times \\ \eta+\xi=1}} 
W_\theta\left(\left(  \begin{matrix}   \eta  & \\ & 1  \end{matrix}  \right)   g  \right)  
W_{\mathfrak{r},s}\left(\left(  \begin{matrix}   \xi  & \\ & 1  \end{matrix}  \right)   g  \right).
\end{eqnarray*}
From the Fourier expansion of $\theta(g)E_{\mathfrak{r},s}(g)$ and the definition of the symmetrized kernel we find the decomposition 
\begin{equation}\label{global kernel decomp} 
B(a;\Theta_{\mathfrak{r},s}) = A_0(a;\Theta_{\mathfrak{r},s}) +A_1(a;\Theta_{\mathfrak{r},s}) + \sum_{ \substack{\eta,\xi\in F^\times \\ \eta+\xi=1}  } B (a,\eta,\xi;\Theta_{\mathfrak{r},s}) 
\end{equation}
in which the terms on the right hand side are defined by
\begin{eqnarray*}
A_0(a;\Theta_{\mathfrak{r},s}) &=&
\sum_{T\subset S} \overline{\chi}_T(\mathfrak{D})|d|_T^{1/2-s} W_\theta(\alpha h_T)  C_{\mathfrak{r},s}(\alpha h_T) \\
A_1(a;\Theta_{\mathfrak{r},s}) &=&  \sum_{T\subset S} \overline{\chi}_T(\mathfrak{D})|d|_T^{1/2-s} C_\theta(\alpha h_T)W_{\mathfrak{r},s}(\alpha h_T) \\
B (a,\eta,\xi;\Theta_{\mathfrak{r},s})  &=&
 \sum_{T\subset S}  \overline{\chi}_T(\mathfrak{D})|d|_T^{1/2-s}  B(\eta a;h_T \theta) B(\xi a; h_T E_{\mathfrak{r},s})   
\end{eqnarray*}
where we have abbreviated $\alpha=\left(\begin{matrix} a\delta^{-1} & \\ & 1 \end{matrix}  \right)$.  If we define
\begin{eqnarray*}\lefteqn{
B_v(a,\eta,\xi;\Theta_{\mathfrak{r},s}) }  \\
&= &
B_v(\eta a;\theta)
\cdot\left\{\begin{array}{ll}
B_v(\xi a;E_{\mathfrak{r},s}) &\mathrm{if\ } v\nmid\mathfrak{d}  \\
B_v(\xi a;E_{\mathfrak{r},s}) 
+ \omega_v(-\eta\xi) |d\delta |_v^{2s-1} B_v(\xi a; E_{\mathfrak{r},1-s}) &\mathrm{if\ }v\mid\mathfrak{d}
\end{array}\right.
\end{eqnarray*}
then the local functional equations of Propositions \ref{eisenstein coefficients} and \ref{theta coefficients} imply the  factorization
$$
B (a,\eta,\xi;\Theta_{\mathfrak{r},s})  = \prod_v B_v(a,\eta,\xi;\Theta_{\mathfrak{r},s}).
$$

\begin{Lem}\label{local kernel functional}
For every place $v$ of $F$, every $a\in\A^\times$, and every  $\eta,\xi\in F^\times$,
$$
B_v(a,\eta,\xi;\Theta_{\mathfrak{r},s}) = \omega_v(-\eta\xi) \epsilon_v(s, \mathfrak{r},\psi)\cdot B_v(a,\eta,\xi;\Theta_{\mathfrak{r},1-s}) .
$$
\end{Lem}

\begin{proof}
This follows from direct examination of the explicit formulas of Propositions \ref{eisenstein coefficients} and \ref{theta coefficients}.  For $v\mid\infty$ one also uses  the functional equation satisfied by $V_s(t)$ found in \cite[Proposition IV.3.3 (c)]{gross-zagier}.
\end{proof}

\begin{Prop}\label{kernel coefficients}
Suppose $\eta,\xi\in F^\times$, $\eta+\xi=1$, and $\omega_v(-\eta\xi)=\epsilon_v(1/2,\mathfrak{r},\psi)$.
Fix $a\in \A^\times$ and abbreviate, here and later, $\Theta_\mathfrak{r}=\Theta_{\mathfrak{r},1/2}$.
\begin{enumerate}
\item
If $v$ is a finite place which  is split in $E$ then  
$$
B_v(a,\eta,\xi;\Theta_{\mathfrak{r}})
=|a|_v |\eta \xi |_v^{1/2}\omega_v(\delta)\big(\ord_v(\xi \mathfrak{ar}^{-1})+1\big)
 \sum_{  \substack{i+j=\ord_v(\eta \mathfrak{a} ) \\ i,j\ge 0}  }  \alpha^i\beta^j
$$
if $\ord_v(\eta\mathfrak{a})$ and $\ord_v(\xi\mathfrak{ar}^{-1})$ are nonnegative, and is $0$ otherwise.
Here $\alpha$ and $\beta$ are as in Proposition \ref{theta coefficients}.

\item
Suppose $v$ is a finite place which  is inert in $E$.  If $\chi_v$ is unramified then
$$
B_v(a,\eta,\xi;\Theta_{\mathfrak{r}})
=|a|_v |\eta \xi |_v^{1/2}\omega_v(\delta)
\chi_v(\varpi)^{\frac{1}{2}\ord_v(\eta a)}
$$
if $\ord_v(\eta\mathfrak{a})$ and $\ord_v(\xi\mathfrak{ar}^{-1})$ are even and nonnegative, and is $0$ otherwise.  If $\chi_v$ is ramified then
$$
B_v(a,\eta,\xi;\Theta_{\mathfrak{r}})
=|a|_v |\eta \xi |_v^{1/2}\omega_v(\delta)
$$
if $\ord_v(\eta\mathfrak{a})=0$ and $\ord_v(\xi\mathfrak{ar}^{-1})$ is even and nonnegative, and is $0$ otherwise. 

\item
If $v\mid\mathfrak{d}$ then 
$$
B_v(a,\eta,\xi;\Theta_{\mathfrak{r}})= 2 \chi_v(\varpi_E)^{\ord_v(\eta a)}
\omega_v(\delta)   |\eta \xi d|_v^{1/2} |a|_v   \epsilon_v(1/2,\omega_v,\psi_v^0)
$$
if $\ord_v(\eta\mathfrak{a})$ and $\ord_v(\xi\mathfrak{a})$ are nonnegative, and is $0$ otherwise.

 \item
 If $v$ is archimedean then 
 $$
 B_v(a,\eta,\xi;\Theta_{\mathfrak{r}})=
 2 i |\eta \xi |_v^{1/2}  |a|_v  \omega_v(\delta) \cdot e_v(-a).
 $$
 \end{enumerate}
\end{Prop}

\begin{proof}
This follows from Propositions \ref{eisenstein coefficients} and \ref{theta coefficients}.  For $v\mid\infty$ one also uses the special value formula for $V_{1/2}(t)$ found in \cite[Proposition IV.3.3 (d)]{gross-zagier}, which implies
$$
 B_v(a;E_{\mathfrak{r},1/2}) = -i |a|_v^{1/2}\omega_v(\delta)  \cdot e_v(-a).
$$
\end{proof}


\subsection{The Rankin-Selberg $L$-function}
\label{RS}


Recall the  automorphic representation $\Pi$ of $\GL_2(\A)$ of \S \ref{notations} and assume Hypothesis \ref{hyp}.  Fix a Haar measure on $\GL_2(\A_f)$ and let $Z$ denote the center of $\GL_2$.  Setting $F_\infty=F\otimes_\Q\R$   we identify 
$$
\GL_2(F_\infty)/Z(F_\infty)\mathrm{SO}_2(F_\infty)\iso \mathcal{H}^{[F:\Q]}
$$
in the usual way, where $\mathcal{H}=\C- \R$ is the union of the upper and lower half-planes equipped with the hyperbolic volume form $y^{-2}dxdy$.  Suppose $F_0$ and  $F_1$ are two automorphic forms on $\GL_2(\A)$ for which $F_0\overline{F}_1$ is a square integrable function on $\GL_2(F)\backslash \mathcal{H}^{[F:\Q]}\times \GL_2(\A_f)/Z(\A_f)$. If  $K\subset \GL_2(\A_f)$ is a compact open subgroup we define the Petersson inner product of level $K$ 
$$
\langle F_0, F_1\rangle_K= \mathrm{Vol}(K)^{-1}\int_{\GL_2(F)\backslash \mathcal{H}^{[F:\Q]}\times \GL_2(\A_f)/Z(\A_f)} F_0\overline{F_1}
$$ 
where   the quotient measure is induced by the Haar measure on $Z(\A_f)$ giving $\widehat{\co}_F^\times$ volume $1$.  For any $b\in\A^\times$ with trivial archimedean components set $R_b=\left(  \begin{matrix} b^{-1} &  \\ & 1  \end{matrix}\right)$ and view $R_b$ as an operator on automorphic forms by $(R_b\phi)(g)=\phi(gR_b)$.  Let $\mathfrak{b}$ be an ideal of $\co_F$ dividing $\mathfrak{dc}^2\mathfrak{s}^{-1}$ and fix  $b\in \A^\times$  with trivial archimedean components and $b\co_F=\mathfrak{b}$. Let
$L(s,\Pi\times\Pi_\chi)$ be the Rankin-Selberg $L$-function defined as in \cite[\S 2.5]{zhang2} (see also \cite[\S 12.6.2]{nek} and the references therein).

\begin{Prop}\label{first rankin integral}
Let $\phi_\Pi\in \Pi$ be the normalized newvector and set $\mathfrak{r}=\mathfrak{mc}^2$. 
Assume that  $\Pi_v$ is a discrete series of weight $2$ for each $v\mid\infty$.  Then
$$
\mathrm{Vol} (K_0(\mathfrak{dr}) )^{-1}  \int \phi_\Pi (gR_b) \theta(g) E_{\mathfrak{r},s} (g)\ dg
= |\delta|^{1/2-s} |b|^{s-1}B(b;\theta)  L(s,\Pi\times \Pi_\chi).
$$
\end{Prop}

\begin{proof}
Hypothesis \ref{hyp} implies that for every finite place $v$ either $\Pi_v$ or $\Pi_{\chi,v}$ is a principal series.  Hence the claim  follows from Propositions 2.5.1 and 2.5.2 of \cite{zhang2}.
\end{proof}

Under the notation and assumptions of Proposition \ref{first rankin integral}, a direct calculation as in \cite[Lemma 3.1.2]{zhang2} gives
\begin{equation}
\label{rankin integral}
\langle R_b \phi_\Pi , \overline{\Theta}_{\mathfrak{r},s} \rangle_{K_0(\mathfrak{dr})} =
 L(s,\Pi\times\Pi_\chi) \cdot  |\delta|^{1/2-s} \prod_{v\mid{\mathfrak{dc}}}\gamma_{s,v}(b)
 \end{equation}
 where
 $$ 
\gamma_{s,v}(b)=
 |b|_v^{-1/2} B_v(b;\theta)\left\{\begin{array}{ll}
  |b|_v^{s-1/2}  +  |b|_v^{1/2-s} 
 & \mathrm{if\ } v\mid\mathfrak{d}\\
 1 &  \mathrm{if\ } v\mid\mathfrak{c}.
 \end{array}\right.
$$


\subsection{Central derivatives and holomorphic projection}
\label{ss:derivative}


Throughout \ref{ss:derivative} we assume that  $\epsilon(1/2,\mathfrak{r})=-1$.  For any $\eta,\xi\in F^\times$ with $\eta+\xi=1$ define the \emph{difference set}
$$
\mathrm{Diff}_\mathfrak{r}(\eta,\xi)=
\{ \mathrm{ places\ }v \mathrm{\ of\ }F \mid \omega_v(-\eta\xi)\not=\epsilon_v(1/2,\mathfrak{r},\psi) \}.
$$
Note that the cardinality of $\mathrm{Diff}_\mathfrak{r}(\eta,\xi)$ is odd, and that  Lemma \ref{local kernel functional}  implies that $B_v (a,\eta,\xi,\Theta_{\mathfrak{r}})=0$ for each $v\in \mathrm{Diff}_\mathfrak{r}(\eta,\xi)$.  In particular $B(a,\eta,\xi;\Theta_{\mathfrak{r}})=0$.  Note also that $\mathrm{Diff}_\mathfrak{r}(\eta,\xi)$ contains only places which are nonsplit in $E$, as $v$ split implies that both $\omega_v(-\eta\xi)$ and $\epsilon_v(1/2,\mathfrak{r},\psi)$ are equal to $1$.  Define
$$
\Theta'_{\mathfrak{r}}(g)=\left.\frac{d}{ds} \Theta_{\mathfrak{r},s}(g) \right|_{s=1/2}
$$
and, with notation as in (\ref{global kernel decomp}), abbreviate
\begin{eqnarray*}
A_i(a;\Theta'_\mathfrak{r}) &=&  \frac{d}{ds} A_i(a;\Theta_{\mathfrak{r},s})\big|_{s=1/2}\\
B(a,\eta,\xi,\Theta'_{\mathfrak{r}}) &=& \frac{d}{ds} B(a,\eta,\xi;\Theta_{\mathfrak{r},s})\big|_{s=1/2}
\end{eqnarray*}
and similarly with $B(\cdot)$ replaced by $B_v(\cdot)$.  For $t$ a positive real number define
$$
q_0(t)=\int_{1}^\infty e^{-xt} \ d^\times x.
$$

\begin{Prop}\label{derivative coefficients}
 If $w\in \mathrm{Diff}_\mathfrak{r}(\eta,\xi)$ then 
$$
B(a,\eta,\xi;\Theta'_{\mathfrak{r}}) = B_w (a,\eta,\xi,\Theta'_{\mathfrak{r}}) 
\cdot \prod_{v\not=w} B_v(a,\eta,\xi;\Theta_{\mathfrak{r}}).
$$  
The value of $B_w (a,\eta,\xi,\Theta'_{\mathfrak{r}})$ is given as follows.
\begin{enumerate}
\item
Suppose $w\nmid\infty$ is inert in $E$.  If $\chi_w$ is unramified then
$$
B_w (a,\eta,\xi,\Theta'_{\mathfrak{r}})= \omega_w(\delta)|\eta\xi|_w^{1/2}|a|_w \log|\xi ar^{-1}\varpi|_w\chi_w(\varpi)^{\frac{1}{2}\ord_v(a\eta)}
$$
if $\ord_w(\eta a)$ is even and nonnegative and $\ord_w(\xi ar^{-1})$ is odd and nonnegative; otherwise the left hand side is $0$.  If $\chi_w$ is ramified then 
$$
B_w (a,\eta,\xi,\Theta'_{\mathfrak{r}})= \omega_w(\delta)|\eta\xi|_w^{1/2}|a|_w \log|\xi ar^{-1}\varpi|_w 
$$
if $\ord_w(\eta a)=0$  and $\ord_w(\xi ar^{-1})$ is odd and nonnegative; otherwise the left hand side is $0$.  
\item
If $w\nmid\infty$ is ramified in $E$ then
$$
B_w (a,\eta,\xi,\Theta'_{\mathfrak{r}})=
2\omega_w(\delta) |\eta\xi|_w^{1/2}|a|_w |d|_w^{1/2}\chi_w(\varpi_E)^{\ord_w(\eta\mathfrak{a})}\cdot\epsilon_w(\omega,\psi_w^0)\cdot\log|\xi ad|_w
$$
if $\ord_w(\eta\mathfrak{a})$ and $\ord_w(\xi\mathfrak{a})$ are  nonnegative; otherwise the left hand side is $0$.
\item
If $w\mid\infty$  then
$$
B_w (a,\eta,\xi,\Theta'_{\mathfrak{r}})=
-4i\omega_w(\delta) |\eta \xi|_w^{1/2} |a|_w e^{2\pi a_w} q_0(4\pi a_w\xi_w)
$$
if $\eta_w a_w<0$ and $\xi_w a_w>0$; otherwise the left hand side is $0$.
\end{enumerate}
\end{Prop}

\begin{proof}
The first claim follows from Lemma \ref{local kernel functional} and the remaining claims follow from the formulas of Propositions \ref{eisenstein coefficients} and \ref{theta coefficients}, together with the equality
$$
\left.\frac{d}{ds}V_s(t) \right|_{s=1/2}= -2\pi i e^{-2\pi t} q_0(-4\pi t)
$$
for $t<0$, which is found in \cite[Proposition IV.3.3(e)]{gross-zagier}.
\end{proof}

\begin{Rem}\label{single difference}
It follows from Lemma \ref{local kernel functional} and the first claim of Proposition \ref{derivative coefficients} that $B(a,\eta,\xi;\Theta'_\mathfrak{r})$ vanishes unless $\mathrm{Diff}_\mathfrak{r}(\eta,\xi)$ consists of a single place, necessarily nonsplit in $E$.
\end{Rem}

Let $\Phi_\mathfrak{r}(g)$ be the holomorphic projection of $\overline{\Theta'_{\mathfrak{r}}(g)}$. 
Thus $\Phi_\mathfrak{r}$ is the unique holomorphic cusp form on $\GL_2(\A)$ of parallel weight $2$ such that  $\langle \phi, \Phi_\mathfrak{r}\rangle_K=\langle\phi,\overline{\Theta_\mathfrak{r}'}\rangle_K$ for any cusp form $\phi$ and any compact open subgroup $K$.   If the representation $\Pi$ of \S \ref{RS}  is discrete of weight $2$ at every archimedean place then (\ref{rankin integral})
implies
$$
\langle \phi_\Pi, \Phi_\mathfrak{r} \rangle_{K_0(\mathfrak{dr})} = 2^{|S|} L'(1/2,\Pi\times\Pi_\chi).
$$
We now describe the coefficients $\widehat{B}(\mathfrak{a},\Phi_\mathfrak{r})$ as in \cite[\S 3.5]{zhang2} (see also \cite[\S 6.4]{zhang1}).   If $w$ is a finite place of $F$ define
\begin{equation}
\label{projection coefficient}
\widehat{B}^w(\mathfrak{a};\Phi_\mathfrak{r})  = ( -2i )^{[F:\Q]}  \omega_\infty(\delta) 
 \sum_{\eta,\xi}  |\eta \xi|_\infty^{1/2} \cdot \overline{B_w(a,\eta,\xi ;\Theta'_{\mathfrak{r}}) }  \prod_{v\nmid w\infty}\overline{B_v(a,\eta,\xi ;\Theta_{\mathfrak{r}}) } 
 \end{equation}
 where the sum is over all $\eta,\xi\in F^\times$ with $\eta+\xi=1$ and $\mathrm{Diff}_\mathfrak{r}(\eta,\xi)=\{w\} $.  This sum is finite and is $0$ for all but finitely many $w$. For $t,\sigma \in\R$ with $\sigma>0$ define
$$
M_\sigma(t)= \left\{\begin{array}{ll}
 \int_1^\infty \frac{-d^\mathrm{Leb}x}{x(1- t x)^{1+\sigma}} & \mathrm{if\ } t<0 \\
 0&\mathrm{otherwise.}
 \end{array} \right.
$$ 
If $w\mid\infty$ then we set
\begin{equation}\label{holomorphic infinity}
\widehat{B}^w(\sigma,\mathfrak{a};\Phi_\mathfrak{r}) =
 (-2i)^{[F:\Q]}  \omega_\infty(\delta)  
\sum_{\eta,\xi} |\eta\xi|_\infty^{1/2} M_\sigma(\xi_w)
\cdot \prod_{v\nmid\infty} \overline{ B_v(a,\eta,\xi; \Theta_{\mathfrak{r}}) }
\end{equation}
where the sum is over all $\eta,\xi\in F^\times$ with $\eta+\xi=1$ and $\mathrm{Diff}_\mathfrak{r}(\eta,\xi)=\{w\} $.

\begin{Prop}\label{holomorphic coefficients}
The Fourier coefficient $\widehat{B}(\mathfrak{a};\Phi_\mathfrak{r})$ decomposes as
$$
\widehat{B}(\mathfrak{a};\Phi_\mathfrak{r}) = 
A(\mathfrak{a})+D(\mathfrak{a})+\sum_{w\nmid\infty} \widehat{B}^w(\mathfrak{a};\Phi_\mathfrak{r})   + \mathrm{const}_{\sigma\to 0} \sum_{w\mid\infty} \widehat{B}^w(\sigma,\mathfrak{a};\Phi_\mathfrak{r})
$$
in which $A(\mathfrak{a})$ is a derivation of $\Pi_{\overline{\chi}}\otimes |\cdot|^{1/2}$ and $D(\mathfrak{a})$ is a sum of  derivations of principal series in the sense of \cite[Definition 3.5.3]{zhang2}.
\end{Prop}

\begin{proof}
When $\chi_0$ is trivial this is exactly \cite[Proposition 3.5.5]{zhang2}, and the proof when $\chi_0$ is nontrivial is exactly the same.
\end{proof}


\subsection{The weight zero kernel}
\label{ss:zero kernel}


We define  an automorphic form $\Theta^*_{\mathfrak{r,s}}$ in exactly the same way as $\Theta_{\mathfrak{r},s}$ but replacing $\theta$ by $\theta_\chi$ everywhere in the construction of \S \ref{ss:kernel}.  Thus
$$
\Theta^*_{\mathfrak{r},s}(g)
=\left(\prod_{v\in S} \sigma_{s,v} \right) \cdot \big[\theta_\chi(g)E_{\mathfrak{r},s}(g)\big] 
 $$
is a nonholomorphic form of parallel weight $0$.  Using the relation
 $$
 B_v(a;\theta_\chi)=\left\{\begin{array}{ll} 
 B_v(a;\theta) &\mathrm{if\ } v\nmid\infty \\
 B_v(-a;\theta) &\mathrm{if\ } v\mid\infty
 \end{array}\right.
 $$
and repeating the arguments of \S \ref{ss:kernel} we find that the weight zero kernel satisfies the functional equation
$$
\Theta^*_{\mathfrak{r},s} (g) = (-1)^{[F:\Q]}\epsilon(s,\mathfrak{r})  \cdot\Theta^*_{\mathfrak{r},1-s}(g)
$$
and admits a decomposition 
$$
B(a;\Theta^*_{\mathfrak{r},s}) = A_0(a;\Theta^*_{\mathfrak{r},s}) +A_1(a;\Theta^*_{\mathfrak{r},s}) + \sum_{ \substack{\eta,\xi\in F^\times \\ \eta+\xi=1}  } B (a,\eta,\xi;\Theta^*_{\mathfrak{r},s}) 
$$
in which $A_0$ and $A_1$ are defined exactly as in \S \ref{ss:kernel} but with $\theta$ replaced by $\theta_\chi$.  There is a further product decomposition
$$
B (a,\eta,\xi;\Theta^*_{\mathfrak{r},s})=\prod_v B_v (a,\eta,\xi;\Theta^*_{\mathfrak{r},s})
$$
where for $v\nmid\infty$ one has $B_v (a,\eta,\xi;\Theta^*_{\mathfrak{r},s})=B_v (a,\eta,\xi;\Theta_{\mathfrak{r},s})$   while for $v\mid\infty$ 
$$
B_v (a,\eta,\xi;\Theta^*_{\mathfrak{r},s})=
\left\{\begin{array}{ll}
-4i |a|_v |\eta\xi|_v^{1/2} \omega_v(\delta)e^{-2\pi a_v(1-2\xi_v)} &\mathrm{if\ }\omega_v(-\eta\xi)=1, \xi_v a_v<0 \\
0 & \mathrm{otherwise.}
\end{array}\right.
$$

Assume that the representation $\Pi$ of \S \ref{notations} satisfies Hypothesis \ref{hyp} and  is a weight $0$ principal series for every archimedean $v$.   The Rankin-Selberg $L$-function $L(s,\Pi\times\Pi_\chi)$ is defined exactly as in \S \ref{RS}, but with the archimedean factors now given by  \cite[(5.4)]{zhang3}.  With notation as in Proposition \ref{first rankin integral} one again  has the integral representation of the Rankin-Selberg $L$-function
\begin{equation}\label{maass rankin}
  \langle R_b \phi_\Pi,  \overline{\Theta^*_{\mathfrak{r},s}}\rangle_{ K_0(\mathfrak{dr}) } =
 L(s,\Pi\times\Pi_\chi) \cdot  |\delta|^{1/2-s} \prod_{v\mid{\mathfrak{dc}}}\gamma_{s,v}(b)
 \end{equation}
exactly as in (\ref{rankin integral}).


\subsection{The quasi-new line}
\label{ss:quasi-new}


Suppose the representation $\Pi$ of \S \ref{notations} satisfies Hypothesis \ref{hyp} and is unitary.  Set $\mathfrak{r}=\mathfrak{mc}^2$.  Fix a place $v$ of $F$ dividing $\mathfrak{dc}$ and a uniformizer $\varpi$ of $F_v$.   As $\Pi_v$ has conductor $\mathfrak{s}_v=\mathfrak{n}_v$, \cite[Proposition 2.3.1]{zhang2} implies that the space of $K_1(\mathfrak{r}_v)$ fixed vectors of $\Pi_v$ is finite dimensional with basis 
$$
\{ R_{\varpi^k}\phi_{\Pi,v} \mid 0\le k\le \ord_v(\mathfrak{rs}^{-1}) \}
$$
where $\phi_{\Pi,v}$ is any newvector in $\Pi_v$ and $R_b$ is as in \S \ref{RS}.  Define a linear functional $\Lambda_v$ on this finite dimensional vector space by the condition
$$
\Lambda_v(R_{\varpi^k}\phi_{\Pi,v} ) = \gamma_{\frac{1}{2},v}(\varpi^k)
$$
where, in the notation of (\ref{rankin integral}),
$$
\gamma_{\frac{1}{2},v}(b)=
 |b|_v^{-1/2} B_v(b;\theta)\left\{\begin{array}{ll}
 2 & \mathrm{if\ } v\mid\mathfrak{d}\\
 1 &  \mathrm{if\ } v\mid\mathfrak{c}.
 \end{array}\right.
$$

\begin{Def}
If $v\mid\mathfrak{dc}$ then the \emph{quasi-new line} in $\Pi_v$ is the orthogonal complement, in the space of $K_1(\mathfrak{r}_v)$ fixed vectors, of the kernel of $\Lambda_v$.  If $v\nmid\mathfrak{dc}$ then the \emph{quasi-new line} is defined to be the span of the  newvectors in  $\Pi_v$, i.e. the line of $K_1(\mathfrak{m}_v)=K_1(\mathfrak{r}_v)$ fixed vectors.  The quasi-new line in $\Pi=\bigotimes_v \Pi_v$ is the tensor product of the local quasi-new lines, and a \emph{quasi-newform} in $\Pi$ is any nonzero vector on the quasi-new line.
\end{Def}

\begin{Prop}\label{quasi-new kernels}
Assume that either $\Pi$ or $\Pi_\chi$ is cuspidal and that $\Pi_v$ is discrete of weight $2$ at each archimedean $v$. The projection of $\overline{\Theta_\mathfrak{r}(g)}$ to $\Pi$ lies on the quasi-new line; if, in addition, $\epsilon(1/2,\mathfrak{r})=-1$ then the projection of $\Phi_\mathfrak{r}(g)$ to $\Pi$ lies on the quasi-new line.  If  we instead assume that $\Pi$ has weight $0$ at every archimedean place then  the projection of $\overline{\Theta_\mathfrak{r}^*(g)}$ to $\Pi$ lies on the quasi-new line.
\end{Prop}

\begin{proof}
There is an evident global characterization of the quasi-new line in $\Pi$: for each $\mathfrak{b}\mid\mathfrak{rs}^{-1}$ fix $b\in\A^\times$ with $b\co_F=\mathfrak{b}$.  The set $\{ R_b\phi_\Pi \mid\mathfrak{b}\mathrm{\  divides\ } \mathfrak{rs}^{-1}\}$ is a basis for the space of $K_1(\mathfrak{r})$-fixed vectors in $\Pi$, and the quasi-new line is the orthogonal complement (in the $K_1(\mathfrak{r})$-fixed vectors) of the kernel of the linear functional $\Lambda$ defined by
$$
\Lambda(R_b\phi_\Pi)=\prod_{v\mid\mathfrak{dc}} \gamma_{\frac{1}{2},v}(b).
$$
In the weight $2$ case (\ref{rankin integral}) implies that the projection of $\overline{\Theta}_{\mathfrak{r}}$ to $\Pi$ is orthogonal to any form in the kernel of $\Lambda$, hence lies on the quasi-new line.  If $\epsilon(1/2,\mathfrak{r})=-1$ then $L(1/2,\Pi\times\Pi_\chi)=0$ and again (\ref{rankin integral}) shows that the projection of $\Phi_\mathfrak{r}$ to $\Pi$ lies on the quasi-new line.  In the weight $0$ case one uses (\ref{maass rankin}) in place of (\ref{rankin integral}).
\end{proof}


\section{CM cycles on quaternion algebras}
\label{quaternion generalities}


Let $B$ be a quaternion algebra over $F$ and assume that there is an embedding $E\map{} B$, which we fix once and for all.  Let $T$ and $G$ denote the algebraic groups over $F$  determined by
$$
T(A)=(E\otimes_F A)^\times\hspace{1cm}G(A)=(B\otimes_F A)^\times
$$
for any $F$-algebra $A$, and let $Z$ denote the center of $G$.  We denote by $\mathrm{N}$ both the norm $T\map{}Z$ and the reduced norm $G\map{}Z$.  Let $t\mapsto\overline{t}$ be the involution of $T(\A)$ induced by the nontrivial Galois automorphism of $E/F$.


\subsection{Preliminaries}
\label{ss:prelims}


Define $B^+=E$ and  $B^-=\{ b\in B\mid b t =\overline{t} b \ \forall t\in E \}.$
It follows from the Noether-Skolem theorem that $B^-$ is nontrivial, and from this one deduces that $B=B^+\oplus B^-$ with each summand free of rank one as a left $E$-module.  For any $\gamma\in G(F)$ the two invariants\begin{equation}\label{eta}
\eta = \frac{ \mathrm{N}(\gamma^+) }{ \mathrm{N}(\gamma) }
\hspace{1cm}
\xi = \frac{ \mathrm{N}(\gamma^-) }{\mathrm{N}(\gamma) }
\end{equation}
where $\gamma^\pm$ denote the projection of $\gamma$ to $B^\pm$, depend only on the double coset $T(F)\gamma T(F)$ and not on $\gamma$ itself.   A simple calculation shows that all elements of $B^-$ are trace-free and that $\mathrm{N}(\gamma)=\mathrm{N}(\gamma^+)+\mathrm{N}(\gamma^-)$.  For any place $v$ of $F$ let $B_v^\pm= B^\pm\otimes_F F_v$.   We say that $\gamma$ is \emph{degenerate} if $\{\eta,\zeta\} = \{0,1\}$ (i.e. if $\gamma\in B^+\cup B^-$), and  that $\gamma$ is \emph{nondegenerate} otherwise.  Of course we may make similar definitions for $\gamma\in G(F_v)$ for $v$ any place of $F$.

\begin{Lem}\label{eta parametrization}
The function $\gamma\mapsto (\eta,\xi)$ defines an injection $$T(F)\backslash G(F)/T(F)\map{}F\times F.$$  The image of this injection is the union of  $\{ (1,0), (0,1)\}$ and the set of pairs $(\eta,\xi)$ such that $\eta,\xi\not=0$, $\eta+\xi=1$, and for every place $v$ of $F$
\begin{eqnarray}\label{eta display}
\omega_v(-\eta\xi)= \left\{ \begin{array}{ll} 1 & \mathrm{if \ }B_v\mathrm{\ is\ split}
 \\   -1 &  \mathrm{otherwise.} \end{array}\right.
\end{eqnarray}
\end{Lem}

\begin{proof}
This is stated without proof in \cite[\S 4.1]{zhang2}.  We leave the injectivity as an easy exercise, and sketch a proof of the second claim.
Choose a generator $\epsilon$ for $B^-$ as a left $E$-module and write $E=F[\sqrt{\Delta}]$.  Then $B$ has as an $F$-basis $\{1, \sqrt\Delta, \epsilon,\sqrt{\Delta} \cdot \epsilon\}$, or, in the standard notation (as in \cite[Example A.2]{conrad}), $B\iso\left(   \frac{\Delta,-\mathrm{N}(\epsilon) }{F}  \right)$.
It follows that the right hand side of (\ref{eta display}) is equal to the Hilbert symbol 
$$
(\Delta,-\mathrm{N}(\epsilon))_v=\omega_v(-\mathrm{N}(\epsilon)).
$$
On the other hand it is easy to see that for any nondegenerate $\gamma\in G(F)$ we have $\omega_v(\eta\xi)=\omega_v(\mathrm{N}(\epsilon))$, so that $(\eta,\xi)$ satisfies (\ref{eta display}).  The condition $\eta+\xi=1$ is clear from the additivity of $\mathrm{N}$ with respect to the decomposition $B=B^+\oplus B^-$ noted earlier.  Conversely, given a pair $\eta,\xi\in F^\times$ satisfying (\ref{eta display}) and $\eta+\xi=1$ we must have $(\Delta,-\mathrm{N}(\epsilon))_v=(\Delta,-\eta\xi)_v$ 
for every place $v$.  It follows from the Hasse-Minkowski theorem that there are $x,y\in F$
such that 
$$
\xi \eta^{-1} \mathrm{N}(\epsilon)^{-1}=  x^2-y^2\Delta .
$$
Taking $\gamma=1+(x+y\sqrt{\Delta})\epsilon$ shows that $(\eta,\xi)$ arises from a nondegenerate $\gamma$.  Any degenerate $\gamma$ generates either $B^+$ or $B^-$ as a left $E$-module and so has image either $(1,0)$ or $(0,1)$, respectively.
\end{proof}

\begin{Lem}\label{global tau}
For any  nondegenerate $\gamma\in G(F)$ and any place $v$ of $F$ set 
$$
\tau_v(\gamma)= \omega_v(\delta) |\eta\xi|_v^{1/2} \chi_v(\eta)\overline{\chi}_v(\gamma^+) \epsilon_v(1/2,\omega,\psi_v^0).
$$
Then $\prod_v\tau_v(\gamma)=1$ where the product is over all places of $F$.  If $v$ is an archimedean place then $\tau_v(\gamma)=\omega_v(\delta)\cdot i\cdot |\eta\xi|_v^{1/2}$.
\end{Lem}

\begin{proof}
The functional equation (\ref{dirichlet functional}) and \cite[Corollary 4.4]{kudla-tate} imply $\epsilon(s,\omega)=|d\delta|^{s-1/2}$ while \cite[(3.29)]{kudla-tate} gives
$$
|\delta|_v^{s-1/2} \omega_v(\delta)\epsilon_v(s,\omega,\psi_v^0)=\epsilon_v(s,\omega,\psi_v).
$$
From this it is clear that $\prod_v\tau_v(\gamma)=1$.  If $v$ is archimedean then $\epsilon(s,\omega,\psi_v^0)=i$ by \cite[Proposition 3.8(iii)]{kudla-tate}. As $\chi_v$ is the trivial character, the final claim follows.
\end{proof}


\subsection{Heights of CM-cycles}
\label{ss:heights}


If $U\subset G(\A_f)$ is a compact open subgroup we define the set of \emph{CM points} of level $U$
$$
C_U= T(F)\backslash G(\A_f) /U.
$$
By a \emph{CM-cycle} of level $U$ we mean a compactly supported (i.e. finitely supported) function on $C_U$.   There is a unique left $T(\A_f)$-invariant measure on $C_U$ with the property that 
$$
\int_{G(\A_f)/U} f(g)\ dg = \int_{C_U} \sum_{t\in T(F)/(Z(F)\cap U)} f(tg)  \ dg
$$
for every locally constant compactly supported function $f$ on $G(\A_f)/U$, where the measure on $G(\A_f)/U$ gives every coset volume one.   The measure on $C_U$  assigns to each double coset $T(F)gU$  a volume equal to the inverse of
$$[T(F)\cap gUg^{-1} : Z(F)\cap U].$$  
Given compact open subgroups $U\subset V$ the measures on $C_U$ and $C_V$ are related by
\begin{equation}\label{measure descent}
\int_{C_V}\sum_{h\in V/U} f(gh)\ dg = \frac{\lambda_U}{\lambda_V} \int_{C_U} f(g)\ dg
\end{equation}
for any CM-cycle $f$ of level $U$, where $\lambda_U=[\co_F^\times:\co_F^\times\cap U]$ and similarly with $U$ replaced by $V$.

Given a $T(F)$ bi-invariant function $m$ on $G(F)$ define a function $k^m_U$ on $G(\A_f)\times G(\A_f)$ by
\begin{equation}\label{kernel def}
k^m_U(x,y) =\sum_{\gamma\in G(F)/(Z(F)\cap U)} \mathbf{1}_U(x^{-1}\gamma y) \cdot m(\gamma)
\end{equation}
where $\mathbf{1}_U$ is the characteristic function of $U$.  We will address the convergence of this sum as the need arises; for the moment assume that the sum converges absolutely for every $x,y$.  Note that $k^m_U$ descends to a function on $C_U\times C_U$.  If $P,Q$ are CM-cycles of level $U$ define the \emph{height pairing} in level $U$ with multiplicity $m$
\begin{equation}\label{pairing}
\langle P,Q \rangle^m_U =\int_{C_U\times C_U} P(x) \cdot k^m_U(x,y) \cdot \overline{Q(y)} \ dx \ dy.
\end{equation}
As in \cite[(4.1.9)]{zhang2} a simple calculation shows that there is a decomposition
\begin{equation}\label{linking decomposition}
\langle P,Q \rangle^m_U = \sum_{\gamma\in T(F)\backslash G(F)/T(F)} 
\langle P,Q \rangle_U^\gamma  \cdot m(\gamma)
\end{equation}
where for every $\gamma\in G(F)$
$$
\langle P ,Q \rangle_U^\gamma =
\int_{C_U} \sum_{\delta\in T(F)\backslash T(F)\gamma T(F)}
P(\delta y) \overline{Q (y) } \ dy
$$
is the  \emph{linking number} of $P$ and $Q$ at $\gamma$.

Abbreviate $U_Z=U\cap Z(\A_f)$ and $U_T=U\cap T(\A_f)$  and suppose now that $U$ is small enough that $\chi$ is trivial on $U_T$.  We will say that a CM-cycle $P$ of level $U$ is \emph{$\chi$-isotypic} if for all $t\in T(\A_f)$ and $g\in G(\A_f)$ we have $P(tg)=\chi(t)P(g).$

\begin{Lem}\label{orbital integrals}
Set $\chi^*(t)=\chi(\overline{t})$.  Suppose  $P$ and $Q$ are $\chi$-isotypic CM-cycles of level $U$ and that $Q$ is supported on the image of $T(\A_f)\map{}C_U$.   If $\gamma\in G(F)$ is degenerate then
$$
\langle P,   Q \rangle_U^\gamma = \overline{Q(1)}    \cdot 
\frac{ [T(\A_f):T(F)U_T]}{[T(F)\cap U:Z(F)\cap U]} 
\left\{\begin{array}{ll}
P(\gamma) &\mathrm{if\ } (\eta,\xi)=(1,0) \\
P(\gamma) &\mathrm{if\ } (\eta,\xi)=(0,1)\mathrm{\ and\ }\chi^*=\chi \\
0&\mathrm{if\ } (\eta,\xi)=(0,1)\mathrm{\ and\ }\chi^*\not=\chi.
\end{array}\right.
$$
If $\gamma$ is nondegenerate then
$$
\langle P,   Q \rangle_U^\gamma =
\overline{Q(1)}\cdot  [Z(\A_f):Z(F) U_Z]
\sum_{ t\in Z(\A_f) \backslash T(\A_f)/ U_T}  P ( t^{-1} \gamma t).
$$
\end{Lem}

\begin{proof}
First suppose that $\gamma$ is degenerate.  Then $\gamma$ normalizes $T(F)$ and so
\begin{eqnarray*}
 \langle P,Q\rangle_U^\gamma
&=&
\int_{C_U}  P(\gamma y) \overline{Q (y)} \ dy \\
&=&
\int_{T(F)\backslash T(\A_f)/U_T } P(y^{-1}\gamma y) \overline{Q(1)} \ dy.
\end{eqnarray*}
If $(\eta,\xi)=(1,0)$ then $\gamma\in T(F)$ leaving
$$
\langle P,Q\rangle_U^\gamma=
\mathrm{Vol}(T(F)\backslash T(\A_f)/ U_T) \cdot P(\gamma )\overline{Q(1)}.
$$
If $(\eta,\xi)=(0,1)$ then $\gamma y=\overline{y}\gamma$ for every $y\in T(\A_f)$, leaving
$$
\langle P,Q\rangle_U^\gamma
 = P(\gamma )\overline{Q(1)}
\cdot \int_{T(F)\backslash T(\A_f)/ U_T} \chi(y)^{-1} \chi^*(y)\ dy.
$$
In either case the first claim  follows. Now suppose that $\gamma$ is nondegenerate. The nondegeneracy of $\gamma$ implies that  $\gamma^{-1} T(F)\gamma\cap T(F)=Z(F)$ and so  
\begin{eqnarray*}
 \langle P,Q\rangle_U^\gamma
&=&
\int_{T(F)\backslash T(\A_f)/U_T} \sum_{\delta\in T(F)\backslash T(F)\gamma T(F)}
P( y^{-1}\delta y)\overline{Q(1)} \ dy \\
&=&
\int_{T(F)\backslash T(\A_f)/U_T} \sum_{t \in  T(F)/Z(F)}
P( y^{-1}\gamma t y) \overline{Q(1)} \ dy \\
&=&
\overline{Q(1)}  \int_{Z(F)\backslash T(\A_f)/U_T}  P( y^{-1}\gamma  y) \ dy
\end{eqnarray*}
where the measure on $Z(F)\backslash T(\A_f)/U_T$ gives each coset volume $1$.  The second claim follows.
\end{proof}

In particular, if the $U=\prod_v U_v$ and $P=\prod_v P_v$  of  Lemma \ref{orbital integrals} are factorizable  and $\gamma$ is nondegenerate  then there is a decomposition
\begin{equation}\label{orbital decomp}
\langle P ,Q \rangle_U^\gamma
=  
\overline{Q(1) }\cdot [Z(\A_f):Z(F) U_Z] \cdot \prod_v O_U^\gamma( P_v ) 
\end{equation}
where the product is over all finite places of $F$ and
\begin{equation}\label{local link}
O_U^\gamma( P_v )  =
 \sum_{ t\in F_v^\times \backslash E_v^\times/ U_{T,v} }  
 P_v ( t^{-1} \gamma t)
\end{equation}
is the \emph{orbital integral} of $P_v$ at $\gamma$, where we abbreviate $U_{T,v}=E_v^\times\cap U_v$.

The remainder of \S \ref{quaternion generalities} is devoted to the computations of orbital integrals for specific CM-cycles, and we fix the following data throughout \S \ref{ss:unramified local calculations} and  \S \ref{ss:ramified local calculations}.  Let $v$ be a finite place of $F$ and fix $\epsilon_v\in B_v^\times$ such that $E_v\epsilon_v=B_v^-$.  We  assume that $\mathrm{N}(\epsilon_v)\in \co_{F,v}$ and let $\mathfrak{e}$ be an ideal of $\co_F$ satisfying $\mathfrak{e}_v=\mathrm{N}(\epsilon_v)\co_{F,v}$.  Define an order of $B_v$ by
$$
R_v=\co_{E,v}+\co_{E,v}\epsilon_v.
$$
Fix a uniformizing parameter $\varpi\in F_v$.


\subsection{Local calculations at primes away from $\mathrm{N}(\mathfrak{C})$}
\label{ss:unramified local calculations}


Assume that $v\nmid\mathrm{N}(\mathfrak{C})$ and set $U_v=R_v^\times$.  Define a function on $G(F_v)/U_v$ by
$$
P_{\chi,v}(g)= \sum_{t\in E_v^\times/\co_{E,v}^\times } \chi_v(t)\mathbf{1}_{U_v}(t^{-1}g).
$$
 For each ideal $\mathfrak{a}\subset\co_F$ set
$
H(\mathfrak{a}_v) = \{ h\in R_v \mid  \mathrm{N}(h)\co_{F,v} = \mathfrak{a}_v \}  
$
and define another function on $G(F_v)/U_v$ 
\begin{eqnarray*}
P_{\chi,\mathfrak{a},v}(g)
&=&
\sum_{h\in H(\mathfrak{a}_v)/U_v} P_{\chi,v}(gh)\\
&=&
\chi_v(\mathfrak{a}) \sum_{t\in E_v^\times/\co_{E,v}^\times } \chi_v(t)\mathbf{1}_{H(\mathfrak{a}_v)} (t^{-1}g).
\end{eqnarray*}
For each nondegenerate $\gamma\in G(F_v)$ we wish to compute the orbital integral
\begin{equation}
\label{unr local link}
  O^\gamma_U(P_{\chi,\mathfrak{a},v })=\sum_{ t\in F_v^\times \backslash E_v^\times / \co_{E,v}^\times}  
 P_{\chi,\mathfrak{a},v} ( t^{-1} \gamma t).
\end{equation}

\begin{Prop}
\label{Prop:unramified link I}
Suppose $v$  is inert in $E$ and $\gamma\in G(F_v)$ is nondegenerate. Then (\ref{unr local link}) is nonzero if and only if $\ord_v(\eta \mathfrak{a})$ and $\ord_v(\xi  \mathfrak{ae}^{-1})$ are both even and nonnegative.  When this is the case
$$
  O^\gamma_U(P_{\chi,\mathfrak{a},v })=\overline{\chi}_v(\eta)  \chi_v(\gamma^+)  \chi_v(\varpi)^{\frac{\ord_v(\eta \mathfrak{a})}{2}}.
$$
\end{Prop}

\begin{proof}
Suppose $\gamma^+=1$, so that $\gamma=1+\beta \epsilon_v$ with $\beta \in E_v^\times$. The expression (\ref{unr local link}) reduces to
\begin{eqnarray*}
 O^\gamma_U(P_{\chi,\mathfrak{a},v }) 
 &=& 
 P_{\chi,\mathfrak{a},v}   (\gamma)  \\
 &=&
\chi_v(\mathfrak{a}) \sum_{k=-\infty}^\infty \chi_v(\varpi)^k  \mathbf{1}_{H(\mathfrak{a}_v) }( \varpi^{-k}\gamma)\\
\end{eqnarray*}
Using $\ord_v(\eta)=-\ord_v(\mathrm{N}(\gamma))$ we see that the only possible contribution to the inner sum is for $k$ satisfying $2k= -\ord_v(\eta \mathfrak{a})$.  Thus we may assume  that $\ord_v(\eta \mathfrak{a})$ is even, leaving
\begin{eqnarray*}
  O^\gamma_U( P_{\chi,\mathfrak{a},v }) 
& = &
\chi_v(\mathfrak{a}) \chi_v(\varpi)^{-\frac{1}{2}\ord_v(\eta \mathfrak{a})}
 \mathbf{1}_{ H(\mathfrak{a}_v) }(\varpi^{\frac{1}{2}\ord_v(\eta \mathfrak{a})}\gamma )  \\
 & = &
\overline{\chi(\eta)} \chi_v(\varpi)^{\frac{1}{2}\ord_v(\eta \mathfrak{a})}
 \mathbf{1}_{ R_v }(\varpi^{\frac{1}{2}\ord_v(\eta \mathfrak{a})}\gamma ) 
\end{eqnarray*}
which is nonzero if and only if 
$$
\varpi^{\frac{1}{2}\ord_v(\eta \mathfrak{a})}(1+\beta\epsilon_v)\in \co_{E,v}+\co_{E,v}\epsilon_v.
$$  
Thus $  O^\gamma_U(P_{\chi,\mathfrak{a},v} )   $ is nonzero if and only if both
$$
\ord_v(\eta \mathfrak{a})\ge 0  \hspace{1cm}
\ord_v(\eta \mathfrak{a})\ge - \ord_v(\mathrm{N}(\beta))
$$  
hold.  The observation that 
$$
\ord_v(\xi \mathfrak{ae}^{-1})=\ord_v(\mathfrak{a})+\ord_v(\mathrm{N}(\beta))-\ord_v(\mathrm{N}(\gamma))
=\ord_v(\eta \mathfrak{a})+\ord_v(\mathrm{N}(\beta)),
$$
together with $\ord_v(\mathrm{N}(\beta))\in 2\Z$ completes the proof when $\gamma^+=1$.  For the general case simply note that if $\gamma$ is replaced by $t \gamma$ with $t \in E_v^\times$ then both sides of the stated equality are multiplied by $\chi_v(t)$.  Thus it suffices to prove the claim for a single element of $E_v^\times\gamma$.
\end{proof}

\begin{Rem}\label{gamma plus remark}
In the proof of Proposition \ref{Prop:unramified link I} it sufficed to treat the case $\gamma^+=1$.  This will remain true in all remaining computations of orbital integrals in \S \ref{ss:unramified local calculations} and \S \ref{ss:ramified local calculations}.  We will continue to state the results for arbitrary $\gamma$, but in the proofs we will assume that $\gamma^+=1$.
\end{Rem}

\begin{Prop}
\label{Prop:unr ram link}
Suppose $v$  is ramified in $E$ and $\gamma\in G(F_v)$ is nondegenerate. Then (\ref{unr local link}) is nonzero if and only if $\ord_v(\eta \mathfrak{a})$ and $\ord_v(\xi  \mathfrak{ae}^{-1})$ are both  nonnegative.  When this is the case
$$
  O^\gamma_U(P_{\chi,\mathfrak{a},v} ) 
 =2 \cdot \overline{\chi_v(\eta)}  \chi_v(\gamma^+)  \chi_v(\varpi_E)^{\ord_v(\eta\mathfrak{a}) }
$$
for any uniformizer $\varpi_E\in E_v$.
\end{Prop}

\begin{proof}
Write $\gamma=1+\beta\epsilon_v$ with $\beta\in E_v^\times$.  Equation (\ref{unr local link}) reduces to
\begin{eqnarray*}
O^\gamma_U(P_{\chi,\mathfrak{a},v} ) 
&=&
P_{\chi,\mathfrak{a},v}(\gamma)+P_{\chi,\mathfrak{a},v}(\varpi_E^{-1}\gamma\varpi_E)\\
&=&
\chi_v(\mathfrak{a})\sum_{k=-\infty}^\infty  \chi_v(\varpi_E)^{-k} \big[ \mathbf{1}_{H(\mathfrak{a}_v)}(\varpi_E^{k}\gamma)+\mathbf{1}_{H(\mathfrak{a}_v)}(\varpi_E^{k-1}\gamma\varpi_E) \big].
\end{eqnarray*}
The only possible contribution to the final sum is the term $k=\ord_v(\eta\mathfrak{a})$, leaving
\begin{eqnarray*}\lefteqn{
O^\gamma_U(P_{\chi,\mathfrak{a},v} )   } \\
&=&
\chi_v(\mathfrak{a})  \chi_v(\varpi_E)^{-\ord_v(\eta\mathfrak{a})} \big[ \mathbf{1}_{R_v  }(\varpi_E^{\ord_v(\eta\mathfrak{a})}\gamma)+\mathbf{1}_{H(\mathfrak{a}_v)}(\varpi_E^{\ord_v(\eta\mathfrak{a})-1}\gamma\varpi_E) \big].
\end{eqnarray*}
The remainder of the proof is exactly as the proof of Proposition \ref{Prop:unramified link I}.
\end{proof}

\begin{Prop}
\label{Prop:unr split link}
Suppose $v$ is split in $E$ and $\gamma\in G(F_v)$ is nondegenerate.  Then (\ref{unr local link})  is nonzero if and only if $\ord(\eta \mathfrak{a})$ and $\ord_v(\xi \mathfrak{ae}^{-1} )$ are both nonnegative. When this is the case 
$$
  O^\gamma_U(P_{\chi,\mathfrak{a},v} ) 
 =\overline{\chi_v(\eta)}  \chi_v(\gamma^+) \cdot (1+\ord_v(\xi \mathfrak{ae}^{-1}))
\sum_{  \substack{i+j=\ord_v(\eta \mathfrak{a}) \\ i,j\ge 0} }  \alpha^i\beta^j
$$
where, under the identification $E_v^\times\iso F_v^\times\times F_v^\times$, 
$$
\alpha=\chi_v(\varpi,1)  \hspace{1cm}  \beta=\chi_v(1,\varpi).
$$
\end{Prop}

\begin{proof}
Write $\gamma=1+\beta\epsilon_v$ with $\beta\in E_v^\times$, so that
$$
\ord_v(\eta)=-\ord_v(\mathrm{N}(\gamma))
\mathrm{\ \ \ and\ \ \ }
\ord_v(\xi \mathfrak{e}^{-1}) = \ord_v(\eta) + \ord_v(\mathrm{N}(\beta)).
$$
For any $t\in T(F_v)$
$$
P_{\chi,\mathfrak{a},v} (t^{-1}\gamma t) = \chi_v(\mathfrak{a})
\sum_{s \in E_v^\times/\co_{E,v}^\times} \overline{\chi}_v(s) \cdot \mathbf{1}_{ H(\mathfrak{a}_v)}( s t^{-1}\gamma t),
$$
and the only terms in the final sum which may contribute are from those $s$ satisfying $\ord_v(\mathrm{N}(s))=\ord_v(\eta a)$.  Fix an isomorphism $\co_{E,v}\iso\co_{F,v}\times\co_{F,v}$ and set $e_{i,j}=(\varpi^i,\varpi^j)$.  Then
\begin{eqnarray}\label{split unramified link}
P_{\chi,\mathfrak{a},v} (t^{-1}\gamma t)= 
\chi_v(\mathfrak{a}) \sum_{i+j=\ord_v(\eta \mathfrak{a})} \alpha^{-i}\beta^{-j}
\mathbf{1}_{R_v }( e_{i,j} t^{-1}\gamma t).
\end{eqnarray}
The set $\{e_{k,0}\mid k\in\Z\}$ is a complete set of coset representatives for $F_v^\times\backslash E_v^\times/ \co_{E,v}^\times$, and 
$$
e_{k,0}^{-1} \cdot \gamma \cdot  e_{k,0}
=
e_{k,0}^{-1}(1+\beta\epsilon_v) e_{k,0}
=
1 + e_{-k,k}\beta\epsilon_v.
$$
Combining (\ref{unr local link}) and (\ref{split unramified link}) therefore gives
\begin{equation}\label{split unramified link II}
  O^\gamma_U(P_{\chi,\mathfrak{a},v} ) 
 = 
\chi_v(\mathfrak{a}) \sum_{i+j=\ord_v(\eta \mathfrak{a})}  \alpha^{-i}\beta^{-j}
\sum_{k=-\infty}^\infty
\mathbf{1}_{ R_v  }( e_{i,j} (1+e_{-k,k}\beta\epsilon_v) ).
\end{equation}
The inner sum  counts the number of $k$ such that 
$$
e_{i,j} +e_{i-k,j+k}\beta\epsilon_v \in \co_{E,v}+\co_{E,v}\epsilon_v.
$$
When $i+j=\ord_v(\eta\mathfrak{a})$ the condition $ e_{i,j}\in \co_{E,v}$ is equivalent to $0\le i,j \le \ord_v(\eta a)$, and so the outer sum may be restricted to $i,j\ge 0$.  The inner sum then counts the number of $k$ such that 
$$
e_{i-k, j+k}\beta \in \co_{E,v}.
$$
Replacing $\beta$ by an $\co_{E,v}^\times$-multiple does not change the number of such $k$, and so we may assume that $\beta=e_{s,t}$ for some $s,t\in\Z$.   The inner sum of (\ref{split unramified link II}) is then equal  to
\begin{eqnarray*}
\# \{k\in \Z \mid i-k+s\ge 0,\ j+k+t\ge 0 \}  &=& i+j+s+t+1 \\
&=& \ord_v(\eta \mathfrak{a}) + \ord_v(\mathrm{N}(\beta))+1 \\
&=& \ord_v(\xi \mathfrak{ae}^{-1})+1
\end{eqnarray*}
if $\ord_v(\xi  \mathfrak{ae}^{-1})\ge 0$, and is equal to $0$ otherwise. Thus (\ref{split unramified link II}) reduces to
\begin{equation*}
O^\gamma_U(P_{\chi,\mathfrak{a},v} ) 
 = 
\chi_v(\mathfrak{a}) (\ord_v(\xi \mathfrak{ae}^{-1})+1) \sum_{ \substack{i+j=\ord_v(\eta \mathfrak{a}) \\ i,j\ge 0}}  \alpha^{-i}\beta^{-j}
\end{equation*}
when $\ord_v(\xi \mathfrak{ae}^{-1} )\ge 0$.  Using $\chi_v(\eta \mathfrak{a})=\alpha^{i+j}\beta^{i+j}$, the proposition follows.
\end{proof}

\begin{Cor}\label{unramified links}
Suppose $v\nmid \mathrm{N}(\mathfrak{C})$,  $\gamma\in G(F_v)$ is nondegenerate, and $\mathfrak{r}$ is an ideal of $\co_F$ with $\mathfrak{r}_v=\mathfrak{e}_v$.   Then
$$
|a|_v |d|^{1/2}_v\tau_v(\gamma)
\cdot  O_U^\gamma(P_{\chi,\mathfrak{a},v} ) = 
B_v(a,\eta,\xi;\Theta_{\mathfrak{r}}) 
$$
where $\tau_v(\gamma)$ is as in Lemma \ref{global tau}.
\end{Cor}

\begin{proof}
Propositions \ref{Prop:unramified link I}, \ref{Prop:unr ram link}, and \ref{Prop:unr split link} give explicit formulas for the left hand side, while Proposition \ref{kernel coefficients} gives  explicit formulas for the right hand side.
\end{proof}

We now turn to the calculation of  $P_{\chi,\mathfrak{a},v}(1)$ and $P_{\chi,\mathfrak{a},v}(\epsilon_v)$.

\begin{Lem}\label{unr degenerate}
Suppose that $v$ is inert in $E$.  Then
\begin{eqnarray*}
P_{\chi,\mathfrak{a},v}(1) &=&  \left\{\begin{array}{ll}
\chi_v(\varpi)^{\frac{1}{2}\ord_v(\mathfrak{a})} &\mathrm{if\ }\ord_v(\mathfrak{a})\mathrm{\ is\ even\ and\ nonnegative}\\
0&\mathrm{otherwise}
\end{array}\right.
\\
P_{\chi,\mathfrak{a},v}(\epsilon_v) &=& \left\{\begin{array}{ll}
\chi_v(\mathfrak{e}) \chi_v(\varpi)^{\frac{1}{2}\ord_v(\mathfrak{ae}^{-1})} &\mathrm{if\ }\ord_v(\mathfrak{ae}^{-1})\mathrm{\ is\ even\ and\ nonnegative}\\
0&\mathrm{otherwise.}
\end{array}\right.
\end{eqnarray*}
\end{Lem}

\begin{proof}
Exactly as in Proposition \ref{Prop:unramified link I}
$$
P_{\chi,\mathfrak{a},v}(g)=
\chi_v(\mathfrak{a}) \sum_{k=-\infty}^\infty \chi_v(\varpi)^{-k} \mathbf{1}_{H(\mathfrak{a}_v)}(\varpi^{k} g).
$$
If $g=1$ then  $\ord_v(\mathrm{N}(g))=0$ and the only contribution to the final sum is when $2k=\ord_v(\mathfrak{a})$.  Thus we may assume that $\ord_v(\mathfrak{a})$ is even, leaving
$$
P_{\chi,\mathfrak{a},v}(1)  = 
 \chi_v(\varpi)^{\frac{1}{2}\ord_v(a) } \mathbf{1}_{R_v}(\varpi^{\frac{1}{2}\ord_v(a) })
$$
which proves the first claim.  If $g=\epsilon_v$  then $\ord_v(\mathrm{N}(g))= \ord_v(\mathfrak{e})$ and the only contribution to the above sum is for $k$ satisfying  $2k+\ord_v(\mathfrak{e})=\ord_v(\mathfrak{a})$.  Thus we may assume $\ord_v(\mathfrak{ae}^{-1})$ is even, leaving
$$
P_{\chi,\mathfrak{a},v} (\epsilon_v)=\chi_v(\mathfrak{a})
\chi_v(\varpi)^{-\frac{1}{2}\ord_v(\mathfrak{ae}^{-1}) } \mathbf{1}_{R_v}(\varpi^{\frac{1}{2}\ord_v(\mathfrak{ae}^{-1}) } \epsilon_v)
$$
which proves the second claim.
\end{proof}

\begin{Lem}\label{unr degenerate II}
Suppose that $v$ is ramified in $E$, and let $\varpi_E$ be a uniformizer of $E_v$.  Then
\begin{eqnarray*}
P_{\chi,\mathfrak{a},v}(1) &=&  \left\{\begin{array}{ll}
\chi_v(\varpi_E)^{\ord_v(\mathfrak{a})} &\mathrm{if\ }\ord_v(\mathfrak{a})\ge 0\\
0&\mathrm{otherwise}
\end{array}\right.
\\
P_{\chi,\mathfrak{a},v}(\epsilon_v) &=& \left\{\begin{array}{ll}
\chi_v(\mathfrak{e}) \chi_v(\varpi_E)^{\ord_v(\mathfrak{ae}^{-1})} &\mathrm{if\ }\ord_v(\mathfrak{ae}^{-1})\ge 0 \\
0&\mathrm{otherwise.}
\end{array}\right.
\end{eqnarray*}
\end{Lem}

\begin{proof}
The proof is nearly identical to that of Lemma \ref{unr degenerate}, and the details are omitted.
\end{proof}

\begin{Lem}\label{unr degenerate III}
Suppose that $v$  is split in $E$, and let $\alpha$ and $\beta$ be as in Proposition \ref{Prop:unr split link}. Then
\begin{eqnarray*}
 P_{\chi,\mathfrak{a},v}(1)  &=& \sum_{\substack{i+j=\ord_v(\mathfrak{a}) \\ i,j\ge 0}} \alpha^i\beta^j  \\
 P_{\chi,\mathfrak{a},v}(\epsilon_v)
&=& \chi_v(\mathfrak{e}) \sum_{\substack{i+j=\ord_v(\mathfrak{ae}^{-1}) \\ i,j\ge 0}} \alpha^i\beta^j.
\end{eqnarray*}
\end{Lem}

\begin{proof}
On the right hand side of
$$
P_{\chi,\mathfrak{a},v}(g ) =
\chi_v(\mathfrak{a}) \sum_{t \in E_v^\times/\co_{E,v}^\times} \overline{\chi}_v(t) \mathbf{1}_{H(\mathfrak{a}_v)}( t g )
$$
the only terms which may contribute are from those $t$ satisfying $$\ord_v(\mathrm{N}(t))=\ord_v(\mathfrak{a})-\ord_v(\mathrm{N}(g)).$$  Fix an isomorphism $\co_{E,v}\iso\co_{F,v}\times\co_{F,v}$ and set $e_{i,j}=(\varpi^i,\varpi^j)$.  Then
\begin{equation*}
P_{\chi,\mathfrak{a},v}(g )= \chi_v(\mathfrak{a})
\sum_{i+j=\ord_v(a)-\ord_v(\mathrm{N}(g))} \alpha^{-i} \beta^{-j}
\mathbf{1}_{ R_v }(e_{i,j} g ).
\end{equation*}
The lemma follows easily from this equality, using $\alpha\beta=\chi_v(\varpi)$.
\end{proof}

\begin{Cor}\label{unr final degen}
Suppose $v$ does not divide $\mathrm{N}(\mathfrak{C})$ and that $a\in\A^\times$ satisfies $a\co_F=\mathfrak{a}$.  Then
$$
P_{\chi,\mathfrak{a},v }(1) = |a|_v^{-1/2} B_v(a;\theta).
$$
 If we pick $e\in\A^\times$ such that $e\co_F=\mathfrak{e}$ then
$$
P_{\chi,\mathfrak{a},v }(\epsilon_v) =  \chi_v(e) |e|_v^{1/2} |a|_v^{-1/2} B_v(ae^{-1};\theta)
$$
\end{Cor}

\begin{proof}
Compare Lemmas \ref{unr degenerate}, \ref{unr degenerate II}, and \ref{unr degenerate III}  with Proposition \ref{theta coefficients}.
\end{proof}


\subsection{Local calculations at primes dividing $\mathrm{N}(\mathfrak{C})$}
\label{ss:ramified local calculations}


Let $v$ be a finite place of $F$ dividing $\mathrm{N}(\mathfrak{C})$ (in particular $v\nmid\mathfrak{d}$).  Assume that 
\begin{equation}
\label{a norm inequality}
\ord_v(\mathrm{N}(\mathfrak{C}))\le\ord_v(\mathfrak{e})
\end{equation}
and let $U_v\subset R_v^\times$ be the kernel of the homomorphism $R_v^\times\map{}(\co_{E,v}/\mathfrak{C}_v)^\times$ given by $x+y\epsilon_v\mapsto x$. Define a function $P_{\chi,v}$ on $G(F_v)$ by
$$
P_{\chi,v}(g)=\sum_{t\in E_v^\times/ U_{T,v}} \chi_v(t)\mathbf{1}_{U_v}(t^{-1}g).
$$
For each nondegenerate $\gamma\in G(F_v)$ we wish to compute the orbital integral
\begin{equation}
\label{ram local link}
 O^\gamma_U(P_{\chi,v} ) =
 \sum_{ t\in F_v^\times \backslash E_v^\times /  U_{T,v} }  
 P_{\chi,v} ( t^{-1} \gamma t).
\end{equation}
In accordance with Remark \ref{gamma plus remark} we will state the results for any nondegenerate $\gamma$ but will assume in the proofs that $\gamma^+=1$ and write $\gamma=1+\beta \epsilon_v$ with $\beta\in E_v^\times$.

\begin{Prop}\label{ram link I}
Suppose $v$ is inert in $E$ and $\gamma\in G(F_v)$ is nondegenerate.  Then  (\ref{ram local link}) is nonzero if and only if $\ord_v(\eta)=0$ and $\ord_v(\xi \mathfrak{e}^{-1})$ is even and nonnegative.  When this is the case 
 $$ 
 O^\gamma_U(  P_{\chi,v} ) = [\co_{E,v}^\times:\co_{F,v}^\times  U_{T,v}]\cdot \chi_v(\gamma^+).
 $$
\end{Prop}

\begin{proof}
In this case (\ref{ram local link}) gives
\begin{eqnarray*}
O^\gamma_U( P_{\chi,v} )  &=&
 \sum_{ t\in \co_{F,v}^\times \backslash \co_{E,v}^\times /  U_{T,v}}  P_{\chi,v}(t^{-1}\gamma t) \\
&=&
 \sum_{ t\in \co_{F,v}^\times \backslash \co_{E,v}^\times / U_{T,v}}  
   \sum_{s\in E_v^\times/ U_{T,v}}  \chi_v(s) \mathbf{1}_{ U_v} ( s^{-1}t^{-1} \gamma t). 
\end{eqnarray*}
As $U_v=U_{T,v}+\co_{E,v}\epsilon_v$,
the only way that $s^{-1}t^{-1} \gamma t=s^{-1}(1+t^{-1}\overline{t}\beta\epsilon_v)$ can lie in $U_v$ is if $s\in U_{T,v}$.  Therefore only the term $s=1$ contributes to the inner sum, leaving
$$
O^\gamma_U( P_{\chi,v} )  =  \sum_{ t\in \co_{F,v}^\times \backslash \co_{E,v}^\times / U_{T,v} }  
 \mathbf{1}_{ U_v} (1+ t^{-1} \overline{t}\beta\epsilon_v).
$$
If $\ord_v(\mathrm{N}(\beta))\ge 0$ then every term in the sum is $1$, and otherwise every term is $0$.
As
$$
\ord_v(\xi\mathfrak{e}^{-1}) = \ord_v(\eta) + \ord_v(\mathrm{N}(\beta))
$$
the condition $\ord_v(\mathrm{N}(\beta))\ge 0$ is equivalent to $\ord_v(\xi\mathfrak{e}^{-1}) \ge \ord_v(\eta)$, and using $\eta+\xi=1$ and $\ord_v(\mathfrak{e})>0$
$$
\ord_v(\xi\mathfrak{e}^{-1}) \ge \ord_v(\eta) \iff \ord_v(\eta)=0\mathrm{\ and\ }\ord_v(\xi\mathfrak{e}^{-1})\ge 0.
$$
\end{proof}

\begin{Prop}\label{ram link II}
Suppose $v$ is split  in $E$ and $\gamma\in G(F_v)$ is nondegenerate.  Then (\ref{ram local link}) is nonzero if and only if $\ord_v(\eta)=0$ and $\ord_v(\xi \mathfrak{e}^{-1})\ge 0$.  When this is the case
$$ 
O^\gamma_U(  P_{\chi,v} )=
[\co_{E,v}^\times:\co_{F,v}^\times U_{T,v}]  \cdot  \chi_v(\gamma^+)(1+ \ord_v(\xi \mathfrak{e}^{-1})).
$$
\end{Prop}

\begin{proof}
Using the notation of Proposition \ref{Prop:unr split link}, so that $e_{i,j}=(\varpi^i,\varpi^j)$,
for any $t\in T(F_v)$ we have
\begin{eqnarray*}
 P_{\chi,v} (t^{-1}\gamma t) &=& \sum_{s\in E_v^\times/  U_{T,v}  } \overline{\chi}_v(s)  \cdot \mathbf{1}_{U_v} (st^{-1}\gamma t) \\
 &=&  \sum_{i,j\in\Z} \sum_{s\in \co_{E,v}^\times/  U_{T,v}}  \overline{\chi}_v(se_{i,j}) \cdot  \mathbf{1}_{U_v} (se_{i,j}  (1+ t^{-1}\overline{t} \beta\epsilon_v ) )
\end{eqnarray*}
As $U_v= U_{T,v}+\co_{E,v}\epsilon_v$, only terms for which $se_{i,j}\in  U_{T,v}$ can contribute to the inner sum, and so the only nonzero term can be the one with $i=j=0$ and $s\in  U_{T,v}$.  This leaves
$$
 P_\chi (t^{-1}\gamma t)=  \mathbf{1}_{U_v} (1+ t^{-1}\overline{t} \beta\epsilon_v )
$$
and so (\ref{ram local link}) becomes
\begin{eqnarray*}
O^\gamma_U(  P_{\chi,v} ) &=& \sum_{t\in F_v^\times\backslash E_v^\times/ U_{T,v}}  \mathbf{1}_{U_v} (1+t^{-1}\overline{t} \beta\epsilon_v )\\
&=& \sum_{k=-\infty}^\infty \sum_{t\in \co_{F,v}^\times\backslash \co_{E,v}^\times/ U_{T,v} }  \mathbf{1}_{U_v} (1+ e_{-k,k} t^{-1}\overline{t} \beta\epsilon_v )\\
&=&   [\co_{E,v}^\times:  \co_{F,v}^\times U_{T,v}  ] \cdot  \sum_{k=-\infty}^\infty   \mathbf{1}_{U_v} (1+ e_{-k,k}  \beta\epsilon_v ).
\end{eqnarray*}
Every term in the final sum is $0$ unless the quantity 
$$
\mathrm{N}(1+e_{-k,k}\beta\epsilon_v)=\mathrm{N}(\gamma)=\eta^{-1}
$$
lies in $\co_F^\times$.  Thus we may assume $\ord_v(\eta)=0$, so that the sum simply counts the number of $k$ for which  $e_{-k,k}\beta\in \co_{E,v}.$
Multiplying $\beta$ by an element of $\co_{E,v}^\times$, we may assume that $\beta=e_{s,t}$ for some $s,t\in\Z$.  The $k$ for which $e_{-k,k}\beta\in \co_{E,v}$ holds are then precisely those for which $s-k\ge 0$ and $t+k\ge 0$, and there are
$$
s+t+1=\ord_v(\mathrm{N}(\beta))+1=\ord_v(\xi\mathfrak{e}^{-1})+1
$$
such $k$ if $\ord_v(\xi\mathfrak{e}^{-1})\ge 0$, and no such $k$ otherwise.
\end{proof}

\begin{Cor}\label{ramified links}
Suppose $v$ divides $\mathrm{N}(\mathfrak{C})$ and  $\gamma$ is nondegenerate. Then
$$
\tau_v(\gamma)\cdot O_U^\gamma (P_{\chi,v} ) =  [\co_{E,v}^\times:\co_{F,v}^\times U_{T,v} ]
\cdot 
B_v(1,\eta,\xi;\Theta_{\mathfrak{r}})
$$
where $\tau_v(\gamma)$ is as in Lemma \ref{global tau} and $\mathfrak{r}_v=\mathfrak{e}_v$.
\end{Cor}

\begin{proof}
Propositions \ref{ram link I} and \ref{ram link II} give explicit formulas for the left hand side while Proposition \ref{kernel coefficients} gives explicit formulas for the right hand side.
\end{proof}

\begin{Lem}\label{ramified degenerate link}
We have the equalities
$P_{\chi,v}(1)=1$ and $P_{\chi,v}(\epsilon_v)=0$.
\end{Lem}

\begin{proof}
Clearly $P_{\chi,v}(1)=1$ simply by definition of $P_{\chi,v}$.   On the other hand
$$
P_{\chi,v}(\epsilon_v) =  \sum_{ t \in T(F_v)/U_{T,v}} \chi_v(t^{-1}) \mathbf{1}_{U_v }(t \epsilon_v).
$$
If this sum is nonzero then $t \epsilon_v\in R_v^\times$ for some $t \in T(F_v)$.  But this would imply both $\mathrm{N}(t\epsilon_v)\in\co_{F,v}^\times$ and $t\epsilon_v  \in \co_{E,v}\epsilon_v$, which implies 
$\ord_v(\mathrm{N}(\epsilon_v))\le 0.$  But $\ord_v(\mathrm{N}(\epsilon_v))=\ord_v(\mathfrak{e})>0$ by (\ref{a norm inequality}), a contradiction.
\end{proof}

\begin{Cor}\label{ram final degen}
Choose $e\in \A^\times$ with $e\co_F=\mathfrak{e}$.  Then
$$
P_{\chi,v}(1)=B_v(1;\theta)
\hspace{1cm}
P_{\chi,v}(\epsilon_v)= \chi_v(e)|e|_v^{1/2} B_v(e^{-1};\theta).
$$
\end{Cor}

\begin{proof}
Compare Lemma \ref{ramified degenerate link} with Proposition \ref{theta coefficients}.
\end{proof}


\section{Central values}
\label{s:central value}


Suppose the representation $\Pi$ of \S \ref{notations} satisfies  Hypothesis \ref{hyp}.  Recall that $\Pi$ has conductor   $\mathfrak{n}=\mathfrak{ms}$ and that $\chi$ has conductor $\mathfrak{C}=\mathfrak{c}\co_E$ for some $\co_F$-ideal $\mathfrak{c}$.  Let $B$ be a quaternion algebra over $F$ satisfying
 \begin{equation}\label{B ramification}
B_v \mathrm{\ is\ split\ }\iff \epsilon_v(1/2,\mathfrak{r},\psi)=1
 \end{equation}
for every \emph{finite} place $v$ of $F$, where $\mathfrak{r}=\mathfrak{mc}^2$ and the local epsilon factor is defined by (\ref{epsilon}).  This implies that the reduced discriminant of $B$ divides $\mathfrak{m}$ and, as $E_v$ is a field whenever $B_v$ is nonsplit,  that there is  an embedding $E\map{}B$ which we fix.  For the moment we do not specify the behavior of $B$ at archimedean places. Let $G$ and $T$ be the algebraic groups over $F$ defined at the beginning of \S \ref{quaternion generalities}.  For any ideal $\mathfrak{b}\subset \co_F$ let $\co_{\mathfrak{b}}=\co_F+\mathfrak{b}\co_{E}$ denote the order of $\co_E$ of conductor $\mathfrak{b}$.


\subsection{Special CM cycles}
\label{ss:special cycles}


 We construct two particular compact open subgroups $U\subset V$ of $G(\A_f)$ and two special CM-cycles $Q_\chi$ and $P_\chi$ of level $V$ and $U$, respectively.  It is ultimately the cycle $Q_\chi$ in which we are interested, but the local orbital integrals (\ref{local link}) of cycles of level $V$ seem too difficult to compute directly.  The subgroup $U$ is chosen to make these orbital integrals more readily computable (indeed, they have already been computed in \S \ref{ss:unramified local calculations} and \ref{ss:unramified local calculations}).
 
 \begin{Lem}\label{choosing the order}
 For every finite place $v$ there is an order in $B_v$ of reduced discriminant $\mathfrak{m}_v$ which contains $\co_{E,v}$.  Such an order is unique up to $E_v^\times$-conjugacy.
  \end{Lem}
 
 \begin{proof}
 If $v$ is inert in $E$ then (\ref{B ramification}) implies that $$\ord_v(\mathfrak{m})\equiv \ord_v(\mathrm{disc}(B_v))\pmod{2}$$ where $\mathrm{disc}(B_v)$ is the reduced discriminant of $B_v$.  Thus the lemma follows from \cite[Proposition 3.4]{gross-modular}.
 \end{proof}
 
 If $v$ is a place of $F$ dividing $\mathfrak{c}$ then, in particular, $v\nmid\mathfrak{dm}$ and   $B_v\iso M_2(F_v)$.  Let $W_v$ denote  a two dimensional $F_v$-vector space on which  $B_v$ acts on the left.  As $W_v$ is free of rank one over $E_v$, we may choose $w_0\in W_v$ such that $W_v=E_v\cdot w_0$.  For each rank two $\co_{F,v}$-submodule $\Lambda_v\subset W_v$ set
$$
\co(\Lambda_v)=\{ b\in B_v\mid b \cdot \Lambda_v\subset \Lambda_v \},
$$
a maximal order of $B_v$. As $\mathfrak{s}\mid\mathfrak{c}$ by Hypothesis \ref{hyp} we may consider the two lattices in $W_v$
$$
L'_v= \co_{\mathfrak{c},v} w_0
\hspace{1cm}
L_v= \co_{\mathfrak{cs}^{-1},v} w_0.
$$ 

 Choose a global order $S\subset B$ such that $S_v=\co(L_v)\cap \co(L_v')$ for every place $v\mid{\mathfrak{c}}$ and such that for every finite place $v\nmid\mathfrak{c}$, $S_v$ has reduced discriminant $\mathfrak{m}_v$ and contains $\co_{E,v}$ (which can be done by Lemma \ref{choosing the order}).  The group  $\widehat{S}^\times$ acts on $\prod_{v\mid\mathfrak{c}} L_v/L_v' \iso \co_F/\mathfrak{s}$ through a homomorphism $\vartheta:\widehat{S}^\times\map{}(\co_F/\mathfrak{s})^\times$, and we define $V$ to be the kernel of $\vartheta$.  One should regard $V\subset G(\A_f)$ as a quaternion analogue of the congruence subgroup $K_0(\mathfrak{m})\cap K_1(\mathfrak{s})$.  Define a CM-cycle of level $V$
\begin{eqnarray*}
Q_\chi(g)   &=&   \left\{
\begin{array}{ll}
\chi(t)  & \mathrm{if\ } g= t v\mathrm{\ for\ some\ } t \in T(\A_f),\  v\in V \\ 
0 & \mathrm{otherwise.}
\end{array}\right.  
\end{eqnarray*}
For this definition to make sense we need to know that $\chi$ is trivial on $T(\A_f)\cap V$.  This is immediate from the following 

\begin{Lem}\label{theta factor}
We have $\widehat{\co}_{\mathfrak{c}}^\times=T(\A_f)\cap \widehat{S}^\times$, and 
$\chi_0\circ\vartheta$ and $\chi$ have the same restriction to
$\widehat{\co}_{\mathfrak{c}}^\times$.
\end{Lem}

\begin{proof}
For $v\nmid\mathfrak{c}$ a finite place of $F$, $\co_{\mathfrak{c},v}\subset S_v$.  As $\co_{\mathfrak{c},v}$ is a maximal order in $E_v$ we must therefore have $\co_{\mathfrak{c},v}=E_v\cap S_v$.  For $v\mid\mathfrak{c}$ it follows from $\co=\{x\in E_v\mid x\co \subset \co \}$ for any order $\co\subset E_v$ that
$$
\co_{\mathfrak{c},v}=\co_{\mathfrak{c},v}\cap\co_{\mathfrak{cs}^{-1},v}=E_v\cap\co(L_v)\cap\co(L_v')=E_v\cap S_v,
$$ 
proving the first claim.  For the second claim, if $v\nmid\mathfrak{s}$ then both $\vartheta_v$ and $\chi_v$ are trivial on $\co_{\mathfrak{c},v}^\times=\co_{F,v}^\times(1+\mathfrak{c}\co_{E,v})^\times$.  If $v\mid\mathfrak{s}$ then $\vartheta_v:\co_{\mathfrak{c},v}^\times\map{}(\co_{F,v}/\mathfrak{s}_v)^\times$ is given by $\vartheta_v(x(1+c y))=x$ for $x\in\co_{F,v}^\times$, $y\in\co_{E,v}$, and $c\in\co_{F,v}$ satisfying $c\co_{F,v}=\mathfrak{c}_v$.  Thus 
$$
(\chi_{0,v}\circ \vartheta)(x(1+cy))=\chi_{0,v}(x)=\chi_v(x)=\chi_v(x(1+cy)).
$$
\end{proof}

\begin{Lem}\label{choice of epsilon}
  For every finite place $v$ there is an $\epsilon_v\in B_v$ satisfying
\begin{enumerate}
\item
 $E_v\epsilon_v=B_v^-$
 \item
$\ord_v(\mathrm{N}(\epsilon_v))=\ord_v(\mathfrak{r})$
\item
If $v\nmid\mathfrak{c}$ then $\epsilon_v\in S_v$
\item
if $v\mid\mathfrak{c}$ then $\epsilon_v w_0\in \mathfrak{c}\co_{E,v}w_0$.
\end{enumerate}
\end{Lem}

\begin{proof}
First fix an $\epsilon_v$ which generates $B_v^-$ as a left $E_v$-module.  If $v$ is split or ramified in $E$ then we may multiply $\epsilon_v$ on the left by an element of $E_v^\times$ to ensure that (b) holds.  If $v$ is inert in $E$ then it follows from the proof of Lemma \ref{eta parametrization} that $\omega_v(\mathrm{N}(\epsilon_v))$ is $1$ if $B_v$ is split and is $-1$ if $B_v$ is ramified.  Condition (\ref{B ramification}) then implies that $\omega_v(\mathrm{N}(\epsilon_v))=\omega_v(\mathfrak{r})$, and so again we may multiply $\epsilon_v$ on the left by an element of $E_v^\times$ so that (b) holds.  Assume now that $v\nmid\mathfrak{c}$ and define an order $R_v=\co_{E,v}+\co_{E,v}\epsilon_v$.  An easy calculation shows that $R_v$ has reduced discriminant $\mathfrak{d}_v\mathfrak{m}_v$, and so may be enlarged to an order $R_v'$ of reduced discriminant $\mathfrak{m}_v$.  By Lemma \ref{choosing the order} $tR_v't^{-1}=S_v$ for some $t\in E_v^\times$.  Replacing $\epsilon_v$ by $t\epsilon_v t^{-1}=t\overline{t}^{-1}\epsilon_v$ we find that (c) holds.  Now assume that $v\mid\mathfrak{c}$.  As $W_v$ is free of rank one over $E_v$ there is an $x\in E_v$ such that $\epsilon_v \cdot w_0=x\cdot  w_0$, and it follows that  $\mathrm{N}(\epsilon_v)w_0=-\epsilon_v^2 w_0=-\mathrm{N}(x)w_0$.
Therefore $\ord_v(\mathfrak{c}^2)=\ord_v(\mathrm{N}(x))$.  If $v$ is inert in $E$ then this implies $x \in\mathfrak{c}\co_{E,v}$ and hence (d) holds.  If $v$ is split in $E$ then we need not have $x\in\mathfrak{c}\co_{E,v}$, but there is some $t\in E_v^\times$ satisfying $\mathrm{N}(t)=1$ and $tx\in\mathfrak{c}\co_{E,v}$.  Replacing $\epsilon_v$ by $t\epsilon_v$ we again find that (d) holds.
\end{proof}

Let  $R\subset B$ be a global order such that $R_v=\co_{E,v}+\co_{E,v}\epsilon_v$ at every finite place $v$,  with $\epsilon_v$ satisfying the properties of Lemma \ref{choice of epsilon}.  There is a natural $\co_F$-algebra homomorphism  $R\map{} \co_E/\mathfrak{c}\co_E$  defined by $b\mapsto b^+$ (with notation as in \S \ref{ss:prelims}), and  the kernel of the induced homomorphism $\widehat{R}^\times\map{}(\co_E/\mathfrak{c}\co_E)^\times$ will be denoted $U$.
Define a CM-cycle of level $U$
\begin{eqnarray*}
P_\chi(g)   &=&   \left\{
\begin{array}{ll}
\chi(t)  & \mathrm{if\ } g=t u\mathrm{\ for\ some\ } t \in T(\A_f),\  u\in U \\ 
0 & \mathrm{otherwise}
\end{array}\right.  
\end{eqnarray*}
so that $P_\chi=\prod_vP_{\chi,v}$ where the function $P_{\chi,v}$ on $G(F_v)/U_v$ agrees with that constructed in  \S \ref{ss:unramified local calculations} and \S \ref{ss:ramified local calculations} (with $\mathfrak{e}=\mathfrak{r}=\mathfrak{mc}^2$).  The compact open subgroups and CM-cycles constructed above satisfy $U\subset V$ and 
\begin{equation}\label{cycle trace}
[V_T:U_T]\cdot Q_\chi(g)=\sum_{h\in V/U} P_\chi(gh).
\end{equation}
For each  ideal $\mathfrak{a}$ prime to $\mathfrak{c}$ we have, from \S  \ref{ss:unramified local calculations} and \S \ref{ss:ramified local calculations},  a CM-cycle of level $U$ defined as the product 
$$
P_{\chi,\mathfrak{a}}(g)=  \prod_{v\mid{\mathfrak{a}}} P_{\chi,\mathfrak{a},v} (g_v) \prod_{v\nmid\mathfrak{a}} P_{\chi,v}(g_v).
$$   
If $\mathfrak{a}$ is prime to $\mathfrak{dr}$ then $R_v$ is a maximal order for each $v\mid\mathfrak{a}$. and we define the Hecke operator $T_\mathfrak{a}$ on CM-cycles of level $U$
$$
(T_\mathfrak{a}P)(g)= \sum_{h\in H(\mathfrak{a})/U} P(gh),
$$
where $H(\mathfrak{a})=\prod_{v\mid\mathfrak{a}} H(\mathfrak{a}_v) \cdot \prod_{v\nmid\mathfrak{a}} U_v$ and $H(\mathfrak{a}_v)$ was defined in \S \ref{ss:unramified local calculations} for $v\mid\mathfrak{a}$.  One then has the relation $T_\mathfrak{a}P_\chi=P_{\chi,\mathfrak{a}}$.

For the remainder of \S \ref{s:central value} the letters $U$ and $V$ will be used exclusively for the compact open subgroups constructed above.


\subsection{Toric newvectors and the Jacquet-Langlands correspondence}
\label{ss:toric}


Let $\mathrm{Ram}(B)$ denote the set of places of $F$ at which $B$ is nonsplit and let $\pi$ be a cuspidal automorphic representation of $\GL_2(\A)$.
If $\pi_v$ is square-integrable  for every $v\in\mathrm{Ram}(B)$ then there is a unique infinite-dimensional automorphic representation $\pi'$ of $G(\A)$ such that for every $v\not\in\mathrm{Ram}(B)$, $\pi_v\iso\pi_v'$ as representations of $G(F_v)\iso\GL_2(F_v)$.  We then say that $\pi$ is the \emph{Jacquet-Langlands lift} of $\pi'$.  There are many references for the Jacquet-Langlands correspondence including \cite{gelbart,gelbart-jacquet,jacquet-langlands,jordan-livne,lubotzky}

\begin{Lem}\label{square-integrable}
With $\Pi$ the automorphic representation fixed at the beginning of \S \ref{s:central value}, if $v\in\mathrm{Ram}(B)$ is a nonarchimedean place then either
\begin{enumerate}
\item
$\ord_v(\mathfrak{m})=1$ and $\Pi_v$ is a twist of the Steinberg representation by an unramified character
\item
or $\ord_v(\mathfrak{m})>1$ and $\Pi_v$ is supercuspidal.
\end{enumerate}
In particular $\Pi_v$ is square integrable.
\end{Lem}

\begin{proof}
If $v\in\mathrm{Ram}(B)$ is nonarchimedean then (\ref{B ramification}) implies that $\ord_v(\mathfrak{m})=\ord_v(\mathfrak{n})$ is odd and $\Pi_v$ has unramified central character. The lemma now follows from standard formulas for the conductor of irreducible admissible representations as in \cite[(12.3.9.1)]{nek}
\end{proof}

For the remainder of \S \ref{ss:toric} we assume that $\Pi$ is cuspidal and that  either $\Pi_v$ is a weight $2$ discrete series at each archimedean $v$ and $B$ is totally definite, or that $\Pi_v$ is a weight $0$ principal series at each archimedean $v$ and $B$ is totally indefinite.  In either case it follows from Lemma  \ref{square-integrable} that $\Pi_v$ is square integrable for each $v\in\mathrm{Ram}(B)$ and so $\Pi$ is the Jacquet-Langlands lift of some $\Pi'$.

\begin{Def}\label{JL-newvector}
For any place $v$ of $F$ we define a \emph{newvector}  $\phi \in \Pi_v'$ to be a nonzero vector such that
\begin{enumerate}
\item
if $v$ is a nonarchimedean place then $\phi$ is $V_v$-fixed,
\item
If $v$ is an archimedean place and we are in the weight $0$ case above, then $\phi$ is fixed by the action of $E_v^\times\iso \R^\times \cdot \mathrm{SO}_2(\R)$,
\item
if $v$ is an archimedean place and we are in weight $2$ case then we impose no condition on $\phi$.
\end{enumerate}
A \emph{newvector} in $\Pi'\iso\bigotimes_v \Pi_v'$ is a product of local newvectors.
\end{Def}

\begin{Lem}\label{quaternion newvector}
Up to scaling there is a unique newvector in $\Pi'$.
\end{Lem}

\begin{proof}
It suffices to prove existence and uniqueness everywhere locally.  If $v$ is archimedean this is clear (in the weight $2$ case $\Pi_v'$ is the one-dimensional trivial representation of $G(F_v)$ by \cite[Lemma 4.2(2)]{jordan-livne}), so assume that $v$ is  nonarchimedean.  If $B_v$ is split then there is an isomorphism $B_v\iso M_2(F_v)$ which identifies $V_v\iso K_0(\mathfrak{m}_v)\cap K_1(\mathfrak{s}_v)$, and so the claim follows from the theory of newvectors for $\GL_2(F_v)$ as in \S \ref{ss:automorphic forms}.   If $B_v$ is nonsplit then (\ref{B ramification}) implies that $v\mid\mathfrak{m}$ and $v\nmid\mathfrak{c}$.  As $V_v=S_v^\times$ with $S_v$ an order of $B_v$ of discriminant $\mathfrak{m}_v$ containing $\co_{E,v}$, the claim is a special case of  \cite[Proposition 6.4]{gross-modular}.
\end{proof}

\begin{Def}\label{JL-toric}
For any place $v$ of $F$ let $E_v^\times$ act on $\Pi_v'$ via the embedding $T(F_v)\map{} G(F_v)$.  We define a \emph{toric newvector} $\phi\in \Pi_v'$ to be a nonzero vector such that
\begin{enumerate}
\item
if $v\nmid\mathfrak{dr}$ then $\phi$ is a  newvector,
\item
if $v\mid\mathfrak{d}$ then $\phi$ is $U_v$-fixed and satisfies $t\cdot \phi =\overline{\chi}_v(t)\cdot \phi$ for every $t\in E_v^\times$,
\item
if $v\mid\mathfrak{r}$ then  $\phi$ is $U_v$-fixed and satisfies $t\cdot \phi=\overline{\chi}_v(t)\cdot \phi$ for every $t\in \co_{E,v}^\times$.
\end{enumerate}
A \emph{toric newvector} in $\Pi'\iso \bigotimes\Pi_v'$ is a product of local toric newvectors.
\end{Def}

\begin{Lem}
Up to scaling there is a unique toric newvector in $\Pi'$.
\end{Lem}

\begin{proof}
Again it suffices to prove the claim everywhere locally.  If $v\nmid\mathfrak{dr}$ then the claim is a restatement of Lemma \ref{quaternion newvector}.    If $v\mid\mathfrak{d}$ then $\chi_v$ has the form $\chi_v=\nu_v\circ\mathrm{N}$ for some unramified character $\nu_v$ of $F_v^\times$.  By a theorem of Waldspurger \cite[Theorem 2.3.2]{zhang2} the representation $\Pi_v'\otimes\nu_v$ has a unique line of $E_v^\times$-fixed  vectors, and by a theorem of Gross-Prasad \cite[Theorem 2.3.3]{zhang2} this line is also fixed by the unit group of any maximal order of $B_v$ containing $\co_{E,v}$.  As $R_v$ may be enlarged to such an order, the $E_v^\times$-fixed vectors in $\Pi_v'\otimes\nu_v$ are also fixed by $U_v=R_v^\times$.  It follows that  $\Pi_v'$ has a unique line of $U_v$-fixed vectors on which $E_v^\times$ acts through $\chi_v^{-1}$.

If $v\mid\mathfrak{m}$ then $R_v= S_v$ (as $R_v\subset S_v$ and both have reduced discriminant $\mathfrak{m}_v$), $U_v=V_v$, and a toric newvector is just a nonzero $V_v$-fixed vector; again the claim follows from  Lemma \ref{quaternion newvector}.
If $v\mid\mathfrak{c}$ but $v\nmid\mathfrak{s}$ then $\chi_{v}$ is trivial on $\co_{F,v}^\times$, and so we may find a character $\chi_v'$ of $E_v^\times$ which is trivial on $F_v^\times$ but agrees with $\chi_v$ on $\co_{E,v}^\times$.  By \cite[Theorem 2.3.5]{zhang2} (Zhang's $\Gamma$ is our $R_v^\times=\co_{E,v}^\times U_v$) there is a unique line of $U_v$-fixed vectors in $\Pi_v'$ on which $\co_{E,v}^\times$ acts through $\overline{\chi}_v'$, and thus a unique toric newvector in $\Pi_v'$.  If $v\mid\mathfrak{s}$ then $\Pi_v'\iso \Pi_v$ is a principal series $\Pi_v\iso \Pi(\mu_v,\chi_{0,v}^{-1}\mu_v^{-1})$ and $\chi_{v}=\nu_v\circ\mathrm{N}$ for some character $\nu_v$ of $F_v^\times$ of conductor  $\mathfrak{c}$  (both claims by Hypothesis \ref{hyp}).   It follows that $\Pi_v'\otimes\nu_v$ has trivial central character and conductor $\mathfrak{c}_v^2$.  As $R_v$ has reduced discriminant $\mathfrak{c}_v^2$ and contains $\co_{E,v}$  there is a unique line of $R_v^\times$-fixed vectors in $\Pi_v' \otimes \nu_v$ by \cite[Theorem 2.3.3]{zhang2}.  As $R_v^\times=\co_{E,v}^\times\cdot U_v$ we find that $\Pi_v'\otimes\nu_v$ has a unique line of $U_v$-fixed vectors on which $\co_{E,v}^\times$ acts through the trivial character, and the claim now follows from the observation that $\mathrm{N}(U_v)\subset 1+\mathfrak{c}_v\subset \ker(\nu_v)$.
\end{proof}


\subsection{Central values for holomorphic forms}
\label{ss:holomorphic values}


In addition to Hypothesis \ref{hyp} we assume that $\Pi_v$ is a discrete series of weight $2$ for every archimedean place $v$, and that $\epsilon(1/2,\mathfrak{r})=1$.  Let $B$ be the (unique up to isomorphism)  totally definite quaternion algebra over $F$ satisfying (\ref{B ramification}) for all finite places of $F$. Taking $m$ to be the constant function $1$ on $G(F)$, let $k_U(x,y)$ be the function on $C_U\times C_U$ defined by (\ref{kernel def}) and let $\langle P,Q\rangle_U$ be the associated height pairing on CM-cycles of level $U$ defined by (\ref{pairing}).  According to \cite[\S 7.2]{vatsal}   the sum defining $k_U(x,y)$ is actually finite.  Recall that we have set $\mathfrak{r}=\mathfrak{mc}^2$ and abbreviate $\Theta_{\mathfrak{r}}=\Theta_{\mathfrak{r},1/2}$.

\begin{Prop}\label{generating series}
 Fix $a\in\A^\times$ and assume that $\mathfrak{a}=a\co_F$ is prime to $\mathfrak{c}$.  Then 
$$
\frac{ H_F }{\lambda_U}  [\widehat{\co}_E^\times: U_T]
\cdot B( -a;\Theta_{\mathfrak{r}})=
2^{[F:\Q]} |d|^{1/2} |a| \langle P_{\chi,\mathfrak{a}},P_\chi \rangle_U \cdot e_\infty(a)
$$
where, as in \S \ref{ss:heights}, $H_F$ is the class number of $F$ and $\lambda_U=[\co_F^\times:\co_F^\times\cap U]$.
\end{Prop}

\begin{proof}
Suppose $\gamma\in G(F)$ is nondegenerate and let $\eta$ and $\xi$ be defined by (\ref{eta}).  Then
Corollaries \ref{unramified links} and \ref{ramified links} show that
$$
 [\co_{E,v}^\times: \co_{F,v}^\times U_{T,v} ]\cdot B_v(a,\eta,\xi;  \Theta_{\mathfrak{r}}) =
\tau_v(\gamma)\cdot |a|_v |d|_v^{1/2}\cdot O_U^\gamma (P_{\chi,\mathfrak{a},v })
$$
for every finite place $v$ of $F$.  By (\ref{orbital decomp})
$$
\langle P_{\chi,\mathfrak{a}} ,P_\chi\rangle_U^\gamma= [Z(\A_f):Z(F) U_Z] \cdot \prod_{v\nmid\infty} O_U^\gamma(P_{\chi,\mathfrak{a},v})
$$
By the final claims of Proposition \ref{kernel coefficients} and Lemma \ref{global tau}, for $v$ an archimedean place
$$
B_v(a, \eta,\xi;\Theta_{\mathfrak{r}}) = 2 \tau_v(\gamma) |a|_v e_v(-a).
$$
Combining these equalities gives
$$
\frac{ H_F }{\lambda_U}  [\widehat{\co}_E^\times: U_T]
\cdot B(a,\eta,\xi;\Theta_{\mathfrak{r},1/2})  
 = 2^{[F:\Q]} |d|^{1/2} |a|  \langle P_{\chi,\mathfrak{a}},P_\chi\rangle_U^\gamma \cdot e_\infty(-a).
$$

By Lemma \ref{local kernel functional}, given  $\eta,\xi\in F^\times$ with $\eta+\xi=1$ we have
$B(a,\eta,\xi;\Theta_{\mathfrak{r}})=0$ unless $\omega_v(-\eta\xi)=\epsilon_v(1/2,\mathfrak{r},\psi)$ for every place $v$ of $F$.   Combining (\ref{B ramification}) with Lemma \ref{eta parametrization} we find that $B(a,\eta,\xi;\Theta_{\mathfrak{r}})=0$ unless the pair $\eta,\xi$ is of the form (\ref{eta}) for some $\gamma\in G(F)$.
Therefore
\begin{eqnarray}\lefteqn{
\frac{ H_F }{\lambda_U}  [\widehat{\co}_E^\times: U_T]
 \cdot \sum_{ \substack{\eta,\xi\in F^\times \\ \eta+\xi=1}  } B(-a,\eta,\xi;\Theta_{\mathfrak{r}}) \nonumber } \\ \label{Fourier compare 1}
& & = 2^{[F:\Q]}  |d|^{1/2} |a|   \sum_{ \substack{\gamma\in T(F)\backslash G(F) /T(F) \\ \gamma\mathrm{\ nondegenerate} } }\langle P_{\chi,\mathfrak{a}},P_\chi\rangle_U^\gamma \cdot e_\infty(a).
\end{eqnarray}

It remains to compare the  linking numbers at the two degenerate choices of $\gamma$  (i.e.  $\gamma\in B^\pm$) with the degenerate terms $A_0(a;\Theta_{\mathfrak{r}})$ and $A_1(a;\Theta_{\mathfrak{r}})$ of  (\ref{global kernel decomp}).  
First suppose $\gamma=\epsilon^\circ$ where $\epsilon^\circ $ satisfies $B^-=E\epsilon^\circ$, so that $(\eta,\xi)=(0,1)$.   Let $z\in\A_E^\times$ be such that $\epsilon_v^\circ =z_v \epsilon_v$ for every finite place $v$.  If $\chi\not=\chi^*$ then both $A_1(a;\Theta_{\mathfrak{r}})$ and $\langle P_{\chi,\mathfrak{a}}, P_\chi\rangle^\gamma_U$ vanish, by Lemmas \ref{cusp theta}  and \ref{orbital integrals}, respectively.  We therefore assume that $\chi=\chi^*$.  If $\chi$ is ramified then $B_v(a;E_{\mathfrak{r},s})=0$ for any $v\mid \mathfrak{c}$ by Proposition \ref{eisenstein coefficients} and the inequality $\ord_v(\mathfrak{ar}^{-1})=-\ord_v(\mathfrak{r})<0$.  Abbreviating  $\alpha=\left(\begin{matrix} a\delta^{-1} & \\ & 1 \end{matrix}\right)$, it follows that
$
W_{\mathfrak{r}, s}\left(\alpha h_T\right) =0
$ 
for any $T\subset S$ and so $A_1(a;\Theta_{\mathfrak{r}})=0$.
Similarly if $\chi$ is ramified then $P_{\chi,\mathfrak{a}}(\epsilon^\circ)=0$ by Lemma \ref{ramified degenerate link}, and so also $\langle P_{\chi,\mathfrak{a}}, P_\chi\rangle_{U,\gamma}=0$ by Lemma \ref{orbital integrals}.  We therefore assume that  $\chi$ is unramified.  
By (\ref{constant to section}), Proposition \ref{constant term}, and Lemma \ref{Lem:theta and eisenstein}, $C_\theta(\alpha h_T)=0$ unless $T=\emptyset$ or $S$, and so
 \begin{eqnarray*}
A_1(a;\Theta_{\mathfrak{r}})
&=&
\sum_{T\subset S} \overline{\chi}_T(\mathfrak{D})  C_\theta(\alpha h_T)  W_{\mathfrak{r},1/2}(\alpha h_T) \\
&=&
 B(a; E_{\mathfrak{r},1/2} )  C_\theta (\alpha)
+
\overline{\chi}(\mathfrak{D}) B(a; h_S E_{\mathfrak{r},1/2} )  C_{\theta}(\alpha h_S)\\
&=&
2\cdot B(a; E_{\mathfrak{r},1/2} )   C_\theta (\alpha) 
\end{eqnarray*}
where we have used  Propositions \ref{ES functional} and \ref{theta functional} for the third equality. Again using Proposition \ref{constant term} and Lemma \ref{Lem:theta and eisenstein} we find
\begin{eqnarray*}
C_\theta(\alpha)&=&
(-1)^{[F:\Q]} \nu(a\delta^{-1})|ad^{-1} \delta^{-1}|^{1/2} L^*(1,\omega) \\
B(a; E_{\mathfrak{r},1/2}) 
&=&
|r|^{1/2}B(ar^{-1};E_{\co_F,1/2}) \\
&=&
(-1)^{[F:\Q]} |dr|^{1/2}  \overline{\nu}(ar^{-1}\delta^{-1}) B(ar^{-1};\theta)
\end{eqnarray*}
where $r\co_F=\mathfrak{r}$ for $r\in \A^\times$ with $r_v=1$ at each archimedean $v$.
Therefore
$$
A_1(a;\Theta_{\mathfrak{r}}) = 2  \nu(\mathfrak{r})|ar\delta^{-1}|^{1/2} B(ar^{-1};\theta) L(1,\omega).
$$
On the other hand using Corollary \ref{unr final degen}, Lemma \ref{ramified degenerate link}, and 
$$
\chi(rz)=\nu(r)^2\nu(\mathrm{N}(z))= \nu(r)^2\nu(\mathrm{N}(\epsilon^\circ)^{-1}) =\nu(r)
$$ 
we find
$$
P_{\chi,\mathfrak{a}}(\epsilon^\circ)\cdot e_\infty(-a)= \nu(\mathfrak{r})|r|^{1/2}|a|^{-1/2} B(ar^{-1};\theta)
$$
and now (\ref{dirichlet functional}), (\ref{dirichlet value}),  and Lemma \ref{orbital integrals} imply that (for $\gamma=\epsilon^\circ$)
\begin{equation}\label{Fourier compare 3}
H_F\lambda_U^{-1}[\widehat{\co}_E^\times:U_T] \cdot A_1(-a;\Theta_{\mathfrak{r}})=
2^{[F:\Q]} |d|^{1/2} |a|  \langle P_{\chi,\mathfrak{a}}, P_\chi\rangle_U^\gamma \cdot e_\infty(a).
\end{equation}

A similar, but easier, argument also shows that (\ref{Fourier compare 3}) continues to hold if $\gamma=1$ and  $A_1$ is replaced by $A_0$. The theorem follows from this together with equation (\ref{Fourier compare 1}), equation  (\ref{linking decomposition}), and the decomposition (\ref{global kernel decomp}).
\end{proof}

We now construct a pairing $[P,Q]$ on CM-cycles of level $U$ taking values in the space of automorphic forms on $\GL_2(\A)$ as in \cite[(4.4.5)]{zhang2}. Endow the (finite)  set $S_U = G(F)\backslash G(\A_f)/U $ with the measure determined by
$$
\int_{S_U} \sum_{\gamma\in T(F)\backslash G(F)} P(\gamma g) \ dg =\int_{C_U} P(g)\ dg
$$
for any CM-cycle $P$ of level $U$.    For each $\mathfrak{a}$ prime to $\mathfrak{dr}$ there is a Hecke operator  $(T_\mathfrak{a}\phi)(g)= \sum \phi (gh)$
on $L^2(S_U)$ where the sum is over $h\in H(\mathfrak{a})/U$ as in \S \ref{ss:special cycles}.  For any $\phi \in L^2(S_U)$ we  have
$$
\int_{S_U} k_U(x,y) \phi (y)\ dy = \phi (x)
$$
and it follows  that that there is a decomposition $k_U(x,y)=\sum_{i=1}^\ell f'_i(x)\overline{f'_i(y)}$ where $\{f'_1,\dots,f'_\ell\}$ is any orthonormal basis for $L^2(S_U)$.   We choose this basis in such a way that  each $f'_i$ is a simultaneous eigenvector for every $T_\mathfrak{a}$ with $(\mathfrak{a},\mathfrak{dr})=1$.  The Jacquet-Langlands correspondence implies that for each $f_i'$ there is a (not necessarily unique) holomorphic automorphic form $f_i$ of weight $2$ on $\GL_2(\A)$ fixed by $K_1(\mathfrak{dr})$ having the same Hecke eigenvalues as $f_i'$.  Indeed, if $f_i'$ generates an infinite dimensional representation $\pi'$ of $G(\A)$ then take $f_i$ to be a newvector in the Jacquet-Langlands lift of $\pi'$.  If $f_i'$ generates a finite dimensional representation of $G(\A)$ then $f_i'(g)=\mu(\mathrm{N}(g))$ with $\mu$ some character of $\A^\times/F^\times$, and one takes  $f_i$ to be an Eisenstein series constructed from a function in the induced representation  $\mathcal{B}(\mu |\cdot|^{1/2},\mu |\cdot|^{-1/2})$. We may, and do, assume that  $\widehat{B}(\co_F,f_i)=1$ for every $i$.  For any CM-cycles $P$ and $Q$ of level $U$ we define a parallel weight $2$, holomorphic, $K_1(\mathfrak{dr})$-fixed automorphic form on $\GL_2(\A)$ 
$$
[P,Q]=\sum_{i=1}^\ell \left( \int_{C_U\times C_U}  P(x)f_i'(x)\overline{f'_i(y)Q(y)} \ dx\ dy \right) f_i.
$$
This form satisfies $\widehat{B}(\co_F,[P,Q])= \langle P,Q\rangle_U$ and,  for any ideal $\mathfrak{a}$ relatively prime to $\mathfrak{dr}$, $T_\mathfrak{a}\cdot [P,Q]=[P,T_\mathfrak{a} Q].$  Set  $\Psi=[P_\chi,P_\chi]$, an automorphic form of  central character $\chi_0^{-1}$ satisfying 
\begin{equation}\label{spectral form display}
\widehat{B}(\co_F;T_\mathfrak{a} \Psi)=\langle P_\chi,P_{\chi,\mathfrak{a}}\rangle_U.
\end{equation}

Let $\Pi'$ be the automorphic representation of $G(\A)$ whose Jacquet-Langlands lift is $\Pi$,  let $\phi^\chi_{\Pi'}$ be the  toric newvector in $\Pi'$ normalized by
$\int_{S_U} | \phi^\chi_{\Pi'} |^2  = 1$
and let $\Psi|_\Pi$ denote the projection of $\Psi$ to $\Pi$.   We may choose the basis $\{f_i'\}$ so  that $\phi^\chi_{\Pi'}= f'_1$. If we set $\mathcal{P}_\chi(g)=\sum_\gamma P_\chi(\gamma g)$ where the sum is over $\gamma\in T(F)\backslash G(F)$ then
$$
\widehat{B}(\co_F;\Psi|_\Pi) 
= \sum_{ \substack{1\le i\le \ell \\ \pi_i=\pi_1} } \left|\int_{S_U} \mathcal{P}_\chi(t) f'_i(t)\ dt \ \right|^2.
$$
The projection of $\overline{\mathcal{P}}_\chi$ to $\pi'_1$ is a toric newvector, hence a scalar multiple of $f'_1$, and so only the term $i=1$ contributes to the sum. It follows that
\begin{equation}\label{toric spectrum}
\widehat{B}(\co_F;\Psi|_\Pi) = 
  \left|\int_{C_U} P_\chi(t) \phi^\chi_{\Pi'} (t)\ dt \ \right|^2.
\end{equation}

\begin{Prop}\label{first holomorphic values}
Let $\phi_\Pi^\#$ be the orthogonal projection of the normalized newform $\phi_\Pi\in \Pi$ to the quasi-new line (defined in \S \ref{ss:quasi-new}).  Then
\begin{eqnarray*}\lefteqn{
2^{|S|}  H_F\lambda_U^{-1} [\widehat{\co}_E^\times: U_T]  \widehat{B}(\co_F;\phi_\Pi^{\#}  )
 L(1/2,\Pi\times\Pi_\chi)   } \\
&=&
|d|^{1/2} 2^{[F:\Q]}   ||\phi_\Pi^{\#} ||_{K_0(\mathfrak{dr})} ^2  \cdot 
\left|   \int_{C_U} P_\chi(t) \phi_{\Pi'}^\chi(t)\ dt \ \right|^2  
\end{eqnarray*}
in which $S$ is the set of prime divisors of $\mathfrak{d}$.
\end{Prop}

\begin{proof}
Let  $\overline{\Theta}_\mathfrak{r}|_\Pi$ and $\Psi|_\Pi$ denote the  projections of $\overline{\Theta}_\mathfrak{r}$ and $\Psi$ to $\Pi$.  Combining Proposition \ref{generating series}  and (\ref{spectral form display}) gives
$$
H_F \lambda_U^{-1} [\widehat{\co}_E^\times: U_T]
\cdot \widehat{B}( \co_F; T_\mathfrak{a} \overline{\Theta}_{\mathfrak{r}}  )
=  2^{[F:\Q]} |d|^{1/2}\widehat{B}(\co_F,T_\mathfrak{a}\Psi)  
$$
for all $\mathfrak{a}$ prime to $\mathfrak{dr}$.  The action of the operators $T_\mathfrak{a}$ with $(\mathfrak{a},\mathfrak{dr})=1$ on the space of all $K_1(\mathfrak{dr})$-fixed, holomorphic, parallel weight two automorphic forms on $\GL_2(\A)$ of central character $\chi_0^{-1}$ generates a semi-simple $\C$-algebra, and it follows from this and strong multiplicity one that there is a polynomial $e_\Pi$ in the Hecke operators $T_\mathfrak{a}$ such that $\overline{\Theta}_\mathfrak{r}|_\Pi=e_\Pi\cdot \overline{\Theta}_\mathfrak{r}$ and $\Psi|_\Pi=e_\Pi\cdot \Psi$.  We therefore deduce that
\begin{equation}\label{projection coefficients}
H_F \lambda_U^{-1}  [\widehat{\co}_E^\times: U_T]  \cdot \widehat{B}(\co_F;\overline{\Theta}_\mathfrak{r}|_\Pi)= 2^{[F:\Q]} |d|^{1/2} \widehat{B}(\co_F;\Psi|_\Pi).
\end{equation}

Under the decomposition $\Pi\iso\bigotimes_v\Pi_v$ the newform $\phi_\Pi$ is decomposable as a pure tensor $\phi_\Pi=\otimes \phi_{\Pi,v}$.  In the notation of \S \ref{ss:quasi-new} $\Lambda_v(\phi_{\Pi,v})\not=0$ for $v\mid\mathfrak{dc}$, and so $\phi_{\Pi,v}$ has nontrivial projection to the quasi-new line in $\Pi_v$.  It follows that $\phi^\#_\Pi\not=0$.  The form $\overline{\Theta}_\mathfrak{r}|_\Pi$ lies on the quasi-new line of $\Pi$ by Proposition \ref{quasi-new kernels}, and so if $\widehat{B}(\co_F;\phi_\Pi^\#)=0$ then also $\widehat{B}(\co_F; \overline{\Theta}_\mathfrak{r}|_\Pi)=0$. Using (\ref{toric spectrum}) and (\ref{projection coefficients}) we then see that both sides of the stated equality are $0$.  Therefore we may assume $\widehat{B}(\co_F;\phi_\Pi^\#)\not=0$ so that
$$
\overline{\Theta}_\mathfrak{r}|_\Pi = \frac{\widehat{B}(\co_F;\overline{\Theta}_\mathfrak{r}|_\Pi)}{  \widehat{B}(\co_F;\phi_\Pi^\#)  } \cdot \phi_\Pi^\#.
$$
Combining this with (\ref{rankin integral}) (with $b=1$) gives
\begin{eqnarray*}
\widehat{B}(\co_F;\overline{\Theta}_\mathfrak{r}|_\Pi )\cdot ||\phi_\Pi^\#||^2_{K_0(\mathfrak{dr})}  
&=&
\widehat{B}(\co_F;\overline{\Theta}_\mathfrak{r}|_\Pi )\cdot \langle \phi_\Pi , \phi_\Pi^\# \rangle_{K_0(\mathfrak{dr})}  \\
&=&
\widehat{B}(\co_F;\phi_\Pi^\#) \cdot  
\langle  \phi_\Pi,  \overline{\Theta}_\mathfrak{r} \rangle_{K_0(\mathfrak{dr})}  \\
&=& 
2^{|S|} \widehat{B}(\co_F;\phi_\Pi^{\#}  ) L(1/2,\Pi\times\Pi_\chi).
\end{eqnarray*}
The claim now follows from (\ref{toric spectrum}) and (\ref{projection coefficients}).
\end{proof}

\begin{Thm}\label{second holomorphic values}
Let $\phi_\Pi\in \Pi$ be the normalized newvector (in the sense of \S \ref{ss:automorphic forms})  and    let $\phi_{\Pi'}\in\Pi'$ be the newvector (in the sense of Definition \ref{JL-newvector}) normalized by $\int_{S_V}  | \phi_{\Pi'} |^2 =1$.  Then
$$
\frac{ L(1/2,\Pi\times\Pi_\chi) }{ ||\phi_\Pi ||^2_{K_0(\mathfrak{n})} }
=
\frac{2^{[F:\Q]} }{ H_{F,\mathfrak{s}} \sqrt{ \mathrm{N}_{F/\Q}  (\mathfrak{dc}^2) }}  \cdot \left|    \int_{C_V} Q_\chi(t) \phi_{\Pi'}(t)\ dt \ \right|^2  .
$$
where $H_{F,\mathfrak{s}}=[Z(\A_f):F^\times V_Z]$ is the order of the ray class group  of conductor $\mathfrak{s}$.
\end{Thm}

\begin{proof}
The proof is postponed until \S \ref{level comparison}.
\end{proof}


\subsection{Central values for Maass forms}
\label{ss:maass}


In addition to Hypothesis \ref{hyp} we assume that $\Pi_v$ is a weight zero principal series  for every archimedean place $v$, and that 
 $ \epsilon(1/2,\mathfrak{r})=(-1)^{[F:\Q]}.$   Thus the weight $0$ kernel of \S \ref{ss:zero kernel} satisfies $\Theta^*_{\mathfrak{r},s}=\Theta^*_{\mathfrak{r},1-s}$.  Let $B$ be the  (unique up to isomorphism)  totally indefinite quaternion algebra over $F$ satisfying (\ref{B ramification}) for every finite place $v$.   Let $\mathbb{S}=\mathrm{Res}_{\C/\R}\mathbb{G}_m$ and set $F_\infty=F\otimes_\Q\R$.  As $F_\infty$ is naturally an $\R$-algebra, 
$$
T_{/F_\infty}=T\times_{\Spec(F)}\Spec(F_\infty)
\hspace{1cm}
G_{/F_\infty}=G\times_{\Spec(F)}\Spec(F_\infty)
$$
are naturally algebraic groups over $\R$.  Fixing an embedding of real algebraic groups $\mathbb{S}\map{} T_{/F_\infty}$ the embedding $T\map{}G$ determines an embedding $x_0:\mathbb{S}\map{}G_{/F_\infty}$, and we let $X$ denote the $G(F_\infty)$-conjugacy class of $x_0$.  As $T(F_\infty)$ is the stabilizer of $x_0$ we may identify $X\iso G(F_\infty)/T(F_\infty)$.  Writing $\mathcal{H}=\C-\R$ and choosing an isomorphism $G(F_\infty)\iso \GL_2(\R)^{[F:\Q]}$, we may fix a point in $\mathcal{H}^{[F:\Q]}$ whose stabilizer under the action of $G(F_\infty)$ is $T(F_\infty)$.  This allows us to identify $X\iso \mathcal{H}^{[F:\Q]}$.  Endowing $\mathcal{H}$ with the usual hyperbolic volume form $y^{-2} dxdy$ we obtain a measure on $X$.  Define 
$$
S_U  = G(F)\backslash X\times G(\A_f)/U
$$
endowed with the quotient measure induced from that on  $G(\A_f)/U$ giving each coset volume $1$. 
The map $G(\A_f)\map{} X\times G(\A_f)$ defined by $g\mapsto (x_0,g)$ restricts to a function on $T(\A_f)$ and determines an embedding $C_U\map{}S_U$.

 If $\phi$ is a weight $0$ Maass form on $\GL_2(\A)$ with \emph{parameter}  $t_v$ in the sense of \cite[\S 4]{zhang3} at an archimedean place $v$  then we set
$$
B_v(a;\phi)=|a|_v^{1/2} \int_0^\infty e^{-\pi |a|_v(y+y^{-1}) }  y^{it_v} \ d^\times y.  
$$
Define $B_\infty(a;\phi)= \prod_{v\mid\infty} B_v(a;\phi)$ and define $\widehat{B}(\mathfrak{a};\phi)$ for $\mathfrak{a}=a\co_F$ by 
$$
B(a;\phi)=B_\infty(a;\phi)\cdot \widehat{B}(\mathfrak{a};\phi).
$$
Let $\Pi'$ be the automorphic representation of $G(\A)$ whose Jacquet-Langlands lift is $\Pi$, and let $\phi^\chi_{\Pi'}$ be the  toric newvector in $\Pi'$ normalized by
$\int_{S_U} |\phi^\chi_{\Pi'}|=1$

\begin{Prop}\label{first maass values}
Let $\phi_\Pi^\#$ be the orthogonal projection  of the normalized newform $\phi_\Pi\in \Pi$ to the quasi-new line in $\Pi$.   Then
\begin{eqnarray*}\lefteqn{
2^{|S|}  H_F\lambda_U^{-1} [\widehat{\co}_E^\times: U_T] \widehat{B}(\co_F;\phi_\Pi^{\#}  )
 L(1/2,\Pi\times\Pi_\chi)   } \\
&=&
|d|^{1/2} 4^{[F:\Q]}   ||\phi_\Pi^{\#} ||_{K_0(\mathfrak{dr})} ^2 \cdot 
\left| \int_{C_U} P_\chi(t) \phi_{\Pi'}^\chi(t)\ dt \ \right|^2  
\end{eqnarray*}
in which $S$ is the set of prime divisors of $\mathfrak{d}$.  
\end{Prop}

\begin{proof}
Fix $a\in\A^\times$ and assume that $\mathfrak{a}=a\co_F$ is prime to $\mathfrak{c}$.  We abbreviate $\Theta^*_\mathfrak{r}=\Theta^*_{\mathfrak{r},1/2}$.  Suppose $v$ is an infinite place of $F$.  For each  $a\in\A^\times$, $\gamma\in  G(F_v)$, and $\eta, \xi$ as in (\ref{eta}) define the multiplicity function
$$
m^*_v(a,\gamma) = \left\{ \begin{array}{ll}
 4 e^{2\pi a_v(\xi-\eta)} &\mathrm{if\ } \xi a_v\le 0\mathrm{\ and\ }\eta a_v \ge 0 \\
0 &\mathrm{otherwise.} \end{array}\right.  
$$
If $\gamma\in G(F_\infty)$ set $m^*_\infty(a,\gamma)=\prod_{v\mid\infty}m^*_v(a,\gamma_v)$. 
Exactly as in  Proposition \ref{generating series}, using the formulas of  \S \ref{ss:zero kernel} to supplement those  of \S \ref{ss:kernel}, we find
\begin{equation}\label{maass match}
\frac{H_F}{\lambda_U}   [\widehat{\co}_E^\times: U_T]  \cdot B( a;\Theta^*_{\mathfrak{r}})=
 |d|^{1/2} |a|  \sum_{\gamma \in T(F)\backslash G(F)/T(F)} \langle P_{\chi,\mathfrak{a}},P_\chi \rangle_U^\gamma \cdot m_\infty^*(a;\gamma).
\end{equation}
The remainder of the proof is similar to that of Proposition \ref{first holomorphic values}; see \cite[\S 4.4]{zhang2} for details.  Briefly, for any Maass form $\phi$ on $S_U$ the kernel 
$$
k_U(a;x,y)= \sum_{\gamma\in G(F)/(Z(F)\cap U) } \mathbf{1}_{U}(x_f^{-1}\gamma y_f) m_\infty^*(a; x_\infty^{-1}\gamma y_\infty)
$$
satisfies
\begin{eqnarray*}
 \int_{S_U} k_U(a;x,y) \phi(y)\ dy 
&=&
\int_{X}  m_\infty^*(a;  y_\infty) \phi(x y_\infty)\ dy_\infty  \\
&=& 4^{[F:\Q]}  B_\infty(a;\phi)\cdot \phi(x).
\end{eqnarray*}
 Exactly as in \cite[Lemma 4.4.3]{zhang2} or \cite[\S 16]{zhang3} this leads to a spectral decomposition of the kernel $k_U(a;x,y)$, and the proposition follows from (\ref{maass match}), which is our analogue of \cite[(16.1)]{zhang3}, exactly as in \cite[\S 16]{zhang3}.  
 \end{proof}

\begin{Thm}\label{second maass values}
Let $\phi_\Pi\in \Pi$ be the normalized newvector (in the sense of \S \ref{ss:automorphic forms})  and    let $\phi_{\Pi'}\in\Pi'$ be the newvector (in the sense of Definition \ref{JL-newvector}) normalized by $\int_{S_V}  | \phi_{\Pi'} |^2 =1$.  Then
$$
\frac{ L(1/2,\Pi\times\Pi_\chi) }{  ||\phi_\Pi ||_{K_0(\mathfrak{n})}^2 }
=
\frac{4^{[F:\Q]} }{ H_{F,\mathfrak{s}}\sqrt{ \mathrm{N}_{F/\Q}  (\mathfrak{dc}^2) }} 
\left|     \int_{C_V} Q_\chi(t) \phi_{\Pi'}(t)\ dt \ \right|^2  
$$
where $H_{F,\mathfrak{s}}$ is the order of the ray class group of $F$ of conductor $\mathfrak{s}$.
\end{Thm}

\begin{proof}
The proof is postponed until \S \ref{level comparison}.
\end{proof}


\subsection{A particular family of Maass forms}
\label{ss:maass family}


Fix a $\tau\in\C$ and if $\chi_0$ is trivial assume that $\tau\not=0,1$. Let $\Pi_\tau$ denote the (irreducible) weight zero principal series representation
$$
\Pi_\tau=\Pi(|\cdot|^{\tau-1/2}, \chi_0^{-1}|\cdot|^{1/2-\tau })
$$
of $\GL_2(\A)$ of conductor $\mathfrak{s}$ and central character $\chi_0^{-1}$.   We construct an Eisenstein series $\mathcal{E}_\tau\in\Pi_\tau$ as follows.  Define a Schwartz function $\Omega=\prod_v \Omega_v$ on $\A\times\A$ by
$$
\Omega_v (x,y )= \left\{\begin{array}{ll}
\mathbf{1}_{\co_{F,v}}(x)\mathbf{1}_{\co_{F,v}}(y) & \mathrm{if\ }v\nmid \mathfrak{s}\infty \\
 \chi_{0,v}^{-1} (y)  \mathbf{1}_{\mathfrak{s}_v }(x) \mathbf{1}_{\co_{F,v}^\times }(y)     & \mathrm{if\ }v\mid\mathfrak{s} \\
e^{-\pi (x^2+y^2)} &\mathrm{if\ }v\mid\infty.
\end{array}\right.
$$
The function
$$
\mathcal{F}_{\tau}(g)  =  |\det(g)|^{\tau} \int_{\A^\times} \Omega \big( [0,x]\cdot g  \big) |x|^{2\tau} \chi_0(x) \ d^\times x
$$
is a newvector in the induced representation $\mathcal{B}(|\cdot|^{\tau-1/2}, \chi_0^{-1}|\cdot|^{1/2-\tau })$ defined in \cite[\S 2.2]{zhang2} and therefore the Eisenstein series (initially defined for $\mathrm{Re}(\tau)\gg 0$ and continued analytically)
$$
\mathcal{E}_{\tau}(g)=\sum_{\gamma\in B(F)\backslash \GL_2(F)} \mathcal{F}_{\tau}(\gamma g)
$$
is a newvector in $\Pi_\tau$.  The discrepancy between $\mathcal{E}_\tau$ and the normalized newvector is determined by the following

\begin{Lem}\label{continuous maass normalization}
$$
\int_{\A^\times} B(a;\mathcal{E}_{\tau})\cdot |a|^{s-1/2}\ d^\times a=
\frac{|\delta|^{\tau-1/2}\epsilon(1/2,\chi_0) }{ \mathrm{N}_{F/\Q}(\mathfrak{s})^{2\tau-1/2}}  
L(s,\Pi_\tau)
$$
\end{Lem}

\begin{proof}
As in \S \ref{ss:ES}, using
$$
B(a;\mathcal{E}_{\tau}) 
=
\int_{\A}\mathcal{F}_{\tau} \left(\begin{matrix} & 1 \\ -a \delta^{-1} & y\end{matrix}\right) \psi(-y)\ dy
$$
we see that  $B(a;\mathcal{E}_\tau)=\prod_v B_v(a;\mathcal{E}_\tau)$ where
$$
B_v(a;\mathcal{E}_\tau) =  |\delta |_v^{\tau-1/2} |a|_v^\tau \chi_{0,v}(\delta )  \int_{F_v}\psi^0_v( y)\int_{F_v^\times} \Omega_v( a x ,  xy ) |x|_v^{2\tau} \chi_{0,v}( x)\ d^\times x dy.
$$
If $v\nmid\mathfrak{s}\infty$ then a short calculation shows
$$
 \int_{F_v}\psi_v^0( y)\int_{F_v^\times} \Omega_v(ax,xy) |x|_v^{2\tau} \chi_{0,v}(x)\ d^\times x dy  
=
|\delta|_v^{1/2} \sum_{k=0}^{\ord_v(a)}  |\varpi^k|_v^{1-2\tau} \chi_{0,v}^{-1}(\varpi^k)
$$
from which we deduce
$$
\int_{F_v^\times} B_v(a;\mathcal{E}_\tau)\cdot  |a|_v^{s-1/2}\ d^\times a =
 \chi_{0,v}(\delta) |\delta|_v^{\tau-1/2} 
L_v(s, \overline{\chi}_0|\cdot|^{1/2-\tau}  ) L_v(s,|\cdot|^{\tau -1/2}_v). 
$$
If $v\mid\mathfrak{s}$ then choose $\sigma\in F_v^\times$ with $\sigma\co_{F,v}=\mathfrak{s}_v$.  We have
\begin{eqnarray*}\lefteqn{
 \int_{F_v}\psi_v^0( y)\int_{F_v^\times} \Omega_v(ax,xy) |x|_v^{2\tau} \chi_{0,v}(x)\ d^\times x\ dy   } \\
 &=&
\int_{F_v^\times}   \left[ \int_{F_v}  \psi_v^0( y) ( yx) \mathbf{1}_{\co_{F,v}^\times} ( yx)  \ dy \right] \mathbf{1}_{\mathfrak{s}_v}(ax) |x|_v^{2\tau} \chi_{0,v}(x)\ d^\times x \\
 &=&
 |\delta \sigma|_v^{1/2} \epsilon_v(\chi_0,\psi_v^0)
 \int_{F_v^\times}  
 \mathbf{1}_{\co_{F,v}^\times}(\sigma x^{-1})
  \mathbf{1}_{\mathfrak{s}_v}(ax) |x|_v^{2\tau-1}  \ d^\times x \\
   &=&
  |\delta |_v^{1/2}  |\sigma|_v^{2\tau-1/2} \epsilon_v(\chi_0,\psi_v^0)\mathbf{1}_{\co_{F,v}}(a).
\end{eqnarray*}
Therefore
$$
\int_{F_v^\times}  B_v(a;\mathcal{E}_\tau)\cdot |a|_v^{s-1/2}\ d^\times a  =
\chi_{0,v}(\delta  ) |\delta |_v^{\tau-1/2}  |\sigma|_v^{2\tau-1/2} \epsilon_v(\chi_0,\psi_v^0)
L_v(s, |\cdot|^{\tau-1/2}).
$$
If $v\mid\infty$ then 
\begin{eqnarray*}\lefteqn{
 \int_{F_v}\psi_v^0( y)\int_{F_v^\times} \Omega_v(ax,xy) |x|_v^{2\tau} \chi_{0,v}(x)\ d^\times x dy   } \\
 &=&
|\delta|_v^{1/2}  \int_{-\infty}^\infty    e^{2\pi i  y}   \int_{-\infty}^\infty    e^{-\pi x^2(a^2+ y^2)} |x|_v^{2\tau-1}  \ d^{\mathrm{Leb}} x d^{\mathrm{Leb}} y.
\end{eqnarray*}
We therefore have
\begin{eqnarray*}\lefteqn{
\chi_{0,v}(\delta^{-1}) |\delta|_v^{1/2-\tau}\int_{F_v^\times}  B_v(a;\mathcal{E}_\tau)\cdot |a|_v^{s-1/2}\ d^\times a  } \\
&=&
  \int_{\R^\times} \int_{-\infty}^\infty \left(  \int_{\R^\times}  |a|^{\tau + s-1/2} e^{-\pi x^2a^2} d^\times a \right) 
 e^{2\pi i  y}    e^{-\pi x^2 y^2} |x|^{2\tau}  \  d^{\mathrm{Leb}} y  \  d^\times x  \\
&=&
 G_1(s+\tau-1/2)    \int_{\R^\times}  \left( \int_{-\infty}^\infty  e^{- 2 \pi iyx}   e^{-  \pi y^2}  \  d^{\mathrm{Leb}} y \right) |x|^{s-\tau +1/2}   \  d^\times x   \\
&=&
G_1(s+\tau-1/2)  \int_{\R^\times}   e^{- \pi x^2}  |x|^{s-\tau +1/2}   \  d^\times x   \\
&=&
 G_1(s+\tau-1/2)  G_1(s-\tau +1/2).
\end{eqnarray*}
Combining these calculations proves the lemma.
\end{proof}

We now  assume that $\Pi_\tau$ satisfies Hypothesis \ref{hyp}, which is really just the  condition that  $\chi_v$ factors through $\mathrm{N}:E_v^\times\map{}F_v^\times$ for each $v\mid\mathfrak{s}$.    Choosing $\Pi=\Pi_\tau$ in the introduction to \S \ref{s:central value}, we wish to prove an analogue  (Corollary \ref{continuous maass answer}) of Theorem \ref{second maass values} for the noncuspidal representation $\Pi_\tau$ by brute force. Note that now  $\mathfrak{m}=\co_F$ and $\epsilon(1/2,\mathfrak{r})=(-1)^{[F:\Q]}$.  To put ourselves in the situation of \S \ref{ss:maass}, suppose $B$ is a split quaternion algebra over $F$ (so that (\ref{B ramification}) holds for all finite $v$) and as always fix an embedding $E\map{}B$.  Let $W$ be a two dimensional $F$-vector space on which $B$ acts on the left, and fix an isomorphism of $F$-vector spaces $W\iso F\times F$.  Writing elements of $W$ as row vectors,  there is an isomorphism  $\rho:B\iso M_2(F)$ determined by  $b\cdot [x,y]=[x,y]\cdot \rho(b)^t $, where  the action on the left is the action of $B$ on $W$, the action on the right is matrix multiplication, and the superscript $t$ indicates transpose. The element  $w_0=[0,1]\in W$ generates $W$ as a left $E$-module, and we define 
$$
L=\co_{\mathfrak{cs}^{-1}}\cdot w_0\hspace{1cm} L'=\co_{\mathfrak{c}}\cdot w_0.
$$
We may pick a $j\in \GL_2(\A)$ having the following properties:
\begin{enumerate}
\item
if $v\mid\mathfrak{s}$ then  $j_v$ satisfies $[0,1]\cdot j^{-1}_v=w_0$  and
$$
L_v= (\co_{F,v}\times\co_{F,v})\cdot j_v^{-1}\hspace{1cm} L_v'=(\mathfrak{s}_v \times\co_{F,v} )\cdot j_v^{-1},
$$
\item
if $v\nmid\mathfrak{s}$ is a finite place of $F$ then $j_f \cdot K_0(\mathfrak{m})\cdot j_f^{-1} = \rho(V_v)$,
\item
if $v$ is an archimedean place then $j_v\cdot \mathrm{SO}(F_v)\cdot j_v^{-1}$ is set of  norm one elements of $\rho(T(F_v))$.
\end{enumerate}
 For every automorphic form $\phi$ on $\GL_2(\A)$ we define an automorphic form  $\phi'$ on $G(\A)$ by $\phi'(g)=\phi(\rho(g)j)$.  The space $\Pi_\tau$ of automorphic forms on $\GL_2(\A)$ thereby determines a space $\Pi'_\tau$ of automorphic forms on $G(\A)$.  Of course $G\iso \GL_2$ and $\Pi_\tau'\iso \Pi_\tau$, but it is useful to maintain these notational distinctions.  Under the definition of \S \ref{ss:toric} $\Pi_\tau$ is  the  Jacquet-Langlands  lift of $\Pi'_\tau$ (a highly degenerate case).  If $\phi\in \Pi_\tau$ is a newvector in the sense of \S \ref{ss:automorphic forms} then $\phi'\in\Pi_\tau'$ is a newvector in the sense of \S \ref{ss:toric}.

\begin{Prop}\label{continuous maass prelim}
Normalize the Haar measures on $T(\A_f)$ and $Z(\A_f)$ to give $\widehat{\co}^\times_{\mathfrak{c}}$ and  $\widehat{\co}^\times_F$ each volume one, respectively, and give  $T(F)\backslash T(\A_f)/Z(\A_f)$ the induced quotient measure.  For every $\tau\in\C$
$$
\mathrm{N}_{F/\Q}\left(\mathfrak{dc}^2 \mathfrak{s}^{-2}  \right)^{\tau/2}   
\frac{1}{2^{[F:\Q]}}  L(\tau,\chi)=
\int_{T(F)\backslash T(\A_f)/Z(\A_f) } \chi(t) \mathcal{E}'_{\tau} (t)\ dt.
$$
\end{Prop}  
  
  \begin{proof}
The restriction of $\mathcal{E}'_\tau$ to $T(\A_f)$ does not depend on the choice of embedding $E\map{}B$, and this embedding may be chosen so that
$$
\rho(\alpha+\beta\sqrt{-\Delta}) = \left(\begin{matrix} \alpha&\beta\Delta\\ -\beta &\alpha \end{matrix}\right)
$$
where   $E=F[\sqrt{-\Delta}]$ with $\Delta\in F$ totally positive.  As the embedding $\rho :T\map{}\GL_2$ identifies $Z(F)\backslash T(F)$ with $B(F)\backslash \GL_2(F)$ we have
$$
\int_{T(F)\backslash T(\A_f)/Z(\A_f) } \chi(t) \mathcal{E}'_{\tau} (t)\ dt  
=
\int_{T(\A_f)/Z(\A_f) } \chi(t) \mathcal{F}_{\tau}( \rho(t) j)\ dt.
$$
Combining this with
$$
\chi(t)\mathcal{F}_\tau(\rho(t) j) = 
  |\det(j)|^{\tau}   \int_{Z(\A)} \Omega \big(   [0,1] \cdot  \rho(tx) j  \big) |\mathrm{N}(tx)|^{\tau} \chi(tx) \ d x  \\
$$
we find
\begin{eqnarray*}
\int_{T(F)\backslash T(\A_f)/Z(\A_f)} \chi(t) \mathcal{E}'_{\tau} (t)\ dt  &=&   
  |\det(j)|^{\tau}  \int_{T(\A_f)} \Omega \big( [0,1] \cdot \rho(t) j \big) |\mathrm{N}(t)|^{\tau} \chi(t) \ dt \\ 
& &
 \cdot \prod_{v\mid\infty} \int_{F_v^\times} \Omega_v \big(  [0,1] \cdot x \big) |x|_v^{2\tau} \chi_{0,v}(x)\ d^\times x.
\end{eqnarray*}

We now compute the right hand side place-by-place. For an archimedean place $v$ we may take $j_v=\left(\begin{matrix}  \sqrt{\Delta_v} & \\ &1 \end{matrix}\right)$ so that 
$$
 \int_{F_v^\times} \Omega_v \big(  [0,1] \cdot x j  \big) |x|_v^{2\tau} \chi_{0,v}(x)\ d^\times x
=  \int_{-\infty}^\infty e^{-\pi x^2} |x|^{2\tau-1} \ d^{\mathrm{Leb}} x.
$$
The integral on the right is  $2^{\tau-1}G_2(\tau)=2^{\tau-1}L_v(\tau,\chi)$. If $v$ is a finite place of $F$ with $v\nmid\mathfrak{s}$ then 
\begin{eqnarray*}
\int_{T(F_v)}\Omega_v([0,1]\cdot \rho(t) j )|\mathrm{N}(t)|_v^{\tau}\chi_v(t)\ d t
&=&
\int_{T(F_v)}\mathbf{1}_{L_v}(\overline{t} \cdot w_0)|\mathrm{N}(t)|_v^{\tau}\chi_v(t)\ d t \\
&=&
\int_{T(F_v)}\mathbf{1}_{\co_{\mathfrak{c},v}}(\overline{t})|\mathrm{N}(t)|_v^{\tau}\chi_v(t)\ d t \\
&=&
\mathrm{Vol}(\co_{\mathfrak{c},v}^\times) \cdot L_v(\tau, \chi),
\end{eqnarray*}
the final equality by the argument  of \cite[p. 238]{zhang3}.  Finally suppose that $v\mid\mathfrak{s}$.    For any $t\in E_v^\times$ the value of  $\Omega_v([0,1]\cdot \rho(t) j)$ is nonzero if and only if 
$[0,1]\rho(t)j$ generates the $\co_{F,v}$-module $(\mathfrak{s}_v\times\co_{F,v})/(\mathfrak{s}_v\times\mathfrak{s}_v)$, and when this is the case $\Omega_v([0,1]\cdot \rho(t) j)=\chi^{-1}_{0,v}(y)$
where $y\in\co_{F,v}^\times$ satisfies $[0,1]\rho(t)j\in [0,y]+\mathfrak{s}_v^2$.  This condition is equivalent to  $t w_0$ being an $\co_{F,w}$-generator of  $L_v' /\mathfrak{s}_v L_v$, in which case  the $y\in\co_{F,v}^\times$ above satisfies $\overline{t} w_0 \in y w_0+\mathfrak{s}_v L_v$.  Thus $y\equiv\vartheta_v(\overline{t})\equiv \vartheta_v(t)\pmod{\mathfrak{s}_v}$ in the notation of \S \ref{ss:special cycles}.  By Lemma \ref{theta factor} $\chi^{-1}_{0,v}(y)=\chi^{-1}_v(t)$.
 As the generators of $\co_{\mathfrak{c},v}/\mathfrak{s}_v\co_{\mathfrak{cs}^{-1},v}$ are exactly the units of $\co_{\mathfrak{c},v}$ we find
\begin{eqnarray*}
\int_{T(F_v)}\Omega_v( [0,1]\cdot \rho(t) j )  \cdot  |\mathrm{N}(t)|_v^{\tau}\chi_v(t)\ d t
&=&
\int_{\co_{\mathfrak{c},v}^\times } \chi_v^{-1}(t) \cdot |\mathrm{N}(t)|_v^{\tau}\chi_v(t)\ d t  \\
&=&  \mathrm{Vol}(\co_{\mathfrak{c},v}^\times).
\end{eqnarray*}

It only remains to compute $\det(j)$.  From the relation 
$$
[(\co_F+\co_F\sqrt{-\Delta}) \cdot w_0]\cdot j^{-1} =\co_{\mathfrak{cs}^{-1}} \cdot w_0
$$
we find 
$$
4\Delta\det(j)^{-2} \co_F=\mathrm{disc}(\co_F+\co_F\sqrt{-\Delta}) \cdot \det (j)^{-2} \co_F=\mathrm{disc}(\co_{\mathfrak{cs}^{-1}} ) =\mathfrak{d}(\mathfrak{c/s})^2.
$$
Using $|\det (j) |_v^2=\Delta_v$ for $v\mid\infty$ we obtain
$2^{[F:\Q]}  |\det(j) | =   \sqrt{ \mathrm{N}(\mathfrak{dc}^2\mathfrak{s}^{-2})}.$
The proposition follows by combining these calculations.
\end{proof}

\begin{Cor}\label{continuous maass answer}
Suppose $\mathrm{Re}(\tau)=1/2$  and let $\phi_\tau\in\Pi_\tau$ be the normalized newvector.  Then
$$
L(1/2,\Pi_\tau\times\Pi_\chi)    = 
   \frac{4^{[F:\Q]}}  { \sqrt{\mathrm{N}_{F/\Q} (\mathfrak{dc}^2 )} }
 \left| \frac{1}{H_{F,\mathfrak{s} }} \int_{C_V }  Q_\chi(g)  \phi'_\tau(g) \ dg \right|^2 
$$
where  $S_V$ is the measure space of \S \ref{ss:maass} defined with $V$ in place of $U$.
\end{Cor}

\begin{proof}
Using $\widehat{\co}_\mathfrak{c}^\times/V_T\iso (\co_F/\mathfrak{s})^\times$, the measures on $T(F)\backslash T(\A_f)/Z(\A_f)$ and $C_V$  are related by
$$
\int_{C_V}  Q_\chi(t)\phi'_\tau(t)  \ dg = H_{F,\mathfrak{s}} \int_{T(F)\backslash T(\A_f)/Z(\A_f)}  Q_\chi(t)\phi'_\tau(t)\ dg
$$
while  Lemma \ref{continuous maass normalization} implies
$$
\epsilon(1/2,\chi_0) \cdot  \phi_\tau= \mathrm{N}_{F/\Q}(\mathfrak{s})^{2\tau-1/2} \cdot  \mathcal{E}_\tau.
$$
The corollary now follows immediately from   $\left|  L(\tau,\chi) \right|^2=   L(1/2,\Pi_\tau\times\Pi_\chi)$, Proposition \ref{continuous maass prelim}, and the fact that the restriction of $Q_\chi$ to $T(\A_f)$ is simply $\chi$.
\end{proof}


\subsection{Descent to low level}
\label{level comparison}


Assume that either $\Pi_v$ is a weight $2$ discrete series at each archimedean $v$ or that $\Pi_v$ is a weight $0$ principal series at each archimedean $v$.    In the weight $2$ case we assume that $\epsilon(1/2,\mathfrak{r})=1$ and $B$ is totally definite, as in \S \ref{ss:holomorphic values}, and in the weight $0$ case we assume that $\epsilon(1/2,\mathfrak{r})=(-1)^{[F:\Q]}$ and $B$ is totally indefinite, as in \S \ref{ss:maass}.    For each $v\mid\mathfrak{dc}$ the representation $\Pi_v$ is isomorphic to a principal series $\Pi(\mu_v,\chi_{0,v}^{-1}\mu_v^{-1})$ with $\mu_v$ unramified, and we set $\alpha_v=\mu_v(\varpi)$ for any uniformizer $\varpi$ of $F_v$.  By the argument of \cite[\S 17]{zhang3} for each $v\mid\mathfrak{dc}$ there are rational functions $\mathbf{a}_v$, $\mathbf{b}_v$, $\mathbf{c}_v$ which, crucially, depend only on the data $(F_v,E_v,\chi_v)$ and not on the representation $\Pi$, such that 
$$
\widehat{B}(\co_F;\phi_\Pi^\#)  =  \widehat{B}(\co_F;\phi_\Pi)\cdot \prod_{v\mid\mathfrak{dc}} \mathbf{a}_v(\alpha_v)
$$
and
$$
  || \phi_\Pi^\# ||_{ K_0(\mathfrak{dr}) }  = 
 ||\phi_\Pi||_{K_0(\mathfrak{n})}^2  \cdot \prod_{v\mid\mathfrak{dc}} \mathbf{b}_v(\alpha_v)
$$
where  $\phi_\Pi\in \Pi$ is the normalized newvector and $\phi_\Pi^\#\in\Pi$ is the projection of $\phi_\Pi$ to the quasi-new line.  Using (\ref{cycle trace}) in place of \cite[Lemma 17.2]{zhang3}, the rational function $\mathbf{c}_v$ is defined by the relation
$$
\frac{1}{||\phi_{\Pi'}||^2} \left|  \int_{C_U}  Q_\chi(g) \phi_{\Pi'}(g)\ dg \right|^2  
=
\frac{1}{||\phi^\chi_{\Pi'}||^2} \left|     \int_{C_U}  P_\chi(g) \phi_{\Pi'}^\chi(g)\ dg \right|^2  \cdot \prod_{v\mid\mathfrak{dc}} \mathbf{c}_v(\alpha_v)
$$
where $\Pi'$ is the automorphic representation of $G(\A)$ whose Jacquet-Langlands lift is $\Pi$, $\phi_{\Pi'}^\chi$ is a toric newvector in $\Pi'$ in the sense of Definition \ref{JL-toric},  $\phi_{\Pi'}\in \Pi'$ is a newvector in the sense of Definition \ref{JL-newvector}, and $||\cdot||$ is any $G(\A)$-invariant norm on $\Pi'$ (e.g. $||\cdot||^2=\int_{S_U}|\cdot|^2$).  If $v\nmid\mathfrak{s}$ then $\chi_{0,v}$ is unramfied and we must have $\mathbf{a}_v(\alpha_v)=\mathbf{a}_v(\alpha_v^{-1}\chi_{0,v}^{-1}(\varpi))$ due to the the isomorphism $\Pi(\mu_v,\chi_{0,v}^{-1}\mu_v^{-1})\iso \Pi(\chi_{0,v}^{-1}\mu_v^{-1},\mu_v)$, and similarly for $\mathbf{b}_v$ and $\mathbf{c}_v$.   Set 
$\mathbf{a}_\Pi=\prod_{v\mid\mathfrak{dc}} \mathbf{a}_v(\alpha_v)$
and define $\mathbf{b}_\Pi$ and $\mathbf{c}_\Pi$ similarly.  Proposition \ref{first holomorphic values} (for the weight $2$ case) and Proposition \ref{first maass values} (for the weight $0$ case) give
\begin{eqnarray*}  \lefteqn{
2^{|S|}    H_F\lambda_U^{-1} [\widehat{\co}_E^\times: U_T]    \widehat{B}(\co_F;\phi_\Pi^{\#}  )
 L(1/2,\Pi\times\Pi_\chi)  }\\
&=&
|d|^{1/2} 2^{f\cdot [F:\Q]}   ||\phi_\Pi^{\#} ||_{K_0(\mathfrak{dr})} ^2  \cdot 
 \frac{ \big|   \int_{C_U} P_\chi(g) \phi_{\Pi'}^\chi(g)\ dg \big|^2  }  {    \int_{S_U} |\phi^\chi_{\Pi'}(g)|^2\ dg  } . \nonumber
\end{eqnarray*}
where $f=1$ in the weight $2$ case and $f=2$ in the weight $0$ case. As $\widehat{B}(\co_F,\phi_\Pi)=1$ we find, using $\lambda_V H_{F,\mathfrak{s}}=H_F[\widehat{\co}^\times_\mathfrak{c} :V_T]$ and (\ref{measure descent}) (which holds also with $C_U$ and $C_V$ replaced by $S_U$ and $S_V$), that 
\begin{equation}\label{descent two}
\kappa \cdot \mathbf{a}_\Pi\mathbf{c}_\Pi    \cdot  \frac{L(1/2,\Pi\times\Pi_\chi) }  {  ||\phi_\Pi ||_{K_0(\mathfrak{n})}^2 }   =    \frac{ \mathbf{b}_\Pi \cdot 2^{f\cdot [F:\Q]} }{ H_{F,\mathfrak{s} }  \sqrt{\mathrm{N}_{F/\Q}(\mathfrak{dc}^2)}}  \cdot  \frac{ \left|  \int_{C_V} Q_\chi(t) \phi_{\Pi'}(t)\ dt \ \right|^2  }  {  \int_{S_V} |\phi_{\Pi'}(g)|^2\ dg  } . 
\end{equation}
Here  $\kappa=\prod_{v\mid\mathfrak{dc}}\kappa_v$ with
$$
\kappa_v= \frac{[\co_{E,v}^\times: U_{T,v}] }{ [\widehat{\co}^\times_{\mathfrak{c},v} : V_{T,v} ]  }
\cdot \frac{|c|_v}{  [V_v:U_v]  }  \left\{\begin{array}{ll} 2 &\mathrm{if\ }v\mid\mathfrak{d} \\ 1 &\mathrm{if\ }v\mid\mathfrak{c} \end{array}\right.
$$
where $c\in\A^\times$ satisfies $c\co_F=\mathfrak{c}$.

\begin{proof}[Proof of Theorems \ref{second holomorphic values} and \ref{second maass values}]
It follows from the definition of the quasi-new line that $\phi_\Pi^\#\not=0$ (in the notation of \S \ref{ss:quasi-new} we have $\Lambda_v(\phi_{\Pi,v})\not=0$ for each $v\mid\mathfrak{dr}$, and so $\phi_{\Pi,v}$ has nontrivial projection to the quasi-new line in $\Pi_v$), and hence $\mathbf{b}_\Pi\not=0$.  It therefore suffices by (\ref{descent two}) to prove that  $\kappa \cdot \mathbf{a}_\Pi \mathbf{c}_\Pi=  \mathbf{b}_\Pi$.    Let us suppose for the moment that $\Pi$ is of parallel weight $0$ and that $\mathfrak{m}=\co_F$.  Thus $\epsilon(1/2,\mathfrak{r})=(-1)^{[F:\Q]}$ and we are in the  situation of \S \ref{ss:maass}.  The quaternion algebra $B$ is split, and we let $\rho:G\iso \GL_2$ and   $j\in \GL_2(\A)$
be as in \S \ref{ss:maass family}.  Set $\Pi'=\Pi$ and for each $\phi\in \Pi$ set $\phi'(g)=\phi(\rho(g) j)$.
  Fix a Haar measure on $\GL_2(\A_f)$ and, as always, normalize the Haar measure on $Z(\A_f)$ to give $\widehat{\co}_F^\times$ volume $1$.   Define a Haar measure on $G(\A_f)$ by demanding that $\rho$ be an isomorphism of measure spaces.   For any $\phi\in\Pi$ we now have,  tediously keeping track of the normalizations of measures,
\begin{eqnarray*}
\int_{S_V} |\phi' |^2 
&=&  \mathrm{Vol}(V)^{-1}   \int_{G(F)\backslash  X\times G(\A_f)/ V } |\phi' |^2 \\
&=& \mathrm{Vol}(V)^{-1} \frac{1}{ [Z(F)\cap \widehat{\co}_F^\times: Z(F) \cap V ] }  \int_{G(F)\backslash  X\times G(\A_f)/ \widehat{\co}_F^\times } |\phi' |^2 \\
&=& \mathrm{Vol}(V)^{-1}   \frac{ [Z(\A_f):Z(F) \widehat{\co}_F^\times ]}{ [Z(F)\cap \widehat{\co}_F^\times: Z(F) \cap V ]   }   \int_{ G(F)\backslash       X\times G(\A_f)/ Z(\A_f) } |\phi'|^2.
  \end{eqnarray*}
  Using  $j K j^{-1}=\rho(V)$ and $V_Z=\{ x\in \widehat{\co}_F^\times\mid x\in 1+\widehat{s}\}$ we find that 
  $$
  \int_{S_V} |\phi' |^2 = H_F\lambda_V^{-1} ||\phi||_K^2 = H_{F, \mathfrak{s}} ||\phi||_{K_0(\mathfrak{n})}^2.
  $$
  We may now write (\ref{descent two}) as 
\begin{equation}\label{descent three}
\kappa \cdot \mathbf{a}_\Pi\mathbf{c}_\Pi    \cdot  L(1/2,\Pi\times\Pi_\chi)  =    \frac{ \mathbf{b}_\Pi \cdot 2^{f\cdot [F:\Q]} }{    \sqrt{\mathrm{N}_{F/\Q}(\mathfrak{dc}^2)}}  \cdot  \left| \frac{1}{H_{F,\mathfrak{s}} } \int_{C_V} Q_\chi(t) \phi_\Pi'(t)\ dt \ \right|^2  .
\end{equation}
 The point is that in this formulation no $L^2$ norms appear, and the statement of the formula makes sense even if $\Pi$ is noncuspidal.  The argument of 
\cite[\S 18]{zhang3} shows that the equality (\ref{descent three}) can be extended to the principal series  representation $\Pi_\tau$ of \S \ref{ss:maass family} for any $\tau\in\C$ with $\mathrm{Re}(\tau)=1/2$ (so that $\Pi_\tau$ is unitary), provided that $\chi$ does not factor through the norm map $\A_E^\times\map{}\A^\times$ (so that $\Pi_\chi$ is cuspidal by Lemma \ref{cusp theta} and (\ref{maass rankin}) still holds).

If for each $v\mid\mathfrak{dc}$ we let $q_v$ denote the cardinality of the residue field of $v$, then taking $\Pi=\Pi_\tau$ and $\phi_\Pi=\phi_\tau$ in  (\ref{descent three}) and comparing with Lemma \ref{continuous maass answer} (and still assuming that $\Pi_\chi$ is cuspidal) gives
$$
\prod_{v\mid\mathfrak{dc}}\kappa_v \mathbf{a}_v(q_v^{1/2-\tau}) \mathbf{c}_v(q_v^{1/2-\tau})
= \prod_{v\mid\mathfrak{dc}}  \mathbf{b}_v(q_v^{1/2-\tau}).
$$
As in the proof of \cite[Proposition 19.2]{zhang3}, letting $\tau$ vary and letting $\chi$ vary over characters which do not factor through the norm while holding the components $\chi_v$ for $v\mid\mathfrak{dc}$ fixed, we find the equality of rational functions
$\kappa \prod \mathbf{a}_v\mathbf{c}_v=\prod \mathbf{b}_v$ where each product is over all $v\mid\mathfrak{dc}$.
\end{proof}


\section{Central derivatives}
\label{derivatives}


In this section we relate the N\'eron-Tate heights of certain CM points on Shimura curves to derivatives of automorphic $L$-functions.  As in \cite{zhang2} the method is to compute the arithmetic intersection pairings of various CM-divisors and compare these intersection multiplicities to the Whittaker coefficients of the automorphic form $\Phi_\mathfrak{r}$ of \S \ref{ss:derivative}.  These intersection multiplicities decompose as a  sum of local intersection multiplicities, and the calculations of \S 5 and \S 6 of \cite{zhang2}  show that the calculation of  local multiplicities can be reduced to the calculation of  linking numbers of CM-cycles on totally definite quaternion algebras.     Fortunately for us, this reduction step is done in \cite{zhang2}  in a very general context, and includes not only on Shimura curves with arbitrary level structure but also  Shimura curves associated to the algebraic group $G$ below (as opposed to the group $G/Z$).  Thus we may cite from Zhang the crucial  Propositions \ref{Zhang's intersection I} and   \ref{bad zhang intersection} below, which reduce the local intersection theory at nonsplit primes to the calculations we have done in  \S \ref{quaternion generalities}.

Throughout \S \ref{derivatives} we assume that the representation $\Pi$ of \S \ref{notations} satisfies Hypothesis \ref{hyp}  and  that $\Pi_v$ lies in the discrete series of weight $2$ for every archimedean $v$.  Set $\mathfrak{r}=\mathfrak{mc}^2$ and assume that $\omega(\mathfrak{m})=(-1)^{[F:\Q]-1}$.  The epsilon factor of \S \ref{ss:kernel} then satisifies $\epsilon(1/2,\mathfrak{r})=-1$ and so  $L(1/2,\Pi\times\Pi_\chi)=0$ by the functional equation (\ref{kernel functional}) and the Rankin-Selberg integral representation (\ref{rankin integral}) with $b=1$.  Fix an archimedean place $w_\infty$ of $F$ and let $B$ be the  quaternion algebra over $F$ characterized by
 $$
B_v \mathrm{\ is\ split\ }\iff \epsilon_v(1/2,\mathfrak{r},\psi)=1\mathrm{\ or\ }v= w_\infty
 $$
for every  place $v$.  Thus  $B$ is indefinite at $w_\infty$ and definite at all other archimedean places. The reduced discriminant of $B$ divides $\mathfrak{m}$ and, as $E_v$ is a field whenever $B_v$ is nonsplit,  there is  an embedding $E\map{}B$ which we fix.  Let $G$, $T$, and $Z$ be the algebraic groups over $F$ defined at the beginning of \S \ref{quaternion generalities}.  For any ideal $\mathfrak{b}\subset \co_F$ let $\co_{\mathfrak{b}}=\co_F+\mathfrak{b}\co_{E}$ denote the order of $\co_E$ of conductor $\mathfrak{b}$. Fix an algebraic closure $F^\alg$ of $F$ containing $E$ and an embedding $F^\alg\hookrightarrow \C$ lying above $w_\infty$.

 General references for Shimura curves include \cite{carayol, milne1, milne2, nek-euler, zhang1, zhang2}.


\subsection{Shimura curves}
\label{shimura curves}


 Throughout \S \ref{shimura curves} we let $U$ be an arbitrary compact open subgroup of $G(\A_f)$.  The  chosen embedding $E\map{}\C$ determines an isomorphism of real algebraic groups $\mathbb{S}\iso T\times_F\R$, where $\mathbb{S}=\mathrm{Res}_{\C/\R}\mathbb{G}_m$.  The embedding $T\map{}G$ therefore determines an embedding of real algebraic groups
$$
x_0:\mathbb{S}\map{} G\times_F\R\map{} (\mathrm{Res}_{F/\Q} G)\times_\Q \R.
$$ 
Let  $X$ be the $G(\R)$-conjugacy class of $x_0$ in the set of all such embeddings.  If $F\not=\Q$ or if $B\not\iso M_2(F)$ we define a compact Riemann surface
\begin{equation}\label{complex points}
X_U(\C) = G(F)\backslash X\times G(\A_f) /U.
\end{equation}
For $x\in X$ and $g\in G(\A_f)$ let $[x,g]$ denote the image of $(x,g)$ in $X_U(\C)$.
If $F=\Q$ and $B$ is split then the right hand side of (\ref{complex points}) is noncompact, and $X_U(\C)$ is defined as the usual compactification of the right hand side obtained by adjoining finitely many cusps.   The connected components of $X_U(\C)$ are indexed by the set
$$
Z_U(\C) = Z(F)^+\backslash Z(\A_f)/\mathrm{N}(U)
$$
where $Z(F)^+\subset Z(F)\iso F^\times$ is the subgroup of totally positive elements and $\mathrm{N}(U)$ is the image of $U$ under the reduced norm $G(\A_f)\map{}Z(\A_f)$.  The canonical map $X_U(\C)\map{}Z_U(\C)$ is given by $[x,g]\mapsto \mathrm{N}(g)$.

Let $X_U$ denote Shimura's canonical model of $X_U(\C)$ over $\Spec(F)$.  Let $F_U/F$ be the abelian extension of $F$ which, under the reciprocity map of class field theory, has $\Gal(F_U/F)\iso Z_U(\C)$. The component map $X_U(\C)\map{}Z_U(\C)$ arises from a morphism of $F$-schemes $X_U\map{}Z_U$ where $Z_U$ is (noncanonically) isomorphic to $\Spec(F_U)$.   For each geometric point $\alpha:\Spec(F^\alg)\map{} Z_U$ define a smooth connected projective curve over $F^\alg$
$$
X_U^\alpha= X_U\times_{Z_U} \Spec(F^\alg).
$$
The \emph{Jacobian} $J_U$ of $X_U$ is the abelian variety over $F$ defined by
$$
J_U=\mathrm{Res}_{Z_U/F} (\mathrm{Pic}^0_{X_U/Z_U})
$$
so that the geometric fiber of $J_U$ decomposes as 
$$
J_U\times_F F^\alg \iso \prod_{\alpha\in Z_U(F^\alg)} J_U^\alpha
$$
where $J_U^\alpha$ is the Jacobian of $X_U^\alpha$.  There is a $\Gal(F^\alg/F)$ invariant function 
$$
\mathrm{Hg}:X_U(F^\alg)\map{}J_U(F^\alg)\otimes_\Z\Q,
$$ 
the \emph{Hodge embedding}, described in detail in \cite[\S 3.5]{cornut-vatsal}.  Briefly,   Zhang \cite[\S 6.2]{zhang2} constructs the \emph{Hodge class}  $\mathcal{L}\in \mathrm{Pic}(X_U)\otimes_\Z\Q$ having degree $1$ on every geometric component.  Each  $P\in X_U(F^\alg)$ determines a geometric point $\alpha\in Z_U(F^\alg)$, and we let $\mathcal{L}_P$ denote the restriction of $\mathcal{L}$ to $X_U^{\alpha}$.  Letting $\co(P)\in\mathrm{Pic}(X_U\times_F F^\alg)$ denote the class of $P$ we define 
$$
\mathrm{Hg}(P)=\co(P)\otimes\mathcal{L}_P^{-1}\in J_U^{\alpha}(F^\alg)\otimes_\Z\Q.
$$

 For any finite extension $L/F$ the N\'eron-Tate height on $J_U(L)$ is denoted by $\langle\cdot,\cdot\rangle^{\mathrm{NT}}_{U,L}$.  The \emph{normalized} N\'eron-Tate height on $J_U(F^\alg)$ is defined by
$$
\langle x,y\rangle^{\mathrm{NT}}_U=\frac{1}{[L:F]}\langle x,y\rangle^{\mathrm{NT}}_{U,L}
$$
where $L$ is any finite extension of $F$ large enough that $x$ and $y$ are defined over $L$.  Fix two points $P,Q\in X_U(F^\alg)$ and choose a finite Galois extension $L/F$ large enough that $P$ and $Q$ are both defined over $L$.   To compute the N\'eron-Tate pairing of $\mathrm{Hg}(P)$ and $\mathrm{Hg}(Q)$ we  use the arithmetic intersection theory of  Gillet-Soul\'e \cite{gillet-soule1, gillet-soule2} as in \S 5.3 and \S 6.1 of \cite{zhang2}.  Suppose  that $U$ is small enough that $X_U$ admits a canonical regular model $\underline{X}_U$, proper and flat over $\co_F$, as in \cite[\S 1.2.5]{zhang1}.  Let $\underline{Z}_U$ be the normalization of $\Spec(\co_F)$ in $Z_U$, so that $\underline{Z}_U\iso \Spec(\co_{F_U})$ (noncanonically) and the component map $X_U\map{}Z_U$ extends to a map of $\co_F$-schemes $\underline{X}_U\map{}\underline{Z}_U$.  As $Z_U(L)\not=\emptyset$ there are $[F_U:F]$ distinct embeddings $F_U\map{}L$, and so $[F_U:F]$ distinct morphisms $\Spec(\co_L) \map{}\underline{Z}_U$.  Let  $\mathcal{Z}_U$ denote the disjoint union of $[F_U:F]$ copies of $\Spec(\co_L)$ so that $\mathcal{Z}_U$ is naturally an $\co_L$-scheme which admits an $\co_F$-morphism  $\mathcal{Z}_U\map{}\underline{Z}_U$. Let $\mathcal{X}_U$ be the minimal resolution of singularities of the $\co_L$-scheme $\underline{X}_U\times_{\underline{Z}_U} \mathcal{Z}_U.$  The scheme $\mathcal{X}_U$ has generic fiber $X_U\times_F L$ and is a disjoint union of $[F_U:F]$ proper and flat curves over $\co_L$ indexed by $Z_U(F^\alg)$, each with geometrically connected generic fiber.  The Hodge class $\mathcal{L}$ on  $X_U$ admits a natural extension to  $\underline{X}_U$  \cite[\S 4.1.4]{zhang1} which we  pull back to a class  $\mathcal{L}\in \mathrm{Pic}(\mathcal{X}_U)\otimes_\Z\Q$.  For each embedding $i:L\map{}\C$ the Riemann surface $(\mathcal{X}_U\times_{\co_L}\C)(\C)$  has a canonical volume form $\mu$ which on each connected component has total volume $1$ and whose pull back to the upper half-plane (under any such parametrization) is a multiple of the hyperbolic volume form $y^{-2}dxdy$.  By \cite[Theorem I.4.2]{lang-arakelov} there is a Hermitian metric $\rho_i$, unique up to scaling, on the pull-back of $\mathcal{L}$ to $\mathcal{X}_U\times_{\co_L}\C$ whose Chern form is $\mu$.  Letting $\rho$ denote the tuple $(\rho_i)$ indexed by embeddings $i$ as above, the pair $\widehat{\mathcal{L}}=(\mathcal{L},\rho)$ is then an element of $\widehat{\mathrm{Pic}}(\mathcal{X}_U)$ as in \cite[\S 6.1]{zhang2}.

Going back to the point $P\in X_U(L)$, let $\mathcal{X}_U^\alpha$ be the connected component of $\mathcal{X}_U$ containing $P$.  The \emph{arithmetic closure} (as in \cite[\S 6.1]{zhang2} or \cite[\S 9]{zhang3}) $\widehat{P}\in\widehat{\mathrm{Div}}(\mathcal{X}_U)$ of $P$ with respect to $\widehat{\mathcal{L}}$ is a pair $\widehat{P}=(\mathcal{P}+D_P,g_P)$ where $\mathcal{P}$ is the Zariski closure of $P$ on $\mathcal{X}_U$ and $g_P=(g_{P,i})$ is a tuple indexed by embeddings $i:L\map{}\C$ with $g_{P,i}$  a smooth function on the complement of $P$ in $(\mathcal{X}_U\times_{\co_L}\C)(\C)$ such  that $2\cdot g_{P,i}$ is a Green's function for $P$ with respect to $\mu$ (in the sense of  \cite[\S II.1]{lang-arakelov}) on the component indexed by $\alpha$, and is identically $0$ on the other components.  Lang and Zhang use different normalizations for Green's functions, hence the factor of $2$; our $g_P$ is Zhang's $g(P,\cdot)$.  Finally $D_P$ is a vertical divisor on $\mathcal{X}^\alpha_U$ chosen so that $\mathcal{P}+D_P$ has trivial intersection multiplicity with every vertical divisor, and so that for any finite place $w$ of $L$ the restriction of $\mathcal{L}$ to the sum of the components of $D_P$ above $w$ has degree $0$.  One defines $\widehat{Q}=(\mathcal{Q}+D_Q,g_Q)$ in the same way. The Hodge index theorem now tells us that 
$$
\langle \mathrm{Hg}(P),\mathrm{Hg}(Q)\rangle^{\mathrm{NT}}_{U}= \frac{-1}{[L:F]} \langle \widehat{P}-\widehat{\mathcal{L}}_P,\widehat{Q}-\widehat{\mathcal{L}}_Q\rangle_{\mathcal{X}_U}^{\mathrm{Ar}}
$$
where $\widehat{\mathcal{L}}_P$ is the restriction of $\widehat{\mathcal{L}}$ to the component of $\mathcal{X}_U$ containing $P$ (and similarly with $P$ replaced by $Q$) and the pairing on the right is the Gillet-Soule arithmetic intersection pairing on $\widehat{\mathrm{Pic}}(\mathcal{X}_U)$ defined by \cite[(9.3)]{zhang3}.

For each place $w$ of $F$ fix an extension $w^\alg$ to $F^\alg$.  As we assume that $P\not=Q$ there is a decomposition of the arithmetic intersection pairing as a sum of local Green's functions
$$
\langle \widehat{P},\widehat{Q}\rangle^{\mathrm{Ar}}_{\mathcal{X}_U}=
  \sum_w\sum_{\sigma\in\Gal(L/F) }   d_w\cdot g(P^\sigma,Q^\sigma)_{U,w^\alg}
$$
where the sum is over all places of $F$ and terms on the right are as follows.  If $w\mid\infty$ then $d_w=1$ and $g(P,Q)_{U,w^\alg}=g_{P,i}(Q)$ where $i:L\map{}\C$ is the embedding determined by $w^\alg$.  If $w$ is nonarchimedean then $d_w=\log q_w$ where $q_w$ is the size of the residue field of $w$, and
$$
g(P,Q)_{U,w^\alg}= e(L_{w^\alg}/F_w)^{-1} i_{w^\alg}(\mathcal{P}+D_P,\mathcal{Q}+D_Q)_{\mathcal{X}_U}
$$  
where $e(L_{w^\alg}/F_w)$ is the ramification index and $i_{w^\alg}(\cdot,\cdot)_{\mathcal{X}_U}$ is the intersection pairing on $\mathcal{X}_U\times_{\co_L}\co_{L,w^\alg}$ defined in \cite[III.2]{lang-arakelov}  for divisors with no common components and extended in \cite[III.3]{lang-arakelov} to divisors with common vertical components.  The Green's function $g(P,Q)_{U,w^\alg}$ does not depend on the choice of $L$ and extends bi-additively to a Hermitian pairing on divisors with complex coefficients on $X_U\times_F F^\alg$ having disjoint support.

If $U$ is not sufficiently small in the sense of \cite[\S 1.2.5]{zhang1} then choose $U'\subset U$ which is sufficiently small and define
$$
g(P,Q)_{U,w^\alg}=\frac{1}{\deg(\pi)} g(\pi^*P,\pi^*Q)_{U',w^\alg}
$$
where $\pi:X_{U'}\map{}X_U$ is the degeneracy map with $\deg(\pi)=[F^\times U:F^\times U']$.  This does not depend on the choice of sufficiently small $U'$.


\subsection{Special cycles and Hecke correspondences}


For the remainder of \S \ref{derivatives} we let $U$ and  $V$ denote the compact open subgroups of $G(\A_f)$ constructed in \S \ref{ss:special cycles} and recall that we constructed there CM cycles $P_\chi$ and  $P_{\chi,\mathfrak{a}}$ of level $U$ (for $\mathfrak{a}$ any ideal of $\co_F$ prime to $\mathfrak{c}$) and a CM cycle $Q_\chi$ of level $V$. Let $\epsilon_v\in B_v$ be the element of Lemma \ref{choice of epsilon} used in the construction of $U$, and note that $U_v$ is a maximal compact open subgroup of $G(F_v)$ for $v\nmid\mathfrak{dr}\infty$.   For $\mathfrak{a}$ prime to $\mathfrak{dr}$ there are algebraic Hecke correspondence $T_{\mathfrak{a}}^\Pic$ and $T_{\mathfrak{a}}^\Alb$ on $X_U$ characterized by their action on points of $X_U(\C)$
$$
T_{\mathfrak{a}}^\Pic    [x,g] = \sum_{h\in  U\backslash H(\mathfrak{a})  } [x,gh^{-1}] 
\hspace{1cm} 
T_{\mathfrak{a}}^\Alb    [x,g] = \sum_{h\in  H(\mathfrak{a})/U  } [x,gh],
$$
where $H(\mathfrak{a})$ was defined in \S \ref{ss:special cycles}.  We also have diamond automorphisms of $X_U$ defined by
$$
\langle \mathfrak{a}\rangle^\Pic  [x,g]=[x, g a^{-1} ]
\hspace{1cm}
\langle \mathfrak{a}\rangle^\Alb  [x,g]=[x, g a ]
$$ 
where $a\in\A^\times$ satisfies $a\co_F=\mathfrak{a}$ and $a_v=1$ for $v\mid\infty$.  Restricting  $T_{\mathfrak{a}}^\Pic$, $T_{\mathfrak{a}}^\Alb$ and the diamond automorphisms to divisors on $X_U$ which have degree zero on every geometric component  we obtain endomorphisms, denoted  the same way,  of $J_U$.

We view the set of CM points of level $U$ on $G$ as a subset of $X_U(\C)$ using the injection  $C_U\map{}X_U(\C)$ defined by $T(F) g U\mapsto [x_0,g]$.  By Shimura's reciprocity law \cite[\S 12]{milne2} all points of $C_U$ are defined over the maximal abelian extension of $E$ in $\C$ and satisfy
$$
[x_0,g]^\sigma =[x_0,t^{-1}g]
$$
where $\sigma=[t ,E]$ is the arithmetic Artin symbol of $t$ as in \cite[\S 5.2]{shimura}.  Any CM-cycle $P$ of level $U$ can be written as a sum of characteristic functions of CM points, and so can be viewed as a divisor (with complex coefficients) on $X_U\times_F F^\alg$ in an obvious way.    Setting $P=[x_0,1]$ we then have
$$
P_\chi=\sum_{t\in T(F)\backslash T(\A_f)/U_T} \overline{\chi(t)} \cdot P^{[t,E]}.
$$
This divisor is  rational over the abelian extension $E_\chi/E$ cut out by $\chi$.  As divisors on $X_U\times_F E_\chi$ we have $T_{\mathfrak{a}}^\Pic P_\chi=P_{\chi,\mathfrak{a}}$ and $\langle \mathfrak{a}\rangle^\Pic P_\chi=\chi_0(\mathfrak{a}) P_\chi$.

For $\mathfrak{a}$ prime to $\mathfrak{dr}$ let $P^0_{\chi,\mathfrak{a}}$ denote the restriction of $P_{\chi,\mathfrak{a}}$ to the complement of the image of $T(\A_f)\map{}C_U$.  In particular $P^0_{\chi,\mathfrak{a}}$ and $P_\chi$ have disjoint support.  Fix $a\in\A^\times$ with $a\co_F=\mathfrak{a}$ and define
$$
r_\chi(\mathfrak{a})=\prod_{v\nmid\infty}|a|_v^{-1/2}B_v(a;\theta).
$$
We note that $r_\chi$ is a derivation of $\Pi_{\chi}\otimes|\cdot |^{1/2}$ in the sense of \cite[Definition 3.5.3]{zhang2}.  Exactly as in \cite[Lemma 6.2.1]{zhang2}, (using our Corollaries \ref{unr final degen} and \ref{ram final degen} to evaluate $P_{\chi,\mathfrak{a}}(1)$ instead of \cite[Lemma 4.2.1]{zhang2}) we have
\begin{equation}\label{make disjoint}
P_{\chi,\mathfrak{a}}=P^0_{\chi,\mathfrak{a}}+r_\chi(\mathfrak{a})\cdot P_\chi.
\end{equation}


\subsection{Intersections at nonsplit primes away from $\mathfrak{dr}$}
\label{ss:good intersections}


 Suppose $w\nmid\mathfrak{dr}$ is a finite place of $F$ which is inert in $E$ and fix a place $w^\alg$ of $F^\alg$ above $w$.  Note that the quaternion algebra $B_w$ is split and, as $R_w=\co_{E,w}+\co_{E,w}\epsilon_w$ is a maximal order of $B_w$,  $U_w=R_w^\times$ is a maximal compact open subgroup $G(F_w)$. We wish to compute  $g(P_\chi,P^0_{\chi,\mathfrak{a}})_{U,w^\alg}$.  Let $\tilde{B}$ be the totally definite quaternion algebra obtained from $B$ by interchanging invariants at $w_\infty$ and $w$.  That is, $\tilde{B}$ is defined by $\{\mathrm{places\ }v \mathrm{\ of\ }F \mid \tilde{B}_v\not\iso B_v \} = \{ w,w_\infty\}.$ As $E_v$ is a field for every place $v$ at which $\tilde{B}$ is nonsplit, we may fix an embedding $E\map{} \tilde{B}$.   Denote by $\tilde{G}$ the algebraic group over $F$ defined by $\tilde{G}(A)=(\tilde{B}\otimes_F A)^\times$.

 For each finite place $v\not=w$ fix an isomorphism $\sigma_v:G(F_v)\iso \tilde{G}(F_v)$ compatible with the embeddings of $T(F_v)$ into $G(F_v)$ and $\tilde{G}(F_v)$ and define
$$
\tilde{\epsilon}_v=\sigma_v(\epsilon_v) 
\hspace{1cm}
\tilde{U}_v=\sigma_v(U_v).
$$
Pick $\tilde{\epsilon}_w\in \tilde{B}_w$ so that $E_w\tilde{\epsilon}_w=\tilde{B}_w^-$ and 
$\ord_w(\tilde{\mathrm{N}}(\tilde{\epsilon}_w))=1$, where $\tilde{\mathrm{N}}$ is the reduced norm on $\tilde{B}_w$.  
Then $\tilde{R}_w=\co_{E,w}+\co_{E,w}\tilde{\epsilon}_w$ is the unique maximal order in $\tilde{B}_w$, and we define $\tilde{U}_w=\tilde{R}_w^\times$.  
Define a function $\sigma_w:G(F_v)\map{}\tilde{G}(F_w)/\tilde{U}_w$ by $\sigma_w(g)=\tilde{g}\tilde{U}_w$ for any $\tilde{g}\in \tilde{G}(F_w)$ satisfying
 $ \ord_w(\mathrm{N}(g))=\ord_w(\tilde{\mathrm{N}}(\tilde{g})).$  Set $\tilde{U}=\prod_v\tilde{U}_v$, a compact open subgroup of $\tilde{G}(\A_f)$.   Taking the product of the $\sigma_v$ we obtain a map of left $T(\A_f)$-sets  $\sigma : G(\A_f)/U \map{  }  \tilde{G}(\A_f)/\tilde{U}$
and a push-forward map $f\mapsto\sigma_*f$ from finitely supported functions on $ G(\A_f)/U $ to finitely supported functions on $\tilde{G}(\A_f)/\tilde{U}$ defined by 
$$
 (\sigma_*f)(x)=\sum_{ \sigma(y)=x } f(y).
 $$
As the natural projection $ G(\A_f)/U \map{  }C_U$ has finite fibers, any CM-cycle of level $U$ may be viewed as a finitely supported function on $G(\A_f)/U$.  The push-forward is then a left $T(F)$-invariant function on $\tilde{G}(\A_f)/\tilde{U}$, and so there is an induced push-forward $\sigma_*$ from CM-cycles on $G$ of level $U$ to CM-cycles on $\tilde{G}$ of level $\tilde{U}$.

 Fix a uniformizer $\varpi$ of $F_w$ and for each $k\ge 0$ let $A_k= \co_{F,w}+\varpi^k\co_{E,w}$.  For each $x\in C_U$ define the $w$-\emph{conductor} of $x=T(F)gU$ to be the integer $k$ determined by
$$
A_k^\times=g_wU_w g_w^{-1}  \cap T(F_w).
$$

\begin{Prop}\label{Zhang's intersection I} 
Suppose that $P$ and $Q$ are disjoint CM-cycles of level $U$ with $P$ supported on points of $w$-conductor $k$ and $Q$ supported on points of $w$-conductor $0$. Then 
$$
g(P, Q )_{U, w^\alg}=
\sum_{\gamma\in T(F)\backslash \tilde{G}(F)/ T(F)}  
\langle \sigma_*P , \sigma_*Q \rangle_{\tilde{U}}^\gamma  \cdot M_k(\gamma)
$$
 where  
 $$
M_k(\gamma)= \left\{\begin{array}{ll}
  \frac{ \ord_w(\xi\varpi)}{2} &\mathrm{if\ }k=0\mathrm{\ and\ }\xi\not=0 \\
  0& \mathrm{if\ }k=0\mathrm{\ and\ }\xi=0 \\
 { [\co_{E,w}^\times: A_k^\times]^{-1}} & \mathrm{if\ }k>0.
 \end{array}\right.
$$
\end{Prop}

\begin{proof}
See Lemmas 5.5.2 and 6.3.5 of \cite{zhang2}.
\end{proof}

Suppose $\mathfrak{a}$ is an ideal of $\co_F$ prime to $\mathfrak{dr}$.
For any finite place $v$ we may replace $B_v$ by $\tilde{B}_v$ and $\epsilon_v$ by $\tilde{\epsilon}_v$ everywhere in \S \ref{ss:unramified local calculations} and \S \ref{ss:ramified local calculations}, giving a function  $\tilde{P}_{\chi,\mathfrak{a},v}$ on $\tilde{G}(F_v)/\tilde{U}_v$.  Taking the product over all finite $v$ gives a CM-cycle $\tilde{P}_{\chi,\mathfrak{a}}$ of level $\tilde{U}$ on $\tilde{G}$.  When $\mathfrak{a}=\co_F$ we omit it from the notation.
Define an ideal $\mathfrak{e}$  of $\co_F$ by $\ord_v(\mathfrak{e})=\ord_v(\mathrm{N}(\tilde{\epsilon}_v))$ for all finite places $v$, so that
\begin{equation}\label{r bump}
\ord_v(\mathfrak{e})=\ord_v(\mathfrak{r})  +  
\left\{ \begin{array}{ll}
1 &\mathrm{if\ }v=w \\
0&\mathrm{otherwise.}
\end{array}\right.
\end{equation}

 \begin{Prop}\label{nearly good derivatives}
Suppose $\mathfrak{a}$ is prime to $\mathfrak{c}$.  There is a constant $\kappa$, independent of $\mathfrak{a}$, such that
$$
g(P^0_{\chi,\mathfrak{a}},P_\chi)_{U,w^\alg} = \kappa \cdot r_\chi(\mathfrak{a})+
\sum_{\gamma\in T(F)\backslash \tilde{G}(F)/T(F)} \langle \tilde{P}_{\chi,\mathfrak{a}}, \tilde{P}_{\chi} \rangle_{\tilde{U}}^\gamma   \cdot  m_\mathfrak{a}(\gamma)
$$
where 
$$
m_\mathfrak{a}(\gamma)= \frac{1}{2} \left\{\begin{array}{ll}
\ord_w(\xi\mathfrak{a})+1 & \mathrm{if\ }\xi\not=0 \mathrm{\ and\ } \ord_w(\xi\mathfrak{a})\mathrm{\ is\ odd\ and\ nonnegative} \\
\ord_w(\mathfrak{a}) & \mathrm{if\ }\xi =0 \mathrm{\ and\ } \ord_w(\mathfrak{a})\mathrm{\ is\ even\ and\ nonnegative} \\
0&\mathrm{otherwise.}
\end{array}\right.
$$
\end{Prop}

\begin{proof}
This is our analogue of \cite[Lemma 6.3.5]{zhang2}.  Decompose
$$
P^0_{\chi,\mathfrak{a}}=\sum_{k=0}^\infty \mathfrak{P}^0_k
\hspace{1cm}
P_{\chi,\mathfrak{a}}=\sum_{k=0}^\infty \mathfrak{P}_k
$$
where $\mathfrak{P}^0_k$ is the restriction of $P^0_{\chi,\mathfrak{a}}$ to points of $w$-conductor $k$, and similarly for $\mathfrak{P}_k$. By (\ref{make disjoint})
$$
\mathfrak{P}_k=\mathfrak{P}_k^0+ \left\{\begin{array}{ll}
r_\chi(\mathfrak{a})P_\chi &\mathrm{if\ }k=0 \\
0&\mathrm{otherwise}\end{array}\right.
$$
and Proposition \ref{Zhang's intersection I} gives
\begin{eqnarray*}
g(P^0_{\chi,\mathfrak{a}}, P_\chi)_{U,w^\alg}
  &=&
 \sum_{\gamma\in T(F)\backslash \tilde{G}(F)/T(F)} 
 \sum_{k=0}^\infty \langle \sigma_*\mathfrak{P}_k, \sigma_*P_\chi \rangle_{\tilde{U}}^\gamma \cdot M_k(\gamma)  \\
& & -  r_\chi(\mathfrak{a})\sum_{\gamma\in T(F)\backslash \tilde{G}(F)/T(F)}  \langle \sigma_*P_\chi, \sigma_* P_\chi\rangle_{\tilde{U}}^ \gamma \cdot M_0(\gamma).
\end{eqnarray*}

The next claim is that 
$\sigma_*\mathfrak{P}_k=c_k\tilde{P}_{\chi,\mathfrak{a}}$
where 
$$
c_k= \left\{ \begin{array}{ll} 
{[\co_{E,w}^\times: A_k^\times]  }& {\mathrm{ if\ } \ord_w(\mathfrak{a})-k \mathrm{\ is\ even\ and\ nonnegative}}  \\ 
0 &\mathrm{ otherwise.}
\end{array}\right.
$$
To prove this define
\begin{eqnarray*}
H_w^k(\mathfrak{a}) &=& \{h\in H_w(\mathfrak{a}) \mid h U_wh^{-1}\cap T(F_w)=A_k^\times\}  \\
H^k(\mathfrak{a})  &=&  \{h\in H(\mathfrak{a}) \mid h_w\in H^k_w(\mathfrak{a})\} \\
\tilde{H}(\mathfrak{a}) &=&\tilde{H}_w(\mathfrak{a}) \cdot  \prod_{v\not= w}\sigma_v( H_v(\mathfrak{a}))
\end{eqnarray*}
where $\tilde{H}_w(\mathfrak{a})=\{ h\in \tilde{R}_w\mid \tilde{\mathrm{N}}(h)\co_F=\mathfrak{a}_v  \}$. The CM-cycles in question are now given by
\begin{eqnarray*}
\mathfrak{P}_k (g) &=& \chi_0(\mathfrak{a}) \sum_{t\in  T(\A_f)/U_T }\chi(t) 
\mathbf{1}_{ H^k(\mathfrak{a}) } (t^{-1}g)  \\
\tilde{P}_{\chi,\mathfrak{a}} (g) &=& \chi_0(\mathfrak{a}) \sum_{t\in T(\A_f)/U_T }\chi(t) 
\mathbf{1}_{\tilde{H}(\mathfrak{a}) } (t^{-1}g) .
\end{eqnarray*}
 As in the proof of \cite[Lemma 6.3.5]{zhang2} there is a decomposition 
$$
G(F_w)=\bigsqcup_{k=0}^\infty T(F_w)h_k U_w
$$
where each $h_k\in R_w$ satisfies $\ord_w(\mathrm{N}(h_k))=k$ and $h_k U_wh_k^{-1}\cap T(F_w)=A_k^\times$.  Fixing a uniformizer $\varpi\in F_w^\times$ we therefore find
$$
H_w^k(\mathfrak{a}) =\left\{   \begin{array}{ll} \varpi^{\frac{\ord_w(\mathfrak{a})-k}{2}} \co_{E,w}^\times h_k U_w &\mathrm{if\ }\ord_w(\mathfrak{a})-k\mathrm{\ is\ even\ and\ nonnegative}  \\
\emptyset &\mathrm{otherwise.} \end{array}\right.
$$
From this it follows that $\#(H_w^k(\mathfrak{a})/U_w)=c_k$.  Write $H_w^k(\mathfrak{a})=\sqcup_{i=1}^{c_k} s_i U_w$.    For any $t\in T(\A_f)$ we have  $\sigma_w(t s_i)= t\tilde{H}_w(\mathfrak{a})$,  and hence  $\sigma_* \mathbf{1}_{ t H^k(\mathfrak{a}) }   =  c_k \cdot \mathbf{1}_{t \tilde{H}(\mathfrak{a}) }$
from which $\sigma_*\mathfrak{P}_k=c_k\tilde{P}_{\chi,\mathfrak{a}}$ follows immediately.

It follows from the above that
$$
 \sum_{k=0}^\infty \langle \sigma_*\mathfrak{P}_k, \sigma_*P_\chi \rangle_{\tilde{U}}^\gamma \cdot M_k(\gamma) =  \langle \tilde{P}_{\chi,\mathfrak{a}}, \tilde{P}_\chi \rangle_{\tilde{U}}^\gamma  \cdot
 \sum_{k=0}^\infty  c_k \cdot M_k(\gamma).
$$
Assume $\langle\tilde{P}_{\chi,\mathfrak{a}},\tilde{P}_\chi\rangle^\gamma_{\tilde{U}}\not=0$.   Suppose first that $\gamma$ is nondegenerate. In particular $O^\gamma_{\tilde{U}}(\tilde{P}_{\chi,\mathfrak{a},w})\not=0$ by (\ref{orbital decomp}), and so Proposition \ref{Prop:unramified link I} implies that $\ord_w(\eta\mathfrak{a})$ and $\ord_w(\xi\mathfrak{a})-1$ are both even and nonnegative.  If $\ord_w(\mathfrak{a})$ is odd then $\ord_w(\eta)$ is odd, and as $\eta+\xi=1$ we must have $\ord_w(\xi)=0$.  Thus
\begin{equation}\label{stupid multiplicity match}
\sum_{k=0}^\infty c_k\cdot M_k(\gamma)= \# \{ k\mid 1\le k\le \ord_w(\mathfrak{a}), k\mathrm{\ odd} \} = m_\mathfrak{a}(\gamma).
\end{equation}
If $\ord_w(\mathfrak{a})$ is even then 
$$
\sum_{k=0}^\infty c_k\cdot M_k(\gamma) = \frac{\ord_w(\xi)+1}{2} +  \#\{ k\mid 1\le k\le \ord_w(\mathfrak{a}), k\mathrm{\ even} \} =m_\mathfrak{a}(\gamma).
$$
Now suppose $\gamma$ is degenerate, so that $\tilde{P}_{\chi,\mathfrak{a}}(\gamma)\not=0$ by Lemma \ref{orbital integrals}.  If $\xi=0$ then we may assume $\gamma=1$ so that Lemma \ref{unr degenerate} implies $\ord_w(\mathfrak{a})$ is even.  Thus 
$$
\sum c_k\cdot M_k(\gamma)=\#\{ k\mid 1\le k\le \ord_w(\mathfrak{a}), k\mathrm{\ even}\} =m_\mathfrak{a}(\gamma).
$$
If $\xi=1$ then similarly $\ord_w(\mathfrak{ae}^{-1})=\ord_w(\mathfrak{a})-1$ is even and so again (\ref{stupid multiplicity match}) holds.  
\end{proof}

\begin{Cor}\label{good inert derivatives}
Suppose $\mathfrak{a}$ is prime to $\mathfrak{dr}$. Then 
$$
 2^{[F:\Q]+1}   \log |\varpi|_w \cdot  g (P_\chi, P^0_{\chi,\mathfrak{a}})_{U,w^\alg}  
=
  [\widehat{\co}^\times_E:  U_T] H_F\lambda_U^{-1}
 \cdot \mathrm{N}(\mathfrak{a}) \widehat{B}^w(\mathfrak{a};\Phi_{\mathfrak{r}})+ A(\mathfrak{a})
 $$
 where $A(\mathfrak{a})$ is a derivation of $\Pi_{\overline{\chi}}\otimes|\cdot|^{1/2}$ in the sense of \cite[Definition 3.5.3]{zhang2}.
 \end{Cor}

\begin{proof}
Fix a nondegenerate $\gamma\in \tilde{G}(F)$ and an $a\in\A^\times$ with $a\co_F=\mathfrak{a}$.
For any place $v$ of $F$, Lemma \ref{eta parametrization} and the definition of $\tilde{B}$ give
$$
\omega_v(-\eta\xi)
= \epsilon_v(1/2,\mathfrak{r}) \cdot
\left\{\begin{array}{ll}
-1 &\mathrm{if\ }v=w\\
1 &\mathrm{if\ }v\not=w.
\end{array}\right.
$$
Thus $\mathrm{Diff}_\mathfrak{r}(\eta,\xi)=\{w\}$, and conversely a pair $\eta,\xi\in F^\times$ with $\eta+\xi=1$ arises from  some choice of nondegenerate $\gamma\in \tilde{G}(F)$ if and only if $\mathrm{Diff}_\mathfrak{r}(\eta,\xi)=\{w\}$.
Comparing  Propositions \ref{derivative coefficients} and \ref{Prop:unramified link I}, and recalling (\ref{r bump}), we find
$$
B_w(a,\eta,\xi;\Theta'_{\mathfrak{r}})
=|a|_w \tau_w(\gamma)\cdot O_{\tilde{U}}^\gamma( \tilde{P}_{\chi,\mathfrak{a},w})
\cdot m_\mathfrak{a}(\gamma) \log|\varpi^2|_w.
$$
On the other hand for any finite place $v\not=w$ we have, using (\ref{r bump}) and  Corollaries \ref{unramified links} and \ref{ramified links},
$$
[\co_{E,v}^\times:\co_{F,v}^\times U_{T,v}] B_v(a,\eta,\xi;\Theta_{\mathfrak{r}})
=|a|_v \tau_v(\gamma)\cdot O_{\tilde{U}}^\gamma( \tilde{P}_{\chi,\mathfrak{a},v }).
$$
Using   (\ref{projection coefficient}), Lemma \ref{global tau}, and (\ref{orbital decomp}) we find
$$
[\widehat{\co}_E^\times:U_T] H_F \lambda_U^{-1} \cdot \mathrm{N}(\mathfrak{a})   \widehat{B}^w(\mathfrak{a},\Phi_\mathfrak{r})=
2^{[F:\Q] + 1} \log |\varpi|_w \sum  \langle \tilde{P}_\chi,\tilde{P}_{\chi,\mathfrak{a}}\rangle_{\tilde{U}}^\gamma \cdot m_\mathfrak{a}(\gamma) 
$$
where the sum is over all nondegenerate $\gamma\in T(F)\backslash \tilde{G}(F)/T(F)$.  If $\gamma$ is degenerate then
$ \langle  \tilde{P}_\chi , \tilde{P}_{\chi,\mathfrak{a}}\rangle_{\tilde{U}}^\gamma \cdot m_\mathfrak{a}(\gamma)$ is a derivation of $\Pi_{\overline{\chi}}\otimes|\cdot|^{1/2}$ (using Lemma \ref{orbital integrals} and Corollaries \ref{unr final degen} and \ref{ram final degen}).  Thus the claim follows from Proposition \ref{nearly good derivatives}.
\end{proof}


\subsection{Intersections at nonsplit primes dividing $\mathfrak{dr}$}
\label{ss:bad intersections}


Suppose that $w$ is a place of $F$ which is nonsplit in $E$ with $w\mid\mathfrak{dr}$ and fix a place $w^\alg$ of $F^\alg$ above $w$.  Again let $\tilde{B}$ be the quaternion algebra over $F$ obtained from $B$ by interchanging invariants at $w$ and $w_\infty$, so that  $\{\mathrm{places\ }v \mathrm{\ of\ }F \mid \tilde{B}_v\not\iso B_v \} = \{ w,w_\infty\}.$  Fix an embedding $E\map{}\tilde{B}$ and for each finite place $v\not= w$ let $\sigma_v$ and $\tilde{\epsilon}_v$ be as in \S \ref{ss:good intersections}.  Choose $\tilde{\epsilon}_w$ so that $\tilde{B}_w^-=E_w\tilde{\epsilon}_w$
and 
$$
\ord_w(\mathrm{N}(\tilde{\epsilon}_w))=\ord_w(\mathfrak{r}) + \left\{\begin{array}{ll}
1 &\mathrm{if\ }w\nmid\mathfrak{d} \\
0 & \mathrm{otherwise.}
\end{array}\right.
$$
Let $\mathfrak{a}$ be prime to $\mathfrak{dr}$.
As in \S \ref{ss:good intersections}, for any finite place $v$ we may repeat the constructions of \S \ref{ss:unramified local calculations} and \S \ref{ss:ramified local calculations} with $B$ replaced by $\tilde{B}$ and $\epsilon_v$ replaced by $\tilde{\epsilon}_v$, giving a compact open subgroup $\tilde{U}_v\subset \tilde{G}(\A_f)$ and a function  $\tilde{P}_{\chi,\mathfrak{a},v}$ on $\tilde{G}(F_v)/\tilde{U}_v$ for each $v$.  Taking the product over all finite $v$ gives a CM-cycle $\tilde{P}_{\chi,\mathfrak{a}}$ of level $\tilde{U}$.    

Define the $w$-\emph{special} CM points of level $U$, denoted $C_U^0$, to be the image of
$$
T(F_w)\times G(\A_f^w)\map{} C_U
$$
 where $\A_f^w=\{x\in \A_f\mid x_w=0\}$. By a $w$-special CM cycle we mean a CM cycle supported on $w$-special points.  Define $C^0_{\tilde{U}}$ similarly, and note that there are bijections
 $$
 C^0_U\iso T^0(F)\backslash G(\A_f^w)/U^w\iso T^0(F)\backslash \tilde{G}(\A_f^w)/\tilde{U}^w\iso C^0_{\tilde{U}}
 $$
 where $U^w=\prod_{v\not=w}U_v$ and similarly for $\tilde{U}^w$, and $T^0(F)$ is defined as
 $$
 T(F)\cap U_w=   T(F)\cap (1+\mathfrak{c}\co_{E,w})^\times  =T(F)\cap\tilde{U}_w.
  $$  
 Thus we may identify $w$-special cycles of level $U$ with $w$-special cycles of level $\tilde{U}$, and we denote this bijection by $P\mapsto\sigma_*P$.  As $\mathfrak{a}$ is prime to $\mathfrak{dr}$, $\ord_w(\mathfrak{a})=0$ and it follows from the construction that $P_{\chi,\mathfrak{a}}$ is $w$-special.  It is easy to see that  $\sigma_*P_{\chi,\mathfrak{a}}=\tilde{P}_{\chi,\mathfrak{a}}$ (as one only needs to check equality locally at $v\not=w$).

\begin{Prop}\label{bad zhang intersection}
Suppose $P$ and $Q$ are $w$-special CM cycles of level $U$ with disjoint support.  There is a locally constant function (independent of $P$ and $Q$) $K(x,y)$ on $\tilde{G}(F)\backslash \tilde{G}(\A_f)$ such that 
\begin{eqnarray*}
g(P,Q)_{U,w^\alg} &=& 
\sum_{\gamma\in T(F)\backslash \tilde{G}(F)/T(F)} \langle \sigma_*P, \sigma_*Q \rangle_{\tilde{U}}^\gamma \cdot M(\gamma) \\
& & +
\int_{[T(F) \backslash \tilde{G}(\A_f)]^2 } (\sigma_*P)(x) K(x,y)\overline{(\sigma_*Q)(y)}\ dx\ dy
\end{eqnarray*}
where 
$$
M(\gamma) =  \left\{ \begin{array}{ll} \frac{\ord_w(\xi) }{2} & \mathrm{if\ }\xi\not=0 \mathrm{\ and\ } \ord_w(\xi)>0 \\
0&\mathrm{otherwise.} \end{array}\right.
$$
\end{Prop}

\begin{proof}
See Lemmas 6.3.7 and 6.3.8 of \cite{zhang2}.
\end{proof}

 \begin{Prop}\label{first bad derivatives}
If $\mathfrak{a}$ is prime to $\mathfrak{dr}$ then
\begin{eqnarray*}
g(P^0_{\chi,\mathfrak{a}},P_\chi)_{U,w^\alg} &=& 
\kappa \cdot r_\chi(\mathfrak{a}) +
\sum_{\gamma\in T(F)\backslash \tilde{G}(F)/T(F)} \langle \tilde{P}_{\chi,\mathfrak{a}}, \tilde{P}_{\chi} \rangle_{\tilde{U}}^\gamma \cdot m(\gamma) \\
& & +
\int_{ [T(F) \backslash \tilde{G}(\A_f)]^2 } \tilde{P}_{\chi,\mathfrak{a}}(x) K(x,y)\overline{\tilde{P}_\chi(y)}\ dx\ dy
\end{eqnarray*}
where $K(x,y)$ is a locally constant function on $[\tilde{G}(F)\backslash \tilde{G}(\A_f)]^2$ and
$$
m(\gamma)= \frac{1}{2} \left\{\begin{array}{ll}
\ord_w(\xi\mathfrak{r}^{-1})+1 & \mathrm{if\ }\xi\not=0, \ord_w(\xi)\ge 0, \mathrm{\ and\ }w\mid \mathfrak{r}  \\
\ord_w(\xi\mathfrak{d}) &\mathrm{if\ }\xi\not=0, \ord_w(\xi)\ge 0, \mathrm{\ and\ }w\mid \mathfrak{d}  \\
0&\mathrm{otherwise.}
\end{array}\right.
$$
\end{Prop}

\begin{proof}
It follows from (\ref{make disjoint}) and Proposition \ref{bad zhang intersection} that the claim holds if one replaces  $m(\gamma)$ with $M(\gamma)$. Thus if we set $m'=m-M$  it suffices to show that
$$
\sum_{\gamma\in T(F)\backslash \tilde{G}(F)/T(F)} \langle \tilde{P}_{\chi,\mathfrak{a}}, \tilde{P}_{\chi} \rangle_{\tilde{U}}^\gamma \cdot m'(\gamma) 
=
\int_{T(F) \backslash \tilde{G}(\A_f) } \tilde{P}_{\chi,\mathfrak{a}}(x) k(x,y)\overline{\tilde{P}_\chi(y)}\ dx\ dy
$$ 
for $k$ some locally constant function on $\tilde{G}(F)\backslash \tilde{G}(\A_f)$.
Note that $m'$ is locally constant for the topology on $G(F)$ induced from $G(F_w)$ (i.e. $m$ and $M$ have the same singularity near $\xi=0$) and let $\tilde{U}'_w\subset\tilde{U}_w$ be small enough that $m'$ is a constant, $\mu$, on $\tilde{U}'_w$. Let $\tilde{U}'$ be the subgroup obtained by shrinking the $w$-component of $\tilde{U}$ from $\tilde{U}_w$ to $\tilde{U}_w'$.  The crucial point is that on the image of $\{1\}\times \tilde{G}(\A^w)\map{}C_{\tilde{U}'}$ we have
$$
k_{\tilde{U}'}^{m'}(x,y)= k_{\tilde{U}'}^\mu(x,y)
$$
where $k_{\tilde{U}'}^\mu$ is the kernel (\ref{kernel def}) constructed with constant multiplicity function $\mu$.  The $w$-special CM-cycles $\tilde{P}_{\chi,\mathfrak{a}} $ and $\tilde{P}_\chi$ are supported on the image of $T(F_w)\times \tilde{G}(\A^w)$  in $C_{\tilde{U}'}$, which equals the image of  $\{1\}\times \tilde{G}(\A^w)$ as $T(F_w) \subset T(F) \tilde{U}_w'  $.  Therefore the pairings (\ref{pairing}) satisfy  
$$
\langle \tilde{P}_{\chi,\mathfrak{a}},  \tilde{P}_\chi\rangle^{m'}_{\tilde{U}'}  
=
\langle \tilde{P}_{\chi,\mathfrak{a}},  \tilde{P}_\chi\rangle^{\mu}_{\tilde{U}'},
$$
and it follows  that $\langle \tilde{P}_{\chi,\mathfrak{a}},  \tilde{P}_\chi\rangle^{m'}_{\tilde{U}}  =\langle \tilde{P}_{\chi,\mathfrak{a}},  \tilde{P}_\chi\rangle^{\mu}_{\tilde{U}}$ (replacing $\tilde{U}'$ by $\tilde{U}$ changes each pairing by a constant depending on the normalizations of measures in \S \ref{ss:heights} but not on the multiplicity function).  As the multiplicity function $\mu$ is constant the kernel  $k^\mu_{\tilde{U}}$ is right $\tilde{G}(F)$-invariant, and we take $k=k_{\tilde{U}}^\mu$.
\end{proof}

\begin{Cor}\label{bad inert derivatives}
Define a function $\mathcal{P}_{\overline{\chi}}$ on $S_{\tilde{U}}=\tilde{G}(F)\backslash \tilde{G}(\A_f)/\tilde{U}$ by
$$
\mathcal{P}_{\overline{\chi}}(g)=\sum_{\gamma\in T(F)\backslash G(F)}  \overline {\tilde{P}_{\chi}(\gamma g) }.
$$
For any  $\mathfrak{a}$ prime to $\mathfrak{dr}$
\begin{eqnarray*}\lefteqn{
 2^{[F:\Q]+1}  |d|^{1/2}  \log |\varpi|_w \cdot  g(  P_\chi, P^0_{\chi,\mathfrak{a}} )_{U,w^\alg}   }\\
 &= & 
  [\widehat{\co}^\times_E:  U_T]  H_F\lambda_U^{-1}  \cdot  
\mathrm{N}(\mathfrak{a})  \widehat{B}^w(\mathfrak{a} ;\Phi_\mathfrak{r}  )   
+ A(\mathfrak{a}) +   \int_{\tilde{G}(F)\backslash \tilde{G}(\A_f) } (T_\mathfrak{a}\mathcal{P}_{\overline{\chi}} )(x) \cdot g(x) \ dx  
\end{eqnarray*}
where $A(\mathfrak{a})$ is a derivation of $\Pi_{\overline{\chi}}\otimes|\cdot|^{1/2}$, $g(x)$ is a locally constant function on $\tilde{G}(F)\backslash \tilde{G}(\A_f)$, and $T_\mathfrak{a}$ is the Hecke operator on $L^2(S_{\tilde{U}})$ defined in \S \ref{ss:holomorphic values}.
\end{Cor}

\begin{proof}
This is deduced from Proposition \ref{first bad derivatives} exactly as in the proof of Corollary \ref{good inert derivatives}, taking
$$
g(x)=\int_{T(F)\backslash \tilde{G}(\A_f)} \overline{K(x,y)}  \tilde{P}_\chi(y)   \ dy.
$$
\end{proof}


\subsection{Archimedean intersections}


Let $w$ be an  archimedean place of $F$ and choose a place $w^\alg$ of $F^\alg$ above $w$.  If $w=w_\infty$ is the archimedean place at which $B$ is  split then set $\tilde{B}=B$.  If $w\not=w_\infty$ then let $\tilde{B}$ be the quaternion algebra obtained from $B$ be interchanging invariants at $w$ and $w_\infty$ as in \S \ref{ss:good intersections}.  As in \S \ref{ss:good intersections} fix an embedding $E\map{}\tilde{B}$ and, for every finite place $v$ of $F$, choose $\sigma_v:B_v\iso \tilde{B}_v$ compatible with the embeddings of $E_v$ into $B_v$ and $\tilde{B}_v$. Define $\tilde{\epsilon}_v=\sigma_v(\epsilon_v)$, set $\tilde{U}_v=\sigma_v(U_v)$, and let $\sigma_*$ denote the induced isomorphism from CM cycles of level $U$ on $G$ to CM cycles of level $\tilde{U}$ on $\tilde{G}$.

For  $\gamma\in \tilde{G}(F)$ view $\xi\in F$ as a real number using the embedding $F\map{}\R$ determined by $w$ and define
\begin{equation*}
m_s(\gamma)=\left\{ \begin{array}{ll}
Q_s(1-2 \xi ) & \mathrm{if\ }\xi<0 \\
0&\mathrm{otherwise,}
\end{array}\right.
\end{equation*} 
where $Q_s$ is defined by \cite[(6.3.3)]{zhang2}, and  a function on $\tilde{G}(\A_f)\times \tilde{G}(\A_f)$ 
$$
k_{\tilde{U}}^s(x,y)= \sum_{\gamma\in \tilde{G}(F)/(Z(F)\cap \tilde{U}) } 
\mathbf{1}_{\tilde{U}}(x^{-1} \gamma y) m_s( \gamma) .
$$
We now recall the statement of \cite[Lemma 6.3.1]{zhang2}.  For any distinct points $P,Q\in C_U$ the sum defining $k_{\tilde{U}}^s(\sigma_*P,\sigma_*Q)$ is convergent for $\mathrm{Re}(s)>0$ and extends to a meromorphic function in a neighborhood of $s=0$ with a simple pole at $s=0$.  Thus for any CM-cycles $P$ and $Q$ of level $U$  the pairing $\langle \sigma_*P,\sigma_*Q\rangle_{\tilde{U}}^{m_s}$ of (\ref{pairing}) has meromorphic continuation with a pole of order at most $1$ at $s=0$, and moreover
$$
g(P,Q)_{U,w^\alg} = \mathrm{const}_{s\to 0} \langle \sigma_*P,\sigma_*Q\rangle_{\tilde{U}}^{m_s}.
$$
In particular, if $\mathfrak{a}$ is prime to $\mathfrak{dr}$ then
\begin{equation}\label{archimedean intersection}
g(P^0_{\chi,\mathfrak{a}},P_\chi)_{U,w^\alg} = 
\mathrm{const}_{s\to 0}
\sum_{\gamma\in T(F)\backslash \tilde{G}(F)/T(F)} \langle \tilde{P}^0_{\chi,\mathfrak{a}}, \tilde{P}_{\chi} \rangle_{\tilde{U}}^\gamma   \cdot  m_s(\gamma)
\end{equation}
where $\tilde{P}^0_{\chi,\mathfrak{a}}=\sigma_* P^0_{\chi,\mathfrak{a}}$ is the cycle defined by replacing $U$ by $\tilde{U}$ and $B$ by $\tilde{B}$ in the definition of $P_{\chi,\mathfrak{a}}^0$, and similarly for $\tilde{P}_\chi$.

\begin{Cor}\label{archimedean derivatives}
For any  $\mathfrak{a}$  prime to $\mathfrak{dr}$
$$
 -2^{  [F:\Q]+1 }   |d|^{1/2}   g(P_\chi,P^0_{\chi,\mathfrak{a}} )_{U,w^\alg}   
=
  [\widehat{\co}^\times_E:  U_T] H_F\lambda_U^{-1}\mathrm{N}(\mathfrak{a})\cdot \mathrm{const}_{s\to 0} \widehat{B}^w(s,\mathfrak{a};\Phi_\mathfrak{r}) 
 $$
 up to  a derivation of $\Pi_{\overline{\chi}}\otimes|\cdot|^{1/2}$.
\end{Cor}

\begin{proof}
Suppose $\mathrm{Re}(\sigma)>0$ and, for any $\gamma\in\tilde{G}(F)$, write $M_s(\gamma)=M_s(\xi_w)$ where the $M_\sigma$ on the right is the function on $\R$ defined in \S \ref{ss:derivative}.   Combining (\ref{holomorphic infinity}) with Corollaries \ref{unramified links} and \ref{ramified links}, and arguing as in the proof of Corollary \ref{good inert derivatives}, we find
\begin{eqnarray*}\lefteqn{
[\widehat{\co}_E^\times:U_T\widehat{\co}_F^\times] \mathrm{N}(\mathfrak{a})\widehat{B}^w(s,\mathfrak{a};\Phi_\mathfrak{r})  } \\
& &= (-2i)^{[F:\Q]} \omega_\infty(\delta) |d|^{1/2}\sum |\eta\xi|_\infty^{1/2} M_s(\gamma)\prod_{v\nmid\infty} \overline{\tau_v(\gamma)}  \cdot  \overline{O_v^\gamma(\tilde{P}_{\chi,\mathfrak{a},v})}
\end{eqnarray*}
where the sum is over all nondegenerate $\gamma\in T(F)\backslash \tilde{G}(F)/T(F)$.  By Lemma \ref{global tau} we have $$\prod_{v\nmid\infty}\tau_v(\gamma)=\omega_\infty(\delta)(-i)^{[F:\Q]}|\eta\xi|_\infty^{-1/2},$$ and combining this with (\ref{orbital decomp}) gives
$$
[\widehat{\co}_E^\times:U_T] H_F\lambda_U^{-1} \mathrm{N}(\mathfrak{a})\widehat{B}^w(s,\mathfrak{a};\Phi_\mathfrak{r})  
= 2^{[F:\Q]}  |d|^{1/2}\sum \langle \tilde{P}_\chi, \tilde{P}_{\chi,\mathfrak{a}} \rangle_{\tilde{U}}^\gamma \cdot M_s(\gamma)
$$
where the sum is again over all nondegenerate $\gamma$ as above.  By the argument in the proof of  \cite[Lemma 6.4.1]{zhang2} the constant term as $s\to 0$ is unchanged if we replace $M_s(\gamma)$ by $-2 m_s(\gamma)$.  Adding in the terms corresponding to the two degenerate choices of $\gamma$ add derivations of $\Pi_{\overline{\chi}}\otimes |\cdot|^{1/2}$, as in the proof of Corollary \ref{good inert derivatives}, and replacing $\tilde{P}_{\chi,\mathfrak{a}}$ by $\tilde{P}^0_{\chi,\mathfrak{a}}$ also adds a derivation of $\Pi_{\overline{\chi}}\otimes |\cdot|^{1/2}$, by (\ref{make disjoint}) with $P$ replaced by $\tilde{P}$.  Thus the claim follows from (\ref{archimedean intersection}).
\end{proof}


\subsection{The twisted Gross-Zagier theorem}


Let $\mathbb{T}$ denote the $\Z$-algebra generated by the Hecke operators $T_\mathfrak{a}$ and the nebentype operators $(\langle \mathfrak{a} \rangle \phi)(g)=\phi(ga)$, where $a\co_F=\mathfrak{a}$ and $a_v=1$ for $v\mid\infty$, acting on holomorphic automorphic forms on $\GL_2(\A)$ of parallel weight $2$ and level $K_1(\mathfrak{dr})$.  Let $\phi_\Pi$ denote the normalized newform in $\Pi$.  The $\C$-algebra $\mathbb{T}_\C=\mathbb{T}\otimes_\Z\C$ is semi-simple, and we let $\mathbb{T}_\Pi$ be the maximal summand of $\mathbb{T}_\C$ in which 
$$
T_\mathfrak{a} = \widehat{B}( \co_F ; T_\mathfrak{a}\phi_\Pi)
\hspace{1cm}
\langle\mathfrak{a}\rangle= \chi_0^{-1} (\mathfrak{a}) .
$$
Let $e_\Pi$ be the idempotent in $\mathbb{T}_\C$ satisfying $e_\Pi\mathbb{T}_\C=\mathbb{T}_\Pi$.
It follows from the Jacquet-Langlands correspondence and the Eichler-Shimura theory that there is a ring homomorphism $\mathbb{T}\map{}\mathrm{End}(J_U)$ taking $T_\mathfrak{a} \mapsto T_{\mathfrak{a}}^\Alb$ and $\langle\mathfrak{a}\rangle \mapsto \langle\mathfrak{a}\rangle^\Alb$, and so $\mathbb{T}_\C$ acts on $J_U(E_\chi)\otimes_\Z\C$.

\begin{Prop}\label{pre-GZ}
Abbreviating $P_{\chi,\Pi}=e_\Pi\cdot \mathrm{Hg}(P_\chi)$,
$$
\frac{ 2^{|S|} H_F [\widehat{\co}_E^\times:U_T]  }{ \lambda_U ||\phi_\Pi^\#||_{K_0(\mathfrak{dr})}^2 }  \widehat{B}(\co_F, \phi_\Pi^\#) L'(1/2,\Pi\times\Pi_\chi)  
 = 2^{[F:\Q]+1} |d|^{1/2} \langle P_{\chi,\Pi} ,  P_{\chi,\Pi}  \rangle^{\mathrm{NT}}_U.
$$
\end{Prop}

\begin{proof}
This follows easily from the formulae of the previous subsections, exactly as in \cite[\S 6.4]{zhang2}, "Conclusion of the Proof of Theorem 1.3.2".  We quickly sketch the argument.

Suppose $\mathfrak{a}$ is prime to $\mathfrak{dr}$.  Using the argument of \cite[Lemma 6.2.2]{zhang2}, up to sums of derivations of principal series and $\Pi_{\overline{\chi}}\otimes |\cdot|^{1/2}$ we have 
\begin{eqnarray*}
\langle T_{\mathfrak{a}}^\Alb \mathrm{Hg}(P_\chi) , \mathrm{Hg}(P_\chi) \rangle_U^{\mathrm{NT}} &=& 
\langle \mathrm{Hg}(P_\chi) , T_{\mathfrak{a}}^\Pic \mathrm{Hg}(P_\chi) \rangle_U^{\mathrm{NT}}  \\
&=&
\langle \mathrm{Hg}(P_\chi) , \mathrm{Hg}(P_{\chi,\mathfrak{a}}) \rangle_U^{\mathrm{NT}} \\
&=&
- \sum_w d_w \cdot  g(P_\chi,P_{\chi,\mathfrak{a}}^0)_{U,w^\alg}
\end{eqnarray*}
where the sum is over all places $w$ of $F$, and where for each $w$ we fix an extension $w^\alg$ to $F^\alg$. Exactly as in \cite[Lemma 6.3.4]{zhang2} the  nonarchimedean places $w$ which split in $E$ contribute derivations of principal series and $\Pi_{\overline{\chi}}\otimes |\cdot|^{1/2}$, and so we may omit such places in the above summation.  Combining Corollaries \ref{good inert derivatives}, \ref{bad inert derivatives}, and \ref{archimedean derivatives} with Proposition \ref{holomorphic coefficients} we find
$$
2^{[F:\Q]+1}|d|^{1/2} \langle T_{\mathfrak{a}}^\Alb \mathrm{Hg}(P_\chi) , \mathrm{Hg}(P_{\chi}) \rangle_U^{\mathrm{NT}}
=
[\widehat{\co}_E^\times:U_T] H_F\lambda_U^{-1}  \widehat{B}(\co_F;T_\mathfrak{a} \Phi_\mathfrak{r})
$$
up to a sum of derivations of principal series, derivations of $\Pi_{\overline{\chi}}\otimes |\cdot|^{1/2}$, and functions of the form 
\begin{equation}\label{nearby error}
  \int_{\tilde{G}(F)\backslash \tilde{G}(\A_f) } (T_\mathfrak{a}\mathcal{P}_{\overline{\chi}} )(x) \cdot g(x) \ dx  
\end{equation}
 for $w\mid\mathfrak{dr}$ as in Corollary \ref{bad inert derivatives}.  
 
 Let us consider (\ref{nearby error}) in more detail.  Fix $w\mid\mathfrak{dr}$ and let $\tilde{U}$, $\tilde{G}$, and so on be as in \S \ref{ss:bad intersections}.  Let $S_{\tilde{U}}=\tilde{G}(F)\backslash \tilde{G}(\A_f)/\tilde{U}$ as in \S \ref{ss:holomorphic values}.  It follows from the Jacquet-Langlands correspondence that the $\C$-algebra generated by the operators $T_\mathfrak{a}$ acting on $L^2(S_{\tilde{U}})$ is a quotient of $\mathbb{T}_\C$.   Thus it makes sense to form  $e_\Pi\cdot \mathcal{P}_{\overline{\chi}}\in L^2(S_{\tilde{U}})$, which is nothing more than the projection of $\mathcal{P}_{\overline{\chi}}$ to the automorphic representation $\tilde{\Pi}$ of $\tilde{G}(\A)$ whose Jacquet-Langlands lift is $\Pi$.
By construction the function $ e_\Pi\cdot \mathcal{P}_{\overline{\chi}}$ has character $\chi_w^{-1}$ under  right multiplication by $T(F_w)$.  On the other hand, if $\Pi'$ is the automorphic representation of $G(\A)$ whose Jacquet-Langlands lift is $\Pi$ then $\Pi'$ contains a nonzero vector on which $T(F_w)$ acts through $\chi_w^{-1}$ (as $\Pi'_w$ admits a toric newvector in the sense of \S \ref{ss:toric}).  Thus if $e_\Pi\mathcal{P}_{\overline{\chi}}\not=0$ we would have nonzero vectors in both $\tilde{\Pi}_w$ and $\Pi'_w$ on which $T(F_w)$ acts through $\chi_w^{-1}$.  This contradicts results of Saito, Tunnell, and  Waldspurger (as described in \cite[\S 10]{gross-representations} or \cite[Proposition 1.1]{gross-prasad}, and using \cite[Lemme 9(iii)]{waldspurger} to relate $T(E_w)$-invariants to $T(E_w)$-coinvariants), and so  $e_\Pi\mathcal{P}_{\overline{\chi}}=0$.  

We now deduce, using \cite[Proposition 4.5.1]{zhang1} for the vanishing of derivations of principal series and theta series, that
$$
2^{[F:\Q]+1}|d|^{1/2} \langle e_\Pi \mathrm{Hg}(P_\chi) , \mathrm{Hg}(P_{\chi}) \rangle_U^{\mathrm{NT}}
=
[\widehat{\co}_E^\times:U_T] H_F\lambda_U^{-1}  \widehat{B}(\co_F; e_\Pi \Phi_\mathfrak{r}).
$$
As $e_\Pi\Phi_\mathfrak{r}$ is the projection of $\Phi_\mathfrak{r}$ to $\Pi$, the proof now follows from 
$$
\widehat{B}(\co_F; e_\Pi \Phi_\mathfrak{r}) \cdot ||\phi_\Pi^\#|| _{K_0(\mathfrak{dr})}^2=2^{|S|} \widehat{B}(\co_F;\phi_\Pi^\#) L'(1/2,\Pi\times\Pi_\chi)
$$
as in the proof of Proposition \ref{first holomorphic values}. 
\end{proof}

As above there is a ring homomorphism $\mathbb{T}\map{}\mathrm{End}(J_V)$ taking $T_\mathfrak{a} \mapsto T_{\mathfrak{a}}^\Alb$ and $\langle\mathfrak{a}\rangle \mapsto \langle\mathfrak{a}\rangle^\Alb$, and so $\mathbb{T}_\C$ acts on $J_V(E_\chi)\otimes_\Z\C$.

\begin{Thm}\label{thm:twisted gz}
Abbreviate $Q_{\chi,\Pi} =e_\Pi\mathrm{Hg}(Q_\chi) \in J_V(E_\chi)\otimes_\Z\C$.
$$
\frac{  L'(1/2,\Pi\times\Pi_\chi) }{  ||\phi_\Pi||_{K_0(\mathfrak{n})}^2  }=
 \frac{2^{[F:\Q]+1}  }{ H_{F,\mathfrak{s}}  \sqrt{\mathrm{N}_{F/\Q}(\mathfrak{dc}^2)}} \langle Q_{\chi,\Pi},  Q_{\chi,\Pi}  \rangle^{\mathrm{NT}}_V.
$$
\end{Thm}

\begin{proof}
Recall the constants $\mathbf{a}_\Pi$, $\mathbf{b}_\Pi$, and $\mathbf{c}_\Pi$ of  \S \ref{level comparison}. The argument of \cite[\S 17]{zhang3} gives the first equality of
$$
\langle P_{\chi,\Pi},P_{\chi,\Pi} \rangle^{\mathrm{NT}}_{U}\cdot \mathbf{c}_\Pi=
\langle \pi^* Q_{\chi,\Pi},\pi^*Q_{\chi,\Pi} \rangle^{\mathrm{NT}}_{U} =
\deg(\pi)\cdot 
\langle Q_{\chi,\Pi},  Q_{\chi,\Pi} \rangle^{\mathrm{NT}}_{V} 
$$
where $\pi^*:J_V\map{}J_U$ is the morphism induced by the natural projection $\pi: X_U\map{}X_V$ of degree $[F^\times V:F^\times U] = [V:U] \lambda_V\lambda_U^{-1}$.  It therefore follows from Proposition \ref{pre-GZ} that
$$ 
\mathbf{a}_\Pi \mathbf{c}_\Pi  \frac{ 2^{|S|} H_F [\widehat{\co}_E^\times:U_T]  }{   [V:U] \lambda_V }   \frac{ L'(1/2,\Pi\times\Pi_\chi)  }{ ||\phi_\Pi^\#||_{K_0(\mathfrak{n})}^2 }
 =  \frac{ \mathbf{b}_\Pi 2^{[F:\Q]+1} }{\sqrt{ \mathrm{N}_{F/\Q}(\mathfrak{d}) }} \langle Q_{\chi,\Pi} ,  Q_{\chi,\Pi}  \rangle^{\mathrm{NT}}_V
$$
and so the theorem follows from the equality of rational functions $\kappa \prod \mathbf{a}_v\mathbf{c}_v=\prod \mathbf{b}_v$  proved in  \S \ref{level comparison}.
\end{proof}

\bibliographystyle{plain}

\begin{thebibliography}{10}

\bibitem{BD-L-functions}
M.~Bertolini and H.~Darmon.
\newblock The {$p$}-adic {$L$}-functions of modular elliptic curves.
\newblock In {\em Mathematics unlimited---2001 and beyond}, pages 109--170.
  Springer, Berlin, 2001.

\bibitem{BD-iwasawa}
M.~Bertolini and H.~Darmon.
\newblock Iwasawa's main conjecture for elliptic curves over anticyclotomic
  {$\mathbb Z\sb p$}-extensions.
\newblock {\em Ann. of Math. (2)}, 162(1):1--64, 2005.

\bibitem{carayol}
H.~Carayol.
\newblock Sur la mauvaise r\'eduction des courbes de {S}himura.
\newblock {\em Compositio Math.}, 59(2):151--230, 1986.

\bibitem{conrad}
B.~Conrad.
\newblock Gross-{Z}agier revisited.
\newblock In {\em Heegner points and Rankin $L$-series}, volume~49 of {\em
  Math. Sci. Res. Inst. Publ.}, pages 67--163. Cambridge Univ. Press,
  Cambridge, 2004.
\newblock With an appendix by W. R. Mann.

\bibitem{cornut-vatsal2}
C.~Cornut and V.~Vatsal.
\newblock C{M} points and quaternion algebras.
\newblock {\em Doc. Math.}, 10:263--309 (electronic), 2005.

\bibitem{cornut-vatsal}
C.~Cornut and V.~Vatsal.
\newblock Nontriviality of rankin-selberg {L}-functions and {CM} points.
\newblock Preprint at \texttt{www.institut.math.jussieu.fr/$\sim$cornut}.

\bibitem{gelbart}
S.~Gelbart.
\newblock {\em Lectures on the {A}rthur-{S}elberg trace formula}, volume~9 of
  {\em University Lecture Series}.
\newblock American Mathematical Society, Providence, RI, 1996.

\bibitem{gelbart-jacquet}
S.~Gelbart and H.~Jacquet.
\newblock Forms of {${\rm GL}(2)$} from the analytic point of view.
\newblock In {\em Automorphic forms, representations and $L$-functions (Proc.
  Sympos. Pure Math., Oregon State Univ., Corvallis, Ore., 1977), Part 1},
  Proc. Sympos. Pure Math., XXXIII, pages 213--251. Amer. Math. Soc.,
  Providence, R.I., 1979.

\bibitem{gillet-soule1}
H.~Gillet and C.~Soul{\'e}.
\newblock Arithmetic intersection theory.
\newblock {\em Inst. Hautes \'Etudes Sci. Publ. Math.}, (72):93--174 (1991),
  1990.

\bibitem{gross-values}
B.~Gross.
\newblock Heights and the special values of {$L$}-series.
\newblock In {\em Number theory (Montreal, Que., 1985)}, volume~7 of {\em CMS
  Conf. Proc.}, pages 115--187. Amer. Math. Soc., Providence, RI, 1987.

\bibitem{gross-modular}
B.~Gross.
\newblock Local orders, root numbers, and modular curves.
\newblock {\em Amer. J. Math.}, 110(6):1153--1182, 1988.

\bibitem{gross-representations}
B.~Gross.
\newblock Heegner points and representation theory.
\newblock In {\em Heegner points and Rankin $L$-series}, volume~49 of {\em
  Math. Sci. Res. Inst. Publ.}, pages 37--65. Cambridge Univ. Press, Cambridge,
  2004.

\bibitem{gross-prasad}
B.~Gross and D.~Prasad.
\newblock Test vectors for linear forms.
\newblock {\em Math. Ann.}, 291(2):343--355, 1991.

\bibitem{gross-zagier}
B.~Gross and D.~Zagier.
\newblock Heegner points and derivatives of {$L$}-series.
\newblock {\em Invent. Math.}, 84(2):225--320, 1986.

\bibitem{howard-derivatives}
B.~Howard.
\newblock Central derivatives of $L$-functions in {H}ida families.
\newblock {\em Math. Ann.}, 399(4):803--818, 2007.

\bibitem{howard-hida}
B.~Howard.
\newblock Variation of {H}eegner points in {H}ida families.
\newblock {\em Invent. Math.}, 167(1):91--129, 2007.
 

\bibitem{jacquet-langlands}
H.~Jacquet and R.~P. Langlands.
\newblock {\em Automorphic forms on {${\rm GL}(2)$}}.
\newblock Springer-Verlag, Berlin, 1970.
\newblock Lecture Notes in Mathematics, Vol. 114.

\bibitem{jordan-livne}
B.~Jordan and R.~Livn{\'e}.
\newblock Integral {H}odge theory and congruences between modular forms.
\newblock {\em Duke Math. J.}, 80(2):419--484, 1995.

\bibitem{kudla-tate}
S.~Kudla.
\newblock Tate's thesis.
\newblock In {\em An Introduction to the Langlands Program (Jerusalem, 2001)},
  pages 109--131. Birkh\"auser Boston, Boston, MA, 2003.

\bibitem{lang-arakelov}
S.~Lang.
\newblock {\em Introduction to {A}rakelov theory}.
\newblock Springer-Verlag, New York, 1988.

\bibitem{lubotzky}
A.~Lubotzky.
\newblock {\em Discrete groups, expanding graphs and invariant measures},
  volume 125 of {\em Progress in Mathematics}.
\newblock Birkh\"auser Verlag, Basel, 1994.
\newblock With an appendix by Jonathan D. Rogawski.

\bibitem{MW}
B.~Mazur and A.~Wiles.
\newblock Class fields of abelian extensions of {$\mathbb{Q}$}.
\newblock {\em Invent. Math.}, 76(2):179--330, 1984.

\bibitem{milne1}
J.~Milne.
\newblock Canonical models of (mixed) {S}himura varieties and automorphic
  vector bundles.
\newblock In {\em Automorphic forms, Shimura varieties, and $L$-functions,
  Vol.\ I (Ann Arbor, MI, 1988)}, volume~10 of {\em Perspect. Math.}, pages
  283--414. Academic Press, Boston, MA, 1990.

\bibitem{milne2}
J.~Milne.
\newblock Introduction to {S}himura varieties.
\newblock In {\em Harmonic analysis, the trace formula, and Shimura varieties},
  volume~4 of {\em Clay Math. Proc.}, pages 265--378. Amer. Math. Soc.,
  Providence, RI, 2005.

\bibitem{miyake}
T.~Miyake.
\newblock {\em Modular Forms}.
\newblock Springer-Verlag, Berlin, 1989.
\newblock Translated from the Japanese by Yoshitaka Maeda.

\bibitem{nek-euler}
Nekov\'{a}\v{r}.
\newblock The euler system method for {CM} points on {S}himura curves.
\newblock Preprint available at \texttt{www.math.jussieu.fr/$\sim$nekovar}.

\bibitem{nek}
J.~Nekov\'{a}\v{r}.
\newblock Selmer Complexes.
\newblock {\em Ast\'{e}risque} 310, 2006.


\bibitem{shimura}
G.~Shimura.
\newblock {\em Introduction to the arithmetic theory of automorphic functions}.
\newblock Publications of the Mathematical Society of Japan, No. 11. Iwanami
  Shoten, Publishers, Tokyo, 1971.

\bibitem{gillet-soule2}
C.~Soul{\'e}.
\newblock {\em Lectures on {A}rakelov geometry}, volume~33 of {\em Cambridge
  Studies in Advanced Mathematics}.
\newblock Cambridge University Press, Cambridge, 1992.
\newblock With the collaboration of D. Abramovich, J.-F.\ Burnol and J. Kramer.

\bibitem{vatsal-values}
V.~Vatsal.
\newblock Special values of anticyclotomic {$L$}-functions.
\newblock {\em Duke Math. J.}, 116(2):219--261, 2003.

\bibitem{vatsal}
V.~Vatsal.
\newblock Special value formulae for {R}ankin {$L$}-functions.
\newblock In {\em Heegner points and Rankin $L$-series}, volume~49 of {\em
  Math. Sci. Res. Inst. Publ.}, pages 165--190. Cambridge Univ. Press,
  Cambridge, 2004.

\bibitem{waldspurger}
J.-L. Waldspurger.
\newblock Quelques propri\'et\'es arithm\'etiques de certaines formes
  automorphes sur {${\rm GL}(2)$}.
\newblock {\em Compositio Math.}, 54(2):121--171, 1985.

\bibitem{weil}
A.~Weil.
\newblock {\em Basic number theory}.
\newblock Classics in Mathematics. Springer-Verlag, Berlin, 1995.
\newblock Reprint of the second (1973) edition.

\bibitem{zhang2}
S.-W. Zhang.
\newblock Gross-{Z}agier formula for {${\rm GL}\sb 2$}.
\newblock {\em Asian J. Math.}, 5(2):183--290, 2001.

\bibitem{zhang1}
S.-W. Zhang.
\newblock Heights of {H}eegner points on {S}himura curves.
\newblock {\em Ann. of Math. (2)}, 153(1):27--147, 2001.

\bibitem{zhang3}
S.-W. Zhang.
\newblock Gross-{Z}agier formula for {$\rm GL(2)$} {II}.
\newblock In {\em Heegner points and Rankin $L$-series}, volume~49 of {\em
  Math. Sci. Res. Inst. Publ.}, pages 191--214. Cambridge Univ. Press,
  Cambridge, 2004.

\end{thebibliography}

\end{document}